\documentclass[12pt]{article}
\usepackage{enumerate}
\usepackage{graphicx}
\usepackage{natbib}
\usepackage{multirow}
\usepackage{amsmath,natbib}
\usepackage{mathrsfs}
\usepackage{amsfonts}
\usepackage{amssymb}
\usepackage{amsmath}
\usepackage[]{geometry}
\pdfoutput=1

\newcommand{\function}[2]{#1\left(#2\right)}
\newcommand{\p}[1]{\mathbf{P}\left(#1\right)}

\newcommand{\e}[1]{\mathbf{E}\left(#1\right)}

\newenvironment{proof}{\textbf{Proof: }}{$\hfill \square \\$}
\newcommand{\emp}[1]{\mathbb{P}_n\left(#1\right)}
\newcommand{\cp}[1]{\mathbf{P_\mathfrak{X}}\left(#1\right)}
\newcommand{\ce}[1]{\mathbf{E_\mathfrak{X}}\left(#1\right)}
\newtheorem{lemma}{Lemma}[section]
\newtheorem{prop}{Proposition}[section]

\newtheorem{cor}{Corollary}[section]
\newcommand{\ind}[1]{\mathbf{1}_{#1}}
\newcommand{\cas}{\stackrel{a.s.}{\longrightarrow}}
\newcommand{\cip}{\stackrel{\mathbf{P}}{\longrightarrow}}

\newcommand{\ccipas}{\stackrel{\mathbf{P}_\mathfrak{X}}{\substack{\longrightarrow \\ a.s.}}}
\newcommand{\ccipip}{\stackrel{\mathbf{P}_\mathfrak{X}}{\substack{\longrightarrow \\ \mathbf{P}}}}
\newcommand{\argmax}{\operatornamewithlimits{\textrm{argmax}}}

\newcommand{\lsup}{\operatornamewithlimits{\overline{\lim}}}
\newcommand{\linf}{\operatornamewithlimits{\underline{\lim}}}
\newcommand{\sargmax}{\operatornamewithlimits{\textrm{sargmax}}}
\newcommand{\largmax}{\operatornamewithlimits{\textrm{largmax}}}
\newcommand{\sargmin}{\operatornamewithlimits{\textrm{sargmin}}}

\usepackage[colorlinks,citecolor=blue,urlcolor=blue,linkcolor=red]{hyperref}

\begin{document}


\title{\textsc{Change-point in stochastic design regression and the bootstrap}}
\author{\begin{tabular}{c}
Emilio Seijo and Bodhisattva Sen\\
Columbia University
\end{tabular}}
\date{}
\maketitle
\fontsize{12}{18pt}
\selectfont

\begin{abstract}
In this paper we study the consistency of different bootstrap procedures for constructing confidence intervals (CIs) for the unique jump discontinuity (change-point) in an otherwise smooth regression function in a stochastic design setting. This problem exhibits nonstandard asymptotics and we argue that the standard bootstrap procedures in regression fail to provide valid confidence intervals for the change-point. We propose a version of smoothed bootstrap, illustrate its remarkable finite sample performance in our simulation study, and prove the consistency of the procedure. The $m$ out of $n$ bootstrap procedure is also considered and shown to be consistent. We also provide sufficient conditions for any bootstrap procedure to be consistent in this scenario.
\end{abstract}

\section{Introduction}\label{intro}
Change-point models may arise when a stochastic system is subject to sudden external influences and are encountered in almost every field of science. In the simplest form the model considers a random vector $X=(Y,Z)$ satisfying the following relation:
\begin{equation}\label{ec1}
Y = \alpha_0 \mathbf{1}_{Z\leq \zeta_0} + \beta_0 \mathbf{1}_{Z> \zeta_0} + \epsilon,
\end{equation}
where $Z$ is a continuous random variable, $\alpha_0 \neq \beta_0 \in \mathbb{R}$, $\zeta_0 \in [a,b] \subset \mathbb{R}$ and $\epsilon$ is a continuous random variable, independent of $Z$ with zero expectation and finite variance $\sigma^2 > 0$. The parameter of interest is $\zeta_0$, the change-point.

Despite its simplicity, model (\ref{ec1}) captures the inherent ``non-standard'' nature of the problem: The least squares estimator of the change-point $\zeta_0$ converges at a rate of $n^{-1}$ to a minimizer of a two-sided, compound Poisson process that depends crucially on the entire error distribution, the marginal density of $Z$, among other nuisance parameters; see \cite{pons}, \cite{koss} (Section 14.5.1, pages 271--277) or \cite{koul}. Therefore, it is not practical to use this limiting distribution to build CIs for $\zeta_0$. Bootstrap methods bypass the estimation of nuisance parameters and are generally reliable in $\sqrt{n}$-convergence problems. In this paper we investigate the performance (both theoretically and through simulation) of different bootstrap schemes in building CIs for $\zeta_0$. We hope that the analysis of the bootstrap procedures employed in this paper will help illustrate the issues that arise when the bootstrap is applied in such non-standard problems.

The problem of estimating a jump-discontinuity (change-point) in an otherwise smooth curve has been under study for at least the last forty years. More recently, it has been extensively studied in the nonparametric regression and survival analysis literature; see for instance \cite{ghkbs}, \cite{dw}, \cite{pons}, \cite{koso}, \cite{lmm} and the references therein. Bootstrap techniques have also been applied in many instances in change point models. \cite{dum} proposed asymptotically valid confidence regions for the change-point by inverting bootstrap tests in a one-sample problem. \cite{huki} considered bootstrap CIs for the change-point of the mean in a time series context. \cite{koso} use a form of parametric bootstrap to estimate the distribution of the estimated change-point in a stochastic design regression model that arises in survival analysis. \cite{ghk}, in a slightly different setting, suggested a bootstrap procedure for model (\ref{ec1}), but did not give a complete proof of its validity.

Our work goes beyond those cited above as follows: We present strong theoretical and empirical evidence to suggest the {\it inconsistency} of the two most natural bootstrap procedures in a regression setup -- the usual nonparametric bootstrap (i.e., sampling from the empirical cumulative distribution function (ECDF) of $(Y,Z)$, often also called as bootstrapping ``pairs'') and the ``residual'' bootstrap. The bootstrap estimators built by both of these methods are the smallest maximizers of certain stochastic processes. We show that these processes do not have any weak limit in probability. This fact strongly suggests not only inconsistency but also the absence of {\it any} weak limit for the bootstrap estimators. In addition, we prove that independent sampling from a smooth approximation to the marginal of $Z$ and the centered ECDF of the residuals, and the $m$ out of $n$ bootstrap from the ECDF of $(Y,Z)$ yield asymptotically valid CIs for $\zeta_0$. The finite sample performance of the different bootstrap methods shows the superiority of the proposed smoothed bootstrap procedure. We also develop a series of convergence results which generalize those obtained in \cite{koss} to triangular arrays of random vectors and can be used to validate the consistency of {\it any} bootstrap scheme in this setup. Moreover, in the process of achieving this we develop convergence results for stochastic processes with a three-dimensional parameter which are continuous on the first two arguments and c\'adl\'ag on the third. In particular, we prove a version of the argmax continuous mapping theorem for these processes which may be of independent interest (see Section \ref{app2}).

Although we develop our results in the setting of (\ref{ec1}), our conclusions have broader implications (as discussed in Section \ref{discussion}). They extend immediately to regression functions with parametrically specified models on either side of the change-point. The \textit{smoothed bootstrap} procedure can also be modified to work in more general nonparametric settings. \cite{ghkbs} consider jump-point estimation in the more general setup of non-parametric regression and develop two-stage procedures to build CI for the change-point. In the second stage of their procedure, they localize to a neighborhood of the change-point and reduce the problem to exactly that of (\ref{ec1}). \cite{lmm} consider a two-stage adaptive sampling procedure to estimate the jump discontinuity. The second stage of their method relies on an approximate CI for the change-point, and the bootstrap methods developed in this paper can be immediately used in their context.

The paper is organized in the following manner: In Section \ref{BootsSchemes} we describe the problem in greater detail, introduce the bootstrap schemes and describe the appropriate notion of consistency. In Section \ref{pgs}, we prove a series of convergence results that generalize those obtained in \cite{koss}. These results will constitute the general framework under which the bootstrap schemes will be analyzed. In Section \ref{incons} we study the inconsistency of the standard bootstrap methods, including the ECDF and residual bootstraps. In Section \ref{cons} we propose two bootstrap procedures and show their consistency. We compare the finite sample performance of the different bootstrap methods through a simulation study in Section \ref{simula}. Finally, in Section \ref{discussion} we discuss the consequences of our analysis in more general change-point regression models. Additionally, we include an Appendix with the proofs and some necessary lemmas and results.

\section{The problem and the bootstrap schemes}\label{BootsSchemes}
Assume that we are given an i.i.d. sequence of random vectors $\left\{X_n=(Y_n,Z_n)\right\}_{n=1}^\infty$ defined on a probability space $\left(\Omega,\mathcal{A},\mathbf{P}\right)$ having a common distribution $\mathbb{P}$ satisfying (\ref{ec1}) for some parameter $\theta_0 :=(\alpha_0, \beta_0, \zeta_0) \in \Theta := \mathbb{R}^2\cup [a,b]$. This is a semi-parametric model with an Euclidean parameter $\theta_0$ and two infinite-dimensional parameters -- the distributions of $Z$ and $\epsilon$. We are interested in estimating $\zeta_0$, the change-point. For technical reasons, we will also assume that $\mathbb{P}(|\epsilon|^3)<\infty$.  Here, and in the remaining of the paper, we take the convention that for any probability distribution $\mu$, we will denote the expectation operator by $\mu(\cdot)$. In addition, we suppose that $Z$ has a uniformly bounded, strictly positive density $f$ (with respect to the Lebesgue measure) on $[a,b]$ such that $\inf_{|z - \zeta_0| \le \eta} f(z) > \kappa > 0$ for some $\eta >0$ and that $\mathbb{P}(Z<a)\land \mathbb{P}(Z>b) > 0$. For $\theta = (\alpha,\beta,\zeta) \in \Theta$, $x=(y,z) \in \mathbb{R}^2$ write
\begin{equation}\label{eq:m_theta}
\function{m_\theta}{x} := -\left(y - \alpha \mathbf{1}_{z\leq \zeta} - \beta \mathbf{1}_{z> \zeta}\right)^2,
\end{equation}
$\mathbb{P}_n$ for the empirical measure defined by $X_1,\ldots,X_n$,
\begin{equation}\label{eq:m_thetabis}\function{M_n}{\theta} := \emp{m_\theta} = - \frac{1}{n} \sum_{i=1}^n \left( Y_i - \alpha \mathbf{1}_{Z_i \leq \zeta} + \beta \mathbf{1}_{Z_i > \zeta} \right)^2,
\end{equation}
and $\function{M}{\theta} := \function{\mathbb{P}}{m_\theta}$. The function $M_n$ is strictly concave in its first two coordinates but c\`{a}dl\`{a}g (right continuous with left limits) in the third; in fact, piecewise constant and with $n$ jumps (w.p. 1). Thus, $M_n$ has unique maximizing values of $\alpha$ and $\beta$, but an entire interval of maximizers for $\zeta$. For this reason, we define the {\it least squares estimator} of $\theta_0$ to be the maximizer of $M_n$ over $\Theta$ with the smallest $\zeta$, and denote it by
\begin{eqnarray}
\hat{\theta}_n := (\hat{\alpha}_n,\hat{\beta}_n,\hat{\zeta}_n) = \sargmax_{\theta \in \Theta} \left\{ M_n(\theta) \right\}, \nonumber
\end{eqnarray}
where $\sargmax$ stands for the {\it smallest argmax}. At this point we would like to clarify what we mean by a maximizer: if $W$ is a c\`adl\`ag process on an interval $I$, a point $x\in I$ is said to be a maximizer if $\function{W}{x}\lor \function{W}{x^-}=\sup\left\{W(s) : s\in I\right\}$. In the context of our problem, $(\alpha,\beta,\zeta)$ is a maximizer of $M_n$ if $\function{M_n}{\alpha,\beta,\zeta} \lor \function{M_n}{\alpha,\beta,\zeta^-}=\sup\left\{M_n(\theta) : (\theta)\in \Theta\right\}$.

The asymptotic properties of this least squares estimator are well known. It is shown in \cite{koss}, pages 271--277, that $\sqrt{n}(\hat{\alpha}_n-\alpha_0)=\function{O_{\mathbf{P}}}{1}$, $\sqrt{n}(\hat{\beta}_n-\beta_0)=\function{O_{\mathbf{P}}}{1}$ and $n(\hat{\zeta}_n-\zeta_0)=\function{O_{\mathbf{P}}}{1}$. It is also shown that the asymptotic distribution of $n(\hat{\zeta}_n-\zeta_0)$ is that of the smallest argmax of a two-sided compound Poisson process. However, the limiting process depends on the distribution of $\epsilon$ and the value of the density of $Z$ at $\zeta_0$. Thus, there is no straightforward way to build CIs for $\zeta_0$ using this limiting distribution. In this connection we investigate the performance of bootstrap procedures for constructing CIs for $\zeta_0$.

\subsection{Bootstrap}\label{Boots}
We start with a brief review of the bootstrap. Given a sample ${\mathbf W}_n=\{W_1, W_2, \ldots,$ $ W_n\}\stackrel{\rm iid}{\sim} L$ from an unknown distribution $L$, suppose that the distribution function $H_n$ of some random variable $R_n \equiv R_n(\mathbf{W}_n, L)$ is of interest; $R_n$ is usually called a {\it root} and it can in general be any measurable function of the data and the distribution $L$. The bootstrap method can be broken into three simple steps:
\begin{itemize}
	\item[(i)] Construct an estimator $\hat{L}_n$ of $L$ from ${\mathbf W}_n$.
	
	\item[(ii)] Generate  ${\mathbf W}_n^{*} =\{W_1^{*},\ldots, W_{m_n}^{*}\} \stackrel{\rm iid} {\sim} \hat{L}_n$ given ${\mathbf W}_n$.
	
	\item[(iii)] Estimate $H_n$ by $\hat H_{n}$, the conditional CDF of $R_n({\mathbf W}_n^{*},\hat{L}_n)$
    given ${\mathbf W}_n$.
\end{itemize}
Let $d$ denote the Prokhorov metric or any other metric metrizing weak convergence of probability measures. We say that $\hat H_{n}$ is {\it weakly consistent} if $d(H_n, \hat H_n)\stackrel{P}{\rightarrow} 0$; if $H_{n}$ has a weak limit $H$, this is equivalent to $\hat H_{n}$ converging weakly to $H$ in probability. Similarly, $\hat H_{n}$ is {\it strongly consistent} if $d(H_n, \hat H_n)\stackrel{a.s.}{\rightarrow} 0$.

The choice of $\hat{L}_n$ mostly considered in the literature is the ECDF. Intuitively, an $\hat L_n$ that mimics the essential properties (e.g., smoothness) of the underlying distribution $L$ can be expected to perform well. Despite being a good estimator in most situations, the ECDF can fail to capture some properties of L that may be crucial for the problem under consideration. This is especially true for nonstandard problems. In Section \ref{incons} we illustrate this phenomenon (the inconsistency of the ECDF bootstrap) when $n(\hat{\zeta}_n - \zeta_0)$ is the random variable (root) of interest.

We denote by $\mathfrak{X}=\function{\sigma}{\left(X_n\right)_{n=1}^\infty}$ the $\sigma$-algebra generated by the sequence $\left(X_n\right)_{n=1}^\infty$ and write $\cp{\cdot}=\p{\cdot\left|\mathfrak{X}\right.}$ and $\ce{\cdot}=\e{\cdot\left|\mathfrak{X}\right.}$. We approximate the CDF of $\Delta_n = n(\hat{\zeta}_n - \zeta_0)$ by $\cp{\Delta_n^* \le x}$, the conditional distribution function of $\Delta_n^* = m_n(\zeta_n^* - \hat{\zeta}_n)$ and use this to build a CI for $\zeta_0$, where $\zeta_n^*$ is the least squares estimator of $\zeta_0$ obtained from the bootstrap sample. In the following we introduce four bootstrap schemes that arise naturally in this problem and investigate their consistency properties in Sections \ref{incons} and \ref{cons}. \medskip

\noindent {\bf Scheme 1 (ECDF bootstrap)}: Draw a bootstrap sample $(Y_{n,1}^*,Z_{n,1}^*), \ldots,$ $(Y_{n,n}^*,Z_{n,n}^*)$ from the ECDF of $(Y_1,Z_1),\ldots,(Y_n,Z_n)$; probably the most widely used bootstrap scheme. \medskip

\noindent {\bf Scheme 2 (Bootstrapping residuals)}: This is another widely used bootstrap procedure in regression models. We first obtain the residuals $$\hat{\epsilon}_{n,j} := Y_j - \hat{\alpha}_n \mathbf{1}_{Z_j\leq \hat{\zeta}_n} - \hat{\beta}_n \mathbf{1}_{Z_j> \hat{\zeta}_n} \;\; \mbox{ for } j=1,\ldots,n,$$ from the fitted model. Note that these residuals are not guaranteed to have mean 0, so we work with the centered residuals, $\hat{\epsilon}_{n,1}-\bar{\epsilon}_n,\ldots,\hat{\epsilon}_{n,n}-\bar{\epsilon}_n$, where $ \bar{\epsilon}_n = \sum_{j=1}^n \hat{\epsilon}_{n,j}/n$. Letting $\mathbb{P}_{n}^{\epsilon}$ denote the empirical measure of the centered residuals, we obtain the bootstrap sample $(Y_{n,1}^*,Z_1),\ldots,(Y_{n,n}^*,Z_n)$ as:
\begin{enumerate}
\item Sample $\epsilon_{n,1}^{*},\ldots,\epsilon_{n,n}^{*}$ independently from $\mathbb{P}_{n}^{\epsilon}$.

\item Fix the predictors $Z_j$, $j=1,\ldots,n$, and define the bootstrapped responses at $Z_j$ as $Y_{n,j}^{*} = \hat{\alpha}_n \mathbf{1}_{Z_j\leq \hat{\zeta}_n} + \hat{\beta}_n \mathbf{1}_{Z_j> \hat{\zeta}_n} + \epsilon_{n,j}^{*}$.
\end{enumerate}
\medskip
%

\noindent {\bf Scheme 3 (Smoothed bootstrap)}: Notice that in (\ref{ec1}), $Z$ is assumed to have a density and it also arises in the limiting distribution of $\Delta_n$. A successful bootstrap scheme must mimic this underlying assumption, and we accomplish this in the following:
\begin{enumerate}
\item Choose an appropriate nonparametric smoothing procedure (e.g., kernel density estimation) to build a distribution $\hat{F}_n$ with a density $\hat{f}_n$ such that $\|\hat{F}_n - F\|_\infty \stackrel{a.s.}{\rightarrow} 0$ and $\hat{f}_n\rightarrow f$ uniformly on some open interval around $\zeta_0$ w.p. 1, where $f$ is the density of $Z$.

\item Get i.i.d. replicates $Z_{n,1}^*,\ldots,Z_{n,n}^*$ from $\hat{F}_n$ and sample, independently, $\epsilon_{n,1}^{*},\ldots,\epsilon_{n,n}^{*}$ from $\mathbb{P}_{n}^{\epsilon}$.

\item Define $Y_{n,j}^{*} = \hat{\alpha}_n \mathbf{1}_{Z_{n,j}^{*}\leq \hat{\zeta}_n} + \hat{\beta}_n \mathbf{1}_{Z_{n,j}^{*}> \hat{\zeta}_n} + \epsilon_{n,j}^{*}$ for all $j=1,\ldots,n$.
\end{enumerate}
\medskip
\noindent {\bf Scheme 4 ($m$ out of $n$ bootstrap)}: A natural alternative to the usual nonparametric bootstrap (i.e., generating bootstrap samples from the ECDF) considered widely in non-regular problems is to use the $m$ out of $n$ bootstrap. We choose a nondecreasing sequence of natural numbers $\{m_n\}_{n = 1}^\infty$ such that $m_n=o(n)$ and $m_n \rightarrow \infty$ and generate the bootstrap sample $(Y_{n,1}^*,Z_{n,1}^*),\ldots,(Y_{n,m_n}^*,Z_{n,m_n}^*)$ from the ECDF of $(Y_1,Z_1),\ldots,(Y_n,Z_n)$. Although there are a number of methods available for choosing the $m_n$ in applications, there is no satisfactory solution to this problem and the obtained CIs usually vary with changing $m_n$.
\medskip

We will use the framework established by our convergence theorems in Section \ref{pgs} to prove that schemes 3 and 4 above yield {\it consistent} bootstrap procedures for building CIs for $\zeta_0$. We will also give strong empirical and theoretical evidence for the {\it inconsistency} of schemes 1 and 2. Note that schemes 1 and 2 are the two most widely used resampling techniques in regression models (see pages 35-36 of \cite{E82}; also see \cite{F81} and \cite{W86}). Thus in this change--point scenario, a typical nonstandard problem, we see that the two standard bootstrap approaches fail. The failure of the usual bootstrap methods in nonstandard situations is not new and has been investigated in the context of M-estimation problems by \cite{BC01} and in situations giving rise to $n^{1/3}$ asymptotics by \cite{AH05} and \cite{sebawo}. But the change-point problem considered in this paper is indeed quite different from the nonstandard problems considered by the above authors --  one key distinction being that compound Poisson processes, as opposed to Gaussian processes, form the backbone of the asymptotic distributions of the estimators -- and thus demands an independent investigation. We will also see later that the performance of scheme 3 clearly dominates that of the $m$ out of $n$ bootstrap procedure (scheme 4), the general recipe proposed in situations where the usual bootstrap does not work (see \cite{LP06} for applications of the $m$ out of $n$ bootstrap procedure in some nonstandard problems). Also note that the performance of the $m$ out of $n$ bootstrap scheme crucially depends on $m$ (see e.g., \cite{BGZ97}) and the choice of this tuning parameter is tricky in applications.

\section{A uniform convergence result}\label{pgs}
In this section we generalize the results obtained in \cite{koss}, pages 271--277, to a triangular array of random variables. Consider the triangular array\newline
$\left\{ X_{n,k} =(Y_{n,k},Z_{n,k}) \right\}_{1\leq k \leq m_n}^{n\in\mathbb{N}}$ defined on a probability space $(\Omega,\mathcal{A},\mathbf{P})$, where $\left(m_n\right)_{n=1}^\infty$ is a nondecreasing sequence of natural numbers such that $m_n \rightarrow \infty$. Throughout the entire paper we will always denote by $\mathbf{E}$ the expectation operator with respect to $\mathbf{P}$. Furthermore, assume that for each $n \in \mathbb{N}$, $(X_{n,1},\ldots, X_{n,m_n})$ constitutes a random sample from an arbitrary bivariate distribution $\mathbb{Q}_n$ with $\mathbb{Q}_n (Y_{n,1}^2) < \infty$ and let $M_n(\theta) := \mathbb{Q}_n(m_\theta)$ for all $\theta\in\Theta$, where $m_\theta$ is defined in (\ref{eq:m_theta}). Let $\mathbb{P}$ be a bivariate distribution satisfying (\ref{ec1}). Recall that $M(\theta) := \mathbb{P}(m_\theta)$ and $\theta_0 := \sargmax M(\theta)$.

Let $\theta_n := (\alpha_n,\beta_n,\zeta_n)$ be given by
\[ \theta_n = \sargmax_{\theta\in\Theta}\{\mathbb{Q}_n(m_\theta)\}.\] Note that $\mathbb{Q}_n$ need not satisfy model (\ref{ec1}) with $(\alpha_n,\beta_n,\zeta_n)$. The existence of $\theta_n$ is guaranteed as $\mathbb{Q}_n(m_\theta)$ is a quadratic function in $\alpha$ and $\beta$ (for a fixed $\zeta$) and bounded and c\'adl\'ag as a function in $\zeta$. For each $n$, let $\mathbb{P}_n^*$ be the empirical measure produced by the random sample $(X_{n,1},\ldots, X_{n,m_n})$, and define the least squares estimator $\theta_n^* = (\alpha_n^*, \beta_n^*, \zeta_n^*)\in\Theta$ to be the smallest argmax of $M_n^*(\theta) := \mathbb{P}_n^* (m_{\theta})$. If $Q$ is a signed Borel measure on $\mathbb{R}^2$ and $\mathscr{F}$ is a class of (possibly) complex-valued functions defined on $\mathbb{R}^2$, write $\left\| Q \right\|_{\mathscr{F}} := \sup \left\{|Q(f)|: f\in\mathscr{F}\right\}$. If $g: K \subset \mathbb{R}^3 \rightarrow \mathbb{R}$ is a bounded function, write $\|g \|_K := \sup_{x \in K} |g(x)|$. Also, for $\left(z,y\right)\in\mathbb{R}^2$ and $n\in\mathbb{N}$ we write
\begin{equation}\label{eq:epsilon-tilde}
\tilde{\epsilon}_n := \function{\tilde{\epsilon}_n}{z,y} = y - \alpha_n \mathbf{1}_{z\leq \zeta_n} - \beta_n \mathbf{1}_{z> \zeta_n}.
\end{equation}
Let $M > 0$ be such that $|\alpha_n | \le M$ for all $n$. We define the following three classes of functions from $\mathbb{R}^2$ into $\mathbb{R}$:
\begin{eqnarray}
\mathcal{F} & := & \left\{ \function{\ind{I}}{z}: I \subset \mathbb{R} \mbox{ is an interval} \right\} , \nonumber \\
\mathcal{G} & := & \left\{y f(z): f \in \mathcal{F} \right\} \cup \left\{|y+\alpha|f(z): f \in \mathcal{F}, |\alpha| \le M \right\}, \nonumber \\
\mathcal{H} & := & \{y^2 f(z): f \in \mathcal{F}\}. \nonumber
\end{eqnarray}
In what follows, we will derive conditions on the distributions $\mathbb{Q}_n$ that will guarantee consistency and weak convergence of $\theta_n^*$.

\subsection{Consistency and the rate of convergence}
We provide first a consistency result for the least squares estimator, whose proof we include in the Appendix (see Section \ref{prueba1}). To this end, we consider the following set of assumptions:
\begin{enumerate}
\item[(I)] $\left\|\mathbb{Q}_n - \mathbb{P}\right\|_\mathcal{F} \rightarrow 0$,
\item[(II)] $\left\|\mathbb{Q}_n - \mathbb{P}\right\|_\mathcal{G} \rightarrow 0$,
\item[(III)] $\left\|\mathbb{Q}_n - \mathbb{P}\right\|_\mathcal{H} \rightarrow 0$,
\item[(IV)] $\theta_n\rightarrow\theta_0$.
\end{enumerate}
\begin{prop}\label{pg1}
Assume that (I)-(IV) hold. Then, $\theta_n^* \cip \theta_0$.
\end{prop}

To guarantee the right rate of convergence, we need to assume stronger regularity conditions. In addition to those of Proposition \ref{pg1}, we require the following:

\begin{enumerate}
\item[(V)] There are $\eta,\rho, L > 0$ with the property that for any $\delta\in(0,\eta)$, there is $N >0$ such that the following inequalities hold for any $n\geq N$:
\begin{eqnarray}
 \inf_{\frac{1}{\sqrt{m_n}}\leq|\zeta-\zeta_n|<\delta^2}\left\{\frac{1}{|\zeta-\zeta_n|}\mathbb{Q}_n (\ind{\zeta\land\zeta_n < Z \leq \zeta\lor\zeta_n})\right\} > \rho, \label{rccs1}\\
 \sup_{|\zeta-\zeta_n|<\delta^2}\left\{|\mathbb{Q}_n(\tilde{\epsilon}_n\ind{\zeta\land\zeta_n < Z \leq \zeta\lor\zeta_n})|\right\} \le \frac{L \delta}{\sqrt{m_n}}, \label{rccs2}\\
 \sup_{|\zeta-\zeta_n|<\delta^2}\left\{|\mathbb{Q}_n(\tilde{\epsilon}_n\ind{Z \leq \zeta\land\zeta_n})|+|\mathbb{Q}_n(\tilde{\epsilon}_n\ind{Z > \zeta\lor\zeta_n})|\right\} \le \frac{L}{\sqrt{m_n}}. \label{rccs3}
\end{eqnarray}
\end{enumerate}
We would like to point out some facts about (V). It must be noted that (\ref{rccs2}) and (\ref{rccs3}) automatically hold in the case where $Z$ and $\tilde\epsilon_n$ are independent under $\mathbb{Q}_n$ with $\mathbb{Q}_n(\tilde\epsilon_n)=0$. Also, (\ref{rccs1}) is easily seen to hold when the $Z$'s, under $\mathbb{Q}_n$, have densities $f_n$ converging uniformly to $f$ in some neighborhood of $\zeta_0$, where $f$ is the density of $Z$ under $\mathbb{P}$; by a consequence of the classical mean value theorem of calculus.

With the aid of these conditions, Proposition \ref{pg1} and Theorem 3.4.1, page 322, of \cite{vw} we can now state and prove (see Section \ref{prueba2}) the rate of convergence result.

\begin{prop}\label{pg2}
Assume that (I)-(V) hold. Then $\sqrt{m_n}(\alpha_n^* - \alpha_n) = \function{O_{\mathbf{P}}}{1}$, $\sqrt{m_n}(\beta_n^* - \beta_n) = \function{O_{\mathbf{P}}}{1}$ and $m_n (\zeta_n^* - \zeta_n) = \function{O_{\mathbf{P}}}{1}$.
\end{prop}

Propositions \ref{pg1} and \ref{pg2} provide sufficient conditions on the measures $\mathbb{Q}_n$, the distribution of each element in the $n$th row of the triangular array, to achieve the same rate of convergence as the original least squares estimators. We would like to highlight that we are not assuming that each $\mathbb{Q}_n$ satisfy the model (\ref{ec1}) with $(\alpha_n,\beta_n,\zeta_n)$; all we need is that $\mathbb{Q}_n$ and $\theta_n$ approach $\mathbb{P}$ and $\theta_0$ respectively, in a suitable manner.

\subsection{Weak Convergence and asymptotic distribution}
We start with some additional set of assumptions:
\begin{enumerate}
\item[(VI)] For any function $\psi:\mathbb{R}\rightarrow \mathbb{C}$ which is either of the form $\psi(x) = e^{i \xi x}$ for some $\xi\in\mathbb{R}$ or defined by $\psi(x)=|x|^p$ for $p=1,2$, we have:
\[m_n\mathbb{Q}_n \left( \psi(\tilde{\epsilon}_n)\ind{\zeta_n - \frac{\delta}{m_n} < Z \leq \zeta_n + \frac{\eta}{m_n}} \right) \rightarrow f(\zeta_0)(\delta+\eta)\mathbb{P}\left(\psi(\epsilon)\right) \ \ \ \forall\  \eta,\delta > 0.\]

\item[(VII)] $\sqrt{m_n} \mathbb{Q}_n(\tilde{\epsilon}_n \ind {Z\leq\zeta_n})\rightarrow 0$ and $\sqrt{m_n} \mathbb{Q}_n (\tilde{\epsilon}_n \ind {Z>\zeta_n})\rightarrow 0$.

\item[(VIII)] $\lsup_{n \rightarrow \infty} \mathbb{Q}_n(|\tilde {\epsilon}_n|^3)<\infty.$
\end{enumerate}

Observe that condition (VI) implies, for all $\eta,\delta>0$, and $p=1,2$,
\begin{eqnarray}\label{ec49}
\sqrt{m_n} \mathbb{Q}_n \left(|\tilde{\epsilon}_n|^p\ind{\zeta_n - \frac{\delta}{m_n}<Z\leq \zeta_n + \frac{\eta}{m_n}}\right) &\rightarrow& 0, \label{ec48} \\
\sqrt{m_n} \mathbb{Q}_n \left(\ind{\zeta_n - \frac{\delta}{m_n}<Z\leq \zeta_n + \frac{\eta}{m_n}}\right) &\rightarrow& 0. \label{ec49}
\end{eqnarray}
For $h = (h_1,h_2,h_3) \in \mathbb{R}^3$, let $\vartheta_{n,h} := \theta_n + \left(\frac{h_1}{\sqrt{m_n}},\frac{h_2}{\sqrt{m_n}},\frac{h_3}{m_n}\right)$ and
\begin{eqnarray}
\hat{E}_n(h) := m_n\mathbb{P}_n^* \left[m_{\vartheta_{n,h}} - m_{\theta_n} \right] \nonumber.
\end{eqnarray}
We will argue that $$h_n^* := \sargmin_{h \in \mathbb{R}^3} \hat E_n(h) = \left(\sqrt{m_n}(\alpha_n^* - \alpha_n), \sqrt{m_n}(\beta_n^* - \beta_n), m_n(\zeta_n^* - \zeta_n)\right)$$ converges in distribution to the smallest argmax of some process involving two independent normal random variables and a two-sided, compound Poisson process (independent of the normal variables).

We derive the asymptotic distribution of the process $\hat E_n$ and then apply continuous mapping techniques to obtain the limiting distribution of $h_n^*$. We consider these stochastic processes as random elements in the space $\mathcal{D}_K$, for a given compact rectangle $K \subset \mathbb{R}^3$, of all functions $W:K \rightarrow \mathbb{R}$ having ``quadrant limits'' (as defined in \cite{neu}), being continuous from above (again, in the terminology of \cite{neu}) and such that $\function{W}{\cdot,\cdot,\zeta}$ is continuous for all $\zeta$ and $\function{W}{\alpha,\beta,\cdot}$ is c\`adl\`ag (right continuous having left limits) for all $(\alpha,\beta)$. Write $\mathcal{D}=\mathcal{D}_{\mathbb{R}^3}$. For any compact interval $I\subset\mathbb{R}$ let \[\Lambda_I=\left\{\lambda:I\rightarrow I\left| \lambda \textrm{ is strictly increasing, surjective and continuous}\right.\right\}\]
and write \[ \left\|\lambda\right\| := \sup_{s\neq t\in I}\left|\log \frac{\lambda(s) - \lambda(t)}{s-t}\right|. \]

Then, for any set of the form $K=A\times I$ with $A\subset\mathbb{R}^2$ define the Skorohod topology as the topology given by the metric
\[ \function{d_K}{\Psi,\Gamma} := \inf_{\lambda\in\Lambda_I}\left\{\sup_{(\alpha,\beta,\zeta)\in K}\left\{\left|\function{\Psi}{\alpha,\beta,\zeta}-\function{\Gamma}{\alpha,\beta,\lambda(\zeta)}\right|\right\} + \left\|\lambda\right\| \right\}\ \ \] for $\Gamma,\Psi\in\mathcal{D}_K$. Endowed with this metric, $\mathcal{D}_K$ becomes a Polish space (it is a closed subspace of the Polish spaces $D_k$ defined in \cite{neu}) and thus the existence of conditional probability distributions for its random elements is ensured (see \cite{du}, Theorem 10.2.2 page 345). Also, let $\tilde{\mathcal{D}_I}$, $I \subset \mathbb{R}$, denote the space of real valued c\`adl\`ag functions on $I$. We refer the reader to Section \ref{D_K} for some results about the Skorohod space.

We express the process $\hat E_n$ as the sum of the four terms $\hat{A}_n$, $\hat{B}_n$, $\hat{C}_n$ and $\hat{D}_n$ where
\begin{eqnarray}
    \hat{A}_n (h_1,h_3) & := & 2h_1 \sqrt{m_n} \mathbb{P}_n^*\left(\tilde{\epsilon}_n\ind{Z\leq \zeta_n \land \left(\zeta_n+\frac{h_3}{m_n}\right)}\right) - h_1^2\mathbb{P}_n^*\left(\ind{Z\leq \zeta_n \land \left(\zeta_n+\frac{h_3}{m_n}\right)}\right), \nonumber \\
    \hat{B}_n (h_2,h_3) & := & 2h_2 \sqrt{m_n} \mathbb{P}_n^*\left(\tilde{\epsilon}_n\ind{Z > \zeta_n \lor \left(\zeta_n+\frac{h_3}{m_n}\right)}\right) - h_2^2\mathbb{P}_n^*\left(\ind{Z > \zeta_n \lor \left(\zeta_n+\frac{h_3}{m_n}\right)}\right), \nonumber \\
    \hat{C}_n (h_2,h_3) & := & -2m_n\left(\alpha_n - \beta_n + \frac{h_2}{\sqrt{m_n}}\right)\mathbb{P}_n^*\left(\tilde{\epsilon}_n \ind {\zeta_n+\frac{h_3}{m_n}<Z\leq \zeta_n }\right) \nonumber \\
    & & \qquad \qquad - \ m_n\left(\alpha_n - \beta_n + \frac{h_2}{\sqrt{m_n}}\right)^2\mathbb{P}_n^*\left( \ind{\zeta_n+\frac{h_3}{m_n}<Z\leq \zeta_n }\right), \nonumber \\
    \hat{D}_n (h_1,h_3) & := & -2m_n\left(\beta_n - \alpha_n + \frac{h_1}{\sqrt{m_n}}\right)\mathbb{P}_n^*\left(\tilde{\epsilon}_n \ind{\zeta_n<Z\leq \zeta_n+\frac{h_3}{m_n}}\right) \nonumber \\
    & & \qquad \qquad - \ m_n\left(\beta_n - \alpha_n + \frac{h_1}{\sqrt{m_n}}\right)^2\mathbb{P}_n^*\left(\ind{\zeta_n<Z\leq \zeta_n+\frac{h_3}{m_n} }\right). \nonumber
\end{eqnarray}
We define another process $E_n^* := A_n^* + B_n^* + C_n^* + D_n^*$ where
\begin{eqnarray}
    A_n^* (h_1) & := & 2h_1 \sqrt{m_n} \mathbb{P}_n^*\left(\tilde{\epsilon}_n\ind{Z\leq \zeta_n }\right) - h_1^2\mathbb{P}_n^*\left(\ind{Z\leq \zeta_n}\right), \nonumber \\
    B_n^* (h_2) & := & 2h_2 \sqrt{m_n} \mathbb{P}_n^*\left(\tilde{\epsilon}_n\ind{Z > \zeta_n }\right) - h_2^2\mathbb{P}_n^*\left(\ind{Z> \zeta_n}\right), \nonumber \\
    C_n^* (h_3) & := & -2m_n(\alpha_n - \beta_n)\mathbb{P}_n^*\left(\tilde {\epsilon}_n\ind{\zeta_n+\frac{h_3}{m_n}<Z\leq \zeta_n }\right) \nonumber \\
    & & \qquad \qquad - \; m_n(\alpha_n - \beta_n)^2\mathbb{P}_n^*\left( \ind{\zeta_n+\frac{h_3}{m_n}<Z\leq \zeta_n }\right), \nonumber \\
    D_n^* (h_3) & := & -2m_n(\beta_n - \alpha_n)\mathbb{P}_n^*\left(\tilde {\epsilon}_n\ind{\zeta_n
    <Z\leq \zeta_n+\frac{h_3}{m_n} }\right) \nonumber \\
    & & \qquad \qquad - \; m_n(\beta_n - \alpha_n)^2\mathbb{P}_n^*\left(\ind{\zeta_n<Z\leq \zeta_n+\frac{h_3}{m_n} } \right). \nonumber
\end{eqnarray}
We work with $E_n^*$ instead of $\hat{E}_n$ as their difference approaches uniformly to 0 in probability, as shown in the next lemma (proved in Section \ref{prueba3}), and the asymptotic distribution of $E_n^*$ is easier to derive.

\begin{lemma}\label{l14}
Let $K\subset\mathbb{R}^3$ be a compact rectangle. If conditions (I)-(IV) and (\ref{ec48}) and (\ref{ec49}) hold, then \[\left\|E_n^* - \hat{E}_n\right\|_K\cip 0.\] Therefore, $E_n^* - \hat{E}_n \cip 0$ as random elements of $\mathcal{D}_K$. In particular, this result is true under conditions (I)-(IV) and (VI).
\end{lemma}

As a first step to finding the asymptotic distribution of $(E_n^*)_{n=1}^\infty$, we show that the random sequence is tight in the Skorohod space $\mathcal{D}_K$ for any compact rectangle $K \subset \mathbb{R}^3$. The proof of the next result is given in Section \ref{prueba4}.
\begin{lemma}\label{l15}
Let $I\subset\mathbb{R}$ be a compact interval and assume that conditions (I)-(VIII) hold. Then, the sequence of $\mathbb{R}^6$-valued processes
\begin{equation}\label{ec44}
\Xi_n (t) := \left(\begin{array}{c}
\sqrt{m_n}\mathbb{P}_n^*(\tilde{\epsilon}_n\ind{Z\leq \zeta_n})\\
\sqrt{m_n}\mathbb{P}_n^*(\tilde{\epsilon}_n\ind{Z> \zeta_n})\\
m_n\mathbb{P}_n^*(\ind{\zeta_n + \frac{t}{m_n}<Z\leq \zeta_n})\\
m_n\mathbb{P}_n^*(\tilde{\epsilon}_n\ind{\zeta_n + \frac{t}{m_n}<Z\leq \zeta_n})\\
m_n\mathbb{P}_n^*(\ind{\zeta_n<Z\leq\zeta_n + \frac{t}{m_n}})\\
m_n\mathbb{P}_n^*(\tilde{\epsilon}_n\ind{\zeta_n<Z\leq\zeta_n + \frac{t}{m_n}})
\end{array}\right)
\end{equation}
is uniformly tight in $\mathbb{R}^2 \times \tilde{\mathcal{D}_I^4}$. Also, if $K\subset\mathbb{R}^3$ is a compact rectangle, the sequence $(E_n^*)_{n=1}^\infty$ is uniformly tight in $\mathcal{D}_K$.
\end{lemma}

It now suffices to show convergence of the finite-dimensional distributions of the processes $E_n^*$ to the finite dimensional distributions of some process $E^*\in\mathcal{D}_K$ to conclude that $E_n^*$ converges weakly to $E^*$ (and thus $\hat{E}_n$ too). With this objective in mind, we make the following definitions: Let $\mathbf{Z}_1 \sim \function{\mathbf{N}}{0,\sigma^2 \mathbb{P}(Z\leq\zeta_0)}$ and $\mathbf{Z}_2 \sim \function{\mathbf{N}}{0, \sigma^2\mathbb{P} (Z > \zeta_0)}$ be two independent normal random variables; $\nu_1$ and $\nu_2$ be, respectively, left-continuous and right-continuous, homogeneous Poisson processes with rate $f(\zeta_0)>0$; $\mathbf{u}=(u_n)_{n =1}^\infty$ and $\mathbf{v}=(v_n)_{n=1}^\infty$ two sequences of i.i.d. random variables having the same distribution as $\epsilon$ under $\mathbb{P}$. Assume, in addition, that $\mathbf{Z}_1$, $\mathbf{Z}_2$, $\nu_1$, $\nu_2$, $\mathbf{v}$ and $\mathbf{u}$ are all mutually independent. Then, define the process $\Xi = (\Xi^{(1)},\ldots,\Xi^{(6)})'$ as
\begin{equation}
\Xi (t) := \left(\begin{array}{c}
\mathbf{Z}_1\\
\mathbf{Z}_2\\
\nu_1 (-t)\ind{t < 0}\\
\sum_{0<j\leq \nu_1(-t)} v_j \ind{t < 0}\\
\nu_2 (t)\ind{t \geq 0}\\
\sum_{0<j\leq \nu_2(t)} u_j\ind{t \geq 0}
\end{array}\right) \label{ec45}
\end{equation}
and let $E^*$ be given by
\begin{eqnarray}\label{eq46}
E^*(h)  & := & 2h_1 \Xi^{(1)}(h_3) - h_1^2 \mathbb{P}(Z\leq \zeta_0) + 2h_2 \Xi^{(2)}(h_3) - h_2^2 \mathbb{P}(Z>\zeta_0)\nonumber\\
& & \qquad \qquad + \ 2(\beta_0 - \alpha_0)\Xi^{(4)}(h_3) - (\alpha_0 - \beta_0)^2\Xi^{(3)}(h_3)\nonumber\\
& & \qquad \qquad + \ 2(\alpha_0 - \beta_0)\Xi^{(6)}(h_3) - (\alpha_0 - \beta_0)^2\Xi^{(5)}(h_3)\label{ec46}
\end{eqnarray}
for $h = (h_1,h_2,h_3) \in \mathbb{R}^3$.  \newline

We will now prove weak convergence of the sequence of processes $(\hat{E}_n)_{n=1}^\infty$ to $E^*$, and then use a continuous mapping theorem for the smallest argmax functional (see Lemma \ref{l11}) to obtain weak convergence of $h_n^* := \sargmax \hat{E}_n(h)$. The application of Lemma \ref{l11} requires the weak convergence of processes $(\hat{E}_n)_{n=1}^\infty$ to $E^*$ and also the weak convergence of their associated jump processes. Let $\mathcal{S}$ be the class of all piecewise constant, c\'adl\'ag functions $\tilde{\psi}:\mathbb{R}\rightarrow\mathbb{R}$ that are continuous on the integers with $\tilde{\psi}(0)=0$; $\tilde{\psi}$ has jumps of size 1, and $\tilde{\psi}(-t)$ and $\tilde{\psi}(t)$ are nondecreasing on $(0,\infty)$. For an interval $I$ containing 0 in its interior, we write $\mathcal{S}_I = \left\{f\left|_I\right. : f\in\mathcal{S}\right\}$. Define the $\mathcal{S}$--valued (pure jump) processes $\hat{J}_n$, $J_n^*$ and $J^*$ as
\begin{eqnarray}
J_n^*(t) = \hat{J}_n (t) & := & m_n\mathbb{P}_n^* (\ind{\zeta_n + \frac{t}{m_n}< Z \leq \zeta_n})
+ m_n\mathbb{P}_n^* (\ind{\zeta_n < Z \leq \zeta_n + \frac{t}{m_n}}), \nonumber\\
J^* (t) & := & \nu_1(-t)\ind{t<0} + \nu_2(t) \ind{t\geq 0}.\nonumber
\end{eqnarray}
\begin{lemma}\label{l16}
Let $I\subset \mathbb{R}$ be a compact interval and $K=A\times B \times I\subset\mathbb{R}^3$ a compact rectangle. If (I)-(VIII) hold, we have
\begin{enumerate}[(i)]
\item $\Xi_n \rightsquigarrow \Xi$ in $\mathbb{R}^2 \times \tilde{\mathcal{D}_I^4}$,
\item $(E_n^*,J_n^*)\rightsquigarrow (E^*,J^*)$ in $\mathcal{D}_K\times \mathcal{S}_I$,
\item $(\hat{E}_n,\hat{J}_n)\rightsquigarrow (E^*,J^*)$ in $\mathcal{D}_K\times \mathcal{S}_I$,
\end{enumerate}
where $\rightsquigarrow$ denotes weak convergence.
\end{lemma}
For a proof of the convergence result, see Section \ref{prueba5}.

To apply the argmax continuous mapping theorem we first show that the the smallest argmax of $E^*$ is well defined. The proof of the next lemma is provided in Section \ref{prueba6}.
\begin{lemma}\label{l17}
Consider the process $E^*$ defined in (\ref{ec46}). Then, for almost every sample path of $E^*$, $\displaystyle \phi^* = (\phi_1^*, \phi_2^*, \phi_3^*):=\sargmax_{h\in\mathbb{R}^3}\{E^*(h)\}$ is well-defined. Moreover, $\phi_1^*$, $\phi_2^*$ and $\phi_3^*$ are independent; and $\phi_1^*$ and $\phi_2^*$ are distributed as normal random variables with mean 0 and variances $\sigma^2/\mathbb{P}(Z\leq \zeta_0)$ and $\sigma^2/\mathbb{P}(Z> \zeta_0)$, respectively.
\end{lemma}

We now state the distributional convergence result for the sequence of least squares estimator $\theta_n^*$. For a proof, we refer the reader to Section \ref{prueba7}.

\begin{prop}\label{pg3}
With the notation of Lemma \ref{l17}, if conditions (I)-(VIII) hold, then
\[ h_n^* = \left(\begin{array}{c}
\sqrt{m_n}(\alpha_n^*-\alpha_n)\\
\sqrt{m_n}(\beta_n^* -\beta_n)\\
m_n (\zeta_n^* - \zeta_n)
\end{array}\right) \rightsquigarrow \sargmax_{h\in\mathbb{R}^3}\{E^*(h)\}. \]
\end{prop}

If we take $\mathbb{Q}_n = \mathbb{P}$ and $m_n = n$ $\forall n\in\mathbb{N}$, it is easily seen that $\theta_n = \theta_0$ and conditions (I)-(VIII) hold. Hence, we immediately get the following corollary.

\begin{cor}[Asymptotic distribution of the least squares estimators]\label{corolario}
For the least squares estimators $(\hat{\alpha}_n,\hat{\beta}_n,\hat{\zeta}_n)$ based on an i.i.d. sequence $(X_n)_{n=1}^\infty$ satisfying (\ref{ec1}), we have
\[ (
\sqrt{n}(\hat{\alpha}_n-\alpha_0),
\sqrt{n}(\hat{\beta}_n -\beta_0),
n (\hat{\zeta}_n - \zeta_0))' \rightsquigarrow \sargmax_{h\in\mathbb{R}^3}\{E^*(h)\}. \]
\end{cor}

\section{Inconsistency of the bootstrap}\label{incons} In this section we argue the inconsistency of the two most common bootstrap procedures in regression: the ECDF bootstrap (scheme 1) and the residual bootstrap (scheme 2). Recall the notation and definitions in the beginning of Section \ref{BootsSchemes}. In particular, note that we have i.i.d. random vectors $\{X_n = (Y_n,Z_n)\}_{n=1}^\infty$ from (\ref{ec1}) with parameter $\theta_0$ defined on a probability space $(\Omega,\mathcal{A},\mathbf{P})$ and let $\mathbb{P}_n$ be the empirical distribution of the first $n$ data points. We start by stating two results that will be used in the sequel. We first show that the least squares estimator $\hat \theta_n$ of $\theta_0$ is strongly consistent. This is an improvement of the result obtained in \cite{koss} and we refer the reader to Section \ref{app3} for a complete proof. The proof of the second lemma can be found in Section \ref{pl5}.
\begin{lemma}\label{l3}
Let $K\subset\Theta$ be any compact rectangle. Then,
\begin{enumerate}[(i)]
\item $\left\|M_n - M\right\|_K \stackrel{a.s.}{\longrightarrow} 0$,
\item $M_n \stackrel{a.s.}{\longrightarrow} M$ in $\mathcal{D}_K$,
\item $\hat{\theta}_n\stackrel{a.s.}{\longrightarrow}\theta_0$.
\end{enumerate}
\end{lemma}

\begin{lemma}\label{l5} Let $K\subset\mathbb{R}$ be a compact interval and $(m_n)_{n=1}^\infty$ be an increasing sequence of natural numbers such that $m_n \rightarrow \infty$ and $m_n=O(n)$. Then,
\begin{enumerate}[(i)]
\item $m_n^\gamma\left\|\mathbb{P}_n(\hat{\zeta}_n + \frac{(\cdot)}{m_n} < Z \leq
    \hat{\zeta}_n)\right\|_K\stackrel{\mathbf{P}}{\longrightarrow} 0$ for any $\gamma<1$, and
\item $m_n^\gamma\left\|\function{\mathbb{P}_n}{|\tilde{\epsilon}_n|^p\ind{\hat{\zeta}_n + \frac{(\cdot)}{m_n} < Z \leq \hat{\zeta}_n}}\right\|_K\stackrel{\mathbf{P}}{\longrightarrow} 0$ for any $\gamma<1$, \mbox{ and } p = 1,2.
\end{enumerate}
These statements are still true if $\ind{\hat{\zeta}_n + \frac{(\cdot)}{m_n} < Z \leq \hat{\zeta}_n}$ is replaced by $\ind{\hat{\zeta}_n < Z \leq \hat{\zeta}_n + \frac{(\cdot)}{m_n}}$.
\end{lemma}

We introduce some notation. Let $(\texttt{X},d)$ be a metric space and consider the $\texttt{X}$-valued random elements $V$ and $(V_n)_{n=1}^\infty$ defined on $(\Omega,\mathcal{A},\mathbf{P})$. We say that $V_n$ converges conditionally in probability to $V$, almost surely, and write $V_n\ccipas V$, if
\begin{equation}\label{ccipas}
\mathbf{P}_\mathfrak{X}(d(V_n,V)>\epsilon)\cas 0\ \ \ \forall\ \epsilon>0.
\end{equation}
Similarly, we write $V_n\ccipip V$ and say that $V_n$ converges conditionally in probability to $V$, in probability, if the left--hand side of (\ref{ccipas}) converges in probability to 0.

\subsection{Scheme 1 (Bootstrapping from ECDF)} Consider the notation and definitions of Section \ref{Boots}. To translate this scheme into the framework of Propositions \ref{pg1}, \ref{pg2} and \ref{pg3}, we set $m_n = n$, $\mathbb{Q}_n = \mathbb{P}_n$ and consider the triangular array $\left\{X_{n,k}^*=(Y_{n,k}^*,Z_{n,k}^*)\right\}_{1\leq k \leq n}^{n\in\mathbb{N}}$. Moreover, from Lemma \ref{l3} we know that $\hat{\theta}_n\cas\theta_0$, so we can also take $\theta_n = \hat{\theta}_n$. We first prove that the bootstrapped estimators converge conditionally in probability to the true value of the parameters, almost surely.

\begin{prop}\label{convs1}
For the ECDF bootstrap, we have $\theta^*_n\ccipas\theta_0$.
\end{prop}

\begin{proof} Since $Y$ has a second moment under $\mathbb{P}$, it is straightforward to see that $\mathcal{F}$, $\mathcal{G}$ and $\mathcal{H}$ are VC-subgraph classes with integrable envelopes $1$, $|Y| + M$ and $Y^2$, respectively. It follows that all these classes are Glivenko--Cantelli and therefore conditions (I)-(III) hold w.p. 1. Also, note that, from Lemma \ref{l3} $(iii)$ condition (IV) holds a.s. The result then follows from Proposition \ref{pg1}.
\end{proof}
Let $\mathbb{P}_n^*$ be the ECDF of $X_{n,1}^*,\ldots,X_{n,n}^*$ and recall the definition of the processes $\hat{A}_n$, $\hat{B}_n$, $\hat{C}_n$, $\hat{D}_n$, $\hat{E}_n$, $A_n^*$, $B_n^*$, $C_n^*$, $D_n^*$ and $E_n^*$. We then have the following result.
\begin{lemma}\label{l20} Let $K\subset\mathbb{R}^3$ be any compact rectangle. Then
\[\hat{E}_n - E_n^* \ccipip 0 \textrm{ in } \mathcal{D}_K .\]
\end{lemma}
\begin{proof}
We already know that conditions (I)-(IV) hold w.p. 1 under this bootstrap scheme. But Lemma \ref{l5} implies that (\ref{ec48}) and (\ref{ec49}) hold in probability. Hence, this result follows by arguing through subsequences and applying Lemma \ref{l14}.
\end{proof}
It is evident that condition (VI) doesn't hold in this situation as we know that
\begin{equation}\label{ecultya}
n\mathbb{P}_n (\zeta_0-\frac{\eta}{n} < Z \leq \zeta_0 + \frac{\delta}{n}) \rightsquigarrow \textrm{Poisson} \big(f(\zeta_0)(\delta+\eta)\big).
\end{equation}
Hence, we cannot use Proposition \ref{pg3} to derive the limit behavior of $h_n^*$.

We will now argue that $E_n^*$, and therefore $\hat{E}_n$, {\it does not have any weak limit} in probability. This statement should be thought in terms of the Prokhorov metric (or any other metric metrizing weak convergence on $\mathcal{D}_K$). If we denote by $\rho_K$ the Prokhorov metric on the space of probability measures on $\mathcal{D}_K$ and by $\mu_n$ the conditional distribution of $E_n^*$ given $\mathfrak{X}$, to say that $(E_n^*)_{n=1}^\infty$ has no weak limit in probability means that there is no probability measure $\mu$ defined on $\mathcal{D}_K$ such that $\function{\rho_K}{\mu_n,\mu} \cip 0$.

The following lemma (proved in Section \ref{prueba21}) will help us show that the (conditional) characteristic functions corresponding to the finite dimensional distributions of $E_n^*$ fail to have a limit in probability, which would, in particular, imply that $E_n^*$ does not have a weak limit in probability.

\begin{lemma}\label{l19}
The following statements hold:
\begin{enumerate}[(i)]
\item For any two real numbers $s<t$, $\left\{n \mathbb{P}_n(\zeta_0+\frac{s}{n}<Z\leq\zeta_0 + \frac{t}{n})\right\}_{n=1}^\infty$ does not converge in probability.

\item There is $h_*>0$ such that for any $h\geq h_*$, the sequences \\ $\left\{n \mathbb{P}_n( \hat{\zeta}_n< Z \leq \hat{\zeta}_n + \frac{h}{n})\right\}_{n=1}^\infty$ and $\left\{ n \mathbb{P}_n(\hat{\zeta}_n - \frac{h}{n}< Z \leq\hat{\zeta}_n)\right\}_{n=1}^\infty$ do not converge in probability.

\item For any two real numbers $s<t$ and any measurable function $\phi:\mathbb{R}\rightarrow\mathbb{R}$, $\left\{n \mathbb{P}_n(\phi(Y)\ind{\zeta_0+\frac{s}{n}<Z\leq\zeta_0 + \frac{t}{n}})\right\}_{n=1}^\infty$ does not converge in probability.

\item Let $\phi$ be a measurable function which is either nonnegative or nonpositive and such that $\phi(\epsilon+\alpha_0)$ and $\phi(\epsilon+\beta_0)$ are nonconstant random variables with finite second moment. Then, there is $h_*>0$ such that for any $h\geq h_*$ \\ $\left\{n \mathbb{P}_n(\phi(Y)\ind{ \hat{\zeta}_n< Z \leq \hat{\zeta}_n + \frac{h}{n}})\right\}_{n=1}^\infty$ and $\left\{ n \mathbb{P}_n(\phi(Y)\ind{\hat{\zeta}_n - \frac{h}{n}< Z \leq\hat{\zeta}_n})\right\}_{n=1}^\infty$ do not converge in probability.
\end{enumerate}
\end{lemma}

With the aid of Lemma \ref{l19} we are now able to state our main result.

\begin{lemma}\label{l21}
There is a compact rectangle $K\subset\mathbb{R}^3$ such that neither $\hat{E}_n$ nor $E_n^*$ has a weak limit in probability in $\mathcal{D}_K$.
\end{lemma}
\begin{proof} Since Lemma \ref{l20} and Slutsky's lemma show that $\hat{E}_n$ has a weak limit in probability if and only if  $E_n^*$ has a weak limit in probability, it suffices to argue that the statement is true for $E_n^*$. To prove this, it is enough to show that there is some $h_3$ such that $E_n^* (0,0,h_3)$ does not converge in distribution. Pick $h_3 >0$ and observe that
\[ E_n^* (0,0,h_3) = (\hat{\alpha}_n - \hat{\beta}_n)\left(n\mathbb{P}_n^* \left[(2\tilde{\epsilon}_n - \hat{\alpha}_n + \hat{\beta}_n) \ind{\hat{\zeta}_n<Z\leq \hat{\zeta}_n + \frac{h_3}{n}}\right]\right).\]
Since $\hat{\alpha}_n - \hat{\beta}_n\cas \alpha_0 - \beta_0 \neq 0$ we see that $E_n^* (0,0,h_3)$ will converge weakly in probability if and only if $\Lambda_n := n \mathbb{P}_n^* \left[(2\tilde{\epsilon}_n - \hat{\alpha}_n + \hat{\beta}_n)\ind{\hat{\zeta}_n<Z\leq \hat{\zeta}_n + \frac{h_3}{n}}\right]$ converges weakly in probability.

The conditional characteristic function of $\Lambda_n$ is given by
\begin{equation}\label{ecrev1}
 \ce{e^{i\xi\Lambda_n}} = \left(1+\frac{1}{n}n\mathbb{P}_n\left((\textrm{e}^{i\xi(2\tilde\epsilon_n +\hat{\beta}_n - \hat{\alpha}_n)} -1)\ind{\hat{\zeta}_n<Z\leq \hat{\zeta}_n + \frac{h_3}{n}}\right)\right)^n,
\end{equation}
which converges in probability if and only if so does
$$n\mathbb{P}_n\left((\textrm{e}^{i\xi(2\tilde\epsilon_n +\hat{\beta}_n - \hat{\alpha}_n)} -1)\ind{\hat{\zeta}_n<Z\leq \hat{\zeta}_n + \frac{h_3}{n}}\right).$$.
But note that
\[n\mathbb{P}_n\left((\textrm{e}^{i\xi(2\tilde\epsilon_n +\hat{\beta}_n - \hat{\alpha}_n)} -1)\ind{\hat{\zeta}_n<Z\leq \hat{\zeta}_n + \frac{h_3}{n}}\right) = n\mathbb{P}_n\left((\textrm{e}^{i\xi(2Y - \hat{\beta}_n - \hat{\alpha}_n)} -1)\ind{\hat{\zeta}_n<Z\leq \hat{\zeta}_n + \frac{h_3}{n}}\right).\]
It is easily seen that (\ref{ecultya}) and the fact that $n(\hat\zeta_n - \zeta_0) = O_\mathbf{P}(1)$ imply that \[n\mathbb{P}_n\left(\ind{\hat{\zeta}_n<Z\leq \hat{\zeta}_n + \frac{h_3}{n}}\right) = O_\mathbf{P}(1).\]
Hence,
$$\left|n\mathbb{P}_n\left((\textrm{e}^{i\xi(2Y - \hat{\beta}_n - \hat{\alpha}_n)} -1)\ind{\hat{\zeta}_n<Z\leq \hat{\zeta}_n + \frac{h_3}{n}}\right) - n\mathbb{P}_n\left((\textrm{e}^{i\xi(2Y - \beta_0 - \alpha_0)} -1)\ind{\hat{\zeta}_n<Z\leq \hat{\zeta}_n + \frac{h_3}{n}}\right)\right|$$
\[ \leq n\mathbb{P}_n\left(\ind{\hat{\zeta}_n<Z\leq \hat{\zeta}_n + \frac{h_3}{n}}\right)(|\hat{\alpha}_n - \alpha_0| + |\hat{\beta}_n - \beta_0|)|\xi|\cip 0.\]
It follows that $\ce{e^{i\xi\Lambda_n}}$ has a limit in probability if and only if $$n\mathbb{P}_n\left((\textrm{e}^{i\xi(2Y - \beta_0 - \alpha_0)} -1)\ind{\hat{\zeta}_n<Z\leq \hat{\zeta}_n + \frac{h_3}{n}}\right)$$
has a limit in probability. But a necessary condition for the latter to happen is that its real part,
\[n\mathbb{P}_n\left(\textrm{Re}(\textrm{e}^{i\xi(2Y - \beta_0 - \alpha_0)} -1)\ind{\hat{\zeta}_n<Z\leq \hat{\zeta}_n + \frac{h_3}{n}}\right)\]
converges in probability. Since $\textrm{Re}(\textrm{e}^{i\xi(2Y - \beta_0 - \alpha_0)} -1)\leq 0$ we can conclude from (iv) of Lemma \ref{l19} that $n\mathbb{P}_n\left(\textrm{Re}(\textrm{e}^{i\xi(2Y - \beta_0 - \alpha_0)} -1)\ind{\hat{\zeta}_n<Z\leq \hat{\zeta}_n + \frac{h_3}{n}}\right)$ does not converge in probability for all $h_3\geq h_*$ for some $h_*>0$ large enough. This in turn implies that, for all $h_3\geq h_*$, the conditional characteristic function in (\ref{ecrev1}) does not converge in probability and hence $E_n^* (0,0,h_3)$ has no weak limit in probability.

Hence, if $K$ is any compact rectangle containing $(0,0,h_*)$ the finite dimensional dimensional distributions of $E_n^*$ on $K$ do not have a weak limit in probability. Therefore, $E_n^*$ does not have a weak limit in probability on $\mathcal{D}_K$. \end{proof}

Note that
\[ \left(\sqrt{n}(\alpha_n^* - \hat{\alpha}_n),\sqrt{n}(\beta_n^* - \hat{\beta}_n),n(\zeta_n^* - \hat{\zeta}_n)\right) = \sargmax_{h\in\mathbb{R}^3} \left\{ \hat{E}_n(h)\right\}. \] Thus, the fact that the sequence $(\hat{E}_n)_{n=1}^\infty$ doesn't have a weak limit in probability makes the existence of a weak limit in probability for $n(\zeta_{n}^* - \hat{\zeta}_n)$ very unlikely. However, we do not have the a rigorous mathematical proof this statement. The main difficulty in such a proof is that the argmax functional is non-linear and that $\hat E_n$ depends on $h_3$ through indicator functions that do not converge in the limit.  \newline

{\bf Remark:} It must be noted in this connection that the bootstrap scheme estimates the distribution of $(\sqrt{n}(\alpha_n^* - \hat{\alpha}_n),\sqrt{n}(\beta_n^* - \hat{\beta}_n))$ correctly, and in fact, valid bootstrap based inference can be conducted to obtain CIs for $\alpha_0$ and $\beta_0$. This follows from the fact that, asymptotically, the maximizers of $\hat E_n (\cdot,\cdot,h_3)$ do not depend on $h_3$ (see the expressions for $\hat A_n^*$, $\hat B_n$, $A_n^*$, $B_n^*$).\newline

We next provide an alternative additional argument that illustrates the inconsistency of the ECDF bootstrap. Our approach is similar to that of \cite{kost} and relies on the asymptotic {\it unconditional} behavior of $$\tilde \Delta_n^* := (\sqrt{n}(\alpha_n^* - \alpha_0),\sqrt{n}(\beta_n^* - \beta_0), n(\zeta_n^* - \zeta_0)).$$ For $h \in \mathbb{R}^3$, we write $\tilde \vartheta_{n,h} := \theta_0 + \left(\frac{h_1}{\sqrt{n}},\frac{h_2}{\sqrt{n}},\frac{h_3}{n}\right)$ and
\begin{eqnarray}
\tilde E_n(h) := n\mathbb{P}_n^* \left[m_{\tilde \vartheta_{n,h}} - m_{\theta_0} \right].\label{ecdereff4}
\end{eqnarray}
This corresponds to centering the objective function around $\theta_0$. As in (\ref{ec44}), we can define the processes
\begin{equation}\label{ecsins}
\tilde{\Xi}_n (t) = \left(\begin{array}{c}
\tilde \Xi_n^{(1)} (t)\\
\tilde \Xi_n^{(2)} (t)\\
\tilde \Xi_n^{(3)} (t)\\
\tilde \Xi_n^{(4)} (t)\\
\tilde \Xi_n^{(5)} (t)\\
\tilde \Xi_n^{(6)} (t)
\end{array}\right) := \left(\begin{array}{c}
\sqrt{n}\mathbb{P}_n^*(\epsilon\ind{Z\leq \zeta_0})\\
\sqrt{n}\mathbb{P}_n^*(\epsilon_n\ind{Z> \zeta_0})\\
n\mathbb{P}_n^*(\ind{\zeta_0 + \frac{t}{n}<Z\leq \zeta_0})\\
n\mathbb{P}_n^*(\epsilon\ind{\zeta_0 + \frac{t}{n}<Z\leq \zeta_0})\\
n\mathbb{P}_n^*(\ind{\zeta_0<Z\leq\zeta_0 + \frac{t}{n}})\\
n\mathbb{P}_n^*(\epsilon\ind{\zeta_0<Z\leq\zeta_0 + \frac{t}{n}})
\end{array}\right)
\end{equation}
and just as in that case, we can also define the process $\tilde E_n^*$ by
\begin{eqnarray}
\tilde E_n^*(h)  & := & 2h_1 \tilde\Xi_n^{(1)}(h_3) - h_1^2 \mathbb{P}_n^*(Z\leq \zeta_0) + 2h_2 \tilde\Xi_n^{(2)}(h_3) - h_2^2 \mathbb{P}_n^*(Z>\zeta_0)\nonumber\\
& & \qquad + \ 2(\beta_0 - \alpha_0)\tilde\Xi_n^{(4)}(h_3) - (\alpha_0 - \beta_0)^2\tilde\Xi_n^{(3)}(h_3)\nonumber\\
& & \qquad +\ 2(\alpha_0 - \beta_0)\tilde\Xi_n^{(6)}(h_3) - (\alpha_0 - \beta_0)^2\Xi_n^{(5)}(h_3).\nonumber
\end{eqnarray}
Then, it can be shown that $\tilde{E}_n - \tilde E_n^* \cip 0 $ in $\mathcal{D}_K$ for any compact rectangle $K\subset\mathbb{R}^3$ and that the sequence $(\tilde E_n^*)_{n=1}^\infty$ is tight in $\mathcal{D}_K$.

In what follows we will describe the limiting distribution of $\tilde E_n^*$, namely $\tilde E^*$, and show that the (unconditional) asymptotic distribution of $\tilde \Delta_n^*$ is that of the smallest argmax of $\tilde E^*$. This result will help us show that the ECDF bootstrap is inconsistent.

We start by introducing some notation. Recall the definitions of the random elements $\mathbf{Z}_1$, $\mathbf{Z}_2$, $\nu_1$, $\nu_2$, $\mathbf{u}$ and $\mathbf{v}$ as in the discussion preceding (\ref{ec45}). Also let $\mathbf{\tau}=(\tau_n)_{n =1}^\infty$ and $\mathbf{\kappa}=(\kappa_n)_{n=1}^\infty$ two sequences of i.i.d. Poisson(1) random variables. Assume, in addition, that $\mathbf{Z}_1$, $\mathbf{Z}_2$, $\nu_1$, $\nu_2$, $\mathbf{v}$, $\mathbf{u}$, $\mathbf{\tau}$ and $\mathbf{\kappa}$ are all mutually independent. Then, define the process $\tilde \Xi = (\tilde \Xi^{(1)},\ldots,\tilde\Xi^{(6)})'$ as
\begin{equation}
\tilde \Xi (t) := \left(\begin{array}{c}
\mathbf{Z}_1\\
\mathbf{Z}_2\\
\sum_{0<j\leq \nu_1(-t)} \kappa_j \ind{t < 0}\\
\sum_{0<j\leq \nu_1(-t)} v_j\kappa_j \ind{t < 0}\\
\sum_{0<j\leq \nu_2(t)} \tau_j\ind{t \geq 0}\\
\sum_{0<j\leq \nu_2(t)} u_j\tau_j\ind{t \geq 0}
\end{array}\right) \label{ecdreff2}
\end{equation}
for $t \in \mathbb{R}$ and let $\tilde E^*$ be given by
\begin{eqnarray}
\tilde E^*(h)  &=& 2h_1 \tilde \Xi^{(1)}(h_3) - h_1^2 \mathbb{P}(Z\leq \zeta_0) + 2h_2 \tilde \Xi^{(2)}(h_3) - h_2^2 \mathbb{P}(Z>\zeta_0)\nonumber\\
 & &\qquad +\ 2(\beta_0 - \alpha_0)\tilde\Xi^{(4)}(h_3) - (\alpha_0 - \beta_0)^2\tilde\Xi^{(3)}(h_3)\nonumber\\
 & &\qquad +\ 2(\alpha_0 - \beta_0)\tilde\Xi^{(6)}(h_3) - (\alpha_0 - \beta_0)^2\tilde\Xi^{(5)}(h_3)\label{ecdreff3}
\end{eqnarray}
for $h = (h_1,h_2,h_3) \in \mathbb{R}^3$. Additionally define the $\mathcal{S}$--valued (pure jump) processes $\tilde{J}_n$, $\tilde J_n^*$ and $\tilde J^*$ as
\begin{eqnarray}
\tilde J_n^*(t) = \tilde{J}_n (t) & := & n\mathbb{P}_n^* (\ind{\zeta_0 + \frac{t}{n}< Z \leq \zeta_0})
+ n\mathbb{P}_n^* (\ind{\zeta_0 < Z \leq \zeta_0 + \frac{t}{n}}), \label{ecdereff5}\\
\tilde J^* (t) & := & \nu_1(-t)\ind{t<0} + \nu_2(t) \ind{t\geq 0}.\label{ecdereff6}
\end{eqnarray}

Lemma \ref{lsho} (proved in Section \ref{plsho}) now states the asymptotic distribution of $\tilde{E}_n$ and of $n(\zeta^*_n - \zeta_0)$.

\begin{lemma}\label{lsho}
Consider the processes $\tilde \Xi_n$, $\tilde{E}_n$, $\tilde{J}_n$, $\tilde \Xi$, $\tilde E^*$ and $\tilde J^*$ as defined in (\ref{ecsins}), (\ref{ecdereff4}), (\ref{ecdereff5}), (\ref{ecdreff2}), (\ref{ecdreff3}) and (\ref{ecdereff6}), respectively. Then, unconditionally,
\begin{enumerate}[(i)]
\item $\tilde\Xi_n\rightsquigarrow\tilde\Xi$ in $\mathbb{R}^2\times\mathcal{D}_I^4$ for any compact interval $I\subset\mathbb{R}$;
\item $(\tilde{E}_n,\tilde{J}_n)\rightsquigarrow (\tilde E^*,\tilde J^*)$ in $\mathcal{D}_K\times\mathcal{S}_I$ for any compact interval $I\subset\mathbb{R}$ and any compact rectangle $K=A\times B\times I\subset\mathbb{R}^3$;
\item $\tilde \Delta_n^* = \sargmax_{h\in\mathbb{R}^3}\{\tilde{E}_n(h)\}\rightsquigarrow \sargmax_{h\in\mathbb{R}^3}\{\tilde E^*(h)\}$.
\end{enumerate}
As a consequence, if the ECDF bootstrap is consistent, the variance of $\sargmax_{h\in\mathbb{R}^3}\{\tilde E^*(h)\}$ must be twice that of $\sargmax_{h\in\mathbb{R}^3}\{E^*(h)\}$.
\end{lemma}

As analytic expressions for the asymptotic variances of $n(\zeta^*_n - \zeta_0)$ and $n(\hat{\zeta}_n - \zeta_0)$ are not known, we use simulations to compute them. As an illustration, we take $\epsilon\sim N(0,1)$, $Z\sim N(0,1)$, $\alpha_0 = -1$, $\beta_0 = 1$ and $\zeta_0=0$ in (\ref{ec1}). We approximate the limiting variances with the sample variances computed from 20,000 observations from each of the two asymptotic distributions. Our results are summarized in the following table, which immediately shows that the asymptotic variance of $n(\zeta^*_n - \zeta_0)$ is not twice that of $n(\hat{\zeta}_n - \zeta_0)$. Thus the ECDF bootstrap cannot be consistent.
\begin{center}
\begin{tabular}{|c|c|}
\hline
Random variable & Asymptotic Variance \\
\hline
$n(\hat{\zeta}_n - \zeta_0)$ & 7.620948 \\
\hline
$n(\zeta^*_n - \zeta_0)$ & 63.98377\\
\hline
\end{tabular}
\end{center}

\subsection{Scheme 2 (Bootstrapping ``residuals'')}
Another resampling procedure that arises naturally in a regression setup is bootstrapping ``residuals''. As with scheme 1, bootstrapping the ``residuals'' fixing the covariates is also {\it inconsistent}. Heuristically speaking, the resampling distribution fails to approximate the density of the predictor at the change-point $\zeta_0$ at rate-$n$, and this leads to the inconsistency.

We recall the notation of Section 2. There we described the basic elements of the traditional fixed-design bootstrap of residuals and how to compute the bootstrap estimates $\theta_n^*$. We first show that these bootstrap estimators converge conditionally in probability (almost surely) to the true value of the parameter. Then, we will provide a strong argument against the consistency of this bootstrap scheme. For notational convenience, we introduce the process $R_n$ given by \[ R_n(\theta) := -\frac{1}{n}\sum_{j=1}^n \left( Y_{n,j}^{*} - \alpha \mathbf{1}_{Z_j\leq \zeta} - \beta \mathbf{1}_{Z_j> \zeta}\right)^2\ \ \forall\ \theta\in\Theta.\]

We start by showing that the ``centered'' empirical distribution for the least squares residuals, $\mathbb{P}_n^\epsilon$, converges to the distribution of $\epsilon$ in total variation distance with probability one and its second moment is an almost surely consistent estimator of $\sigma^2$. This lemma will also be useful for the analysis of the smoothed bootstrap procedure. The proof can be found in Section \ref{prueba11}.
\begin{lemma}\label{l7}
Let $G$ and $\varphi$ be, respectively, the distribution and characteristic functions of $\epsilon$. Then,
\begin{enumerate}[(i)]
\item for any $\eta>0$ we have that
$\displaystyle \sup_{|\xi|\leq \eta}\left\{\left|\int e^{i\xi x}d\mathbb{P}_n^\epsilon (x) - \function{\varphi}{\xi}\right|\right\}\cas 0;$
\item $\left\|\mathbb{P}_n^\epsilon - G\right\|_\mathbb{R} \cas 0$;
\item $\displaystyle \int x^2 d\mathbb{P}_n^\epsilon(x) \cas \sigma^2$;
\item $\displaystyle \int |x|d\mathbb{P}_n^\epsilon (x) \cas \mathbb{P}(|\epsilon|)$;
\item if $\epsilon$ has a finite third moment under $\mathbb{P}$, then
\[ \lsup_{n\rightarrow\infty} \int |x|^3 d \mathbb{P}_n^{\epsilon} (x) <\infty\ \ \ \textrm{almost surely}.\]
\end{enumerate}
\end{lemma}
The next result (proved in Section \ref{prueba17}) shows that the bootstrapped least squares estimators converge conditionally in probability with probability one.
\begin{prop}\label{pfd1} Let $K\subset\Theta$ be a compact rectangle. Then,
\begin{enumerate}[(i)]
\item $\left\|R_n + \mathbb{P}_n^*(\tilde{\epsilon}_n^2) - M_n - \sigma^2\right\|_K \ccipas 0$;
\item $\left\|R_n + \mathbb{P}_n^*(\tilde{\epsilon}_n^2) - M - \sigma^2\right\|_K \ccipas 0$;
\item $\theta_n^*\ccipas\theta_0$ and $\theta_n^*-\hat{\theta}_n \ccipas 0$.
\end{enumerate}
where $M_n$ and $M$ are defined as in (\ref{eq:m_thetabis}) and the subsequent paragraph.
\end{prop}
Consider the following process \[ \hat{E}_n (h) = -\sum_{j=1}^n \left( Y_{n,j}^{*} - \left(\hat{\alpha_n}+\frac{h_1}{\sqrt{n}}\right)\mathbf{1}_{Z_j\leq \hat{\zeta}_n +\frac{h_3}{n}} - \left(\hat{\beta_n}+\frac{h_2}{\sqrt{n}}\right)\mathbf{1}_{Z_j> \hat{\zeta}_n +\frac{h_3}{n}}\right)^2 +\sum_{j=1}^n(\epsilon_{n,j}^*)^2. \]
Then for $n$ large enough we have that
\[ \left(\sqrt{n}(\alpha_n^* - \hat{\alpha}_n),\sqrt{n}(\beta_n^* - \hat{\beta}_n),n(\zeta_n^* - \hat{\zeta}_n)\right) = \sargmax_{h\in\mathbb{R}^3} \left\{ \hat{E}_n(h)\right\}. \]
Next we argue that the sequence $(\hat{E}_n)_{n=1}^\infty$ does not have a weak limit in probability and therefore distributional convergence of their corresponding smallest minimizers seems unreasonable. We refer the reader to Section \ref{prueba19} for a complete proof of the statement.

\begin{lemma}\label{pfd3}
There is a compact rectangle $K\subset\mathbb{R}^3$ such that the sequence of processes $(\hat{E}_n)_{n=1}^\infty$ does not have a weak limit in probability in $\mathcal{D}_K$.
\end{lemma}

\section{Consistent bootstrap procedures}\label{cons}
Here we will prove that the ``smoothed bootstrap'' (scheme 3) and the $m$ out of $n$ bootstrap (scheme 4) procedures yield consistent methods for constructing confidence intervals around the parameters.

\subsection{Scheme 3 (Smoothed Bootstrap)}
To show that scheme 3 (smoothed bootstrap + bootstrapping residuals) achieves consistency we appeal to Propositions \ref{pg1}, \ref{pg2} and \ref{pg3} by proving that the regularity conditions (I)-(VIII) of Section 3 hold for this scheme. Recall the description of this bootstrap procedure given in Section \ref{BootsSchemes}. Let $\hat{f}_n$ and $\hat{F}_n$ be the estimated smoothed density and distribution function of $Z$, respectively. For $I := [c,d] \subset\mathbb{R}$, a compact interval such that $\zeta_0 \in (c,d)$, we require the following two properties of $\hat{f}_n$ and $\hat{F}_n$:
\begin{eqnarray}
\|\hat{F}_n - F\|_{\mathbb{R}} & \cas & 0 ; \label{eq:consF_n} \\
\|\hat{f}_n - f\|_{I} & \cas & 0. \label{ascci}
\end{eqnarray}
We would want to highlight that these conditions are fulfilled by many density estimation procedures. In particular, they hold when the density $f$ is continuous and we let $\hat{f}_n$ be the kernel density estimator constructed from a suitable choice of kernel and bandwidth (e.g., see \cite{sil}).

Let $\theta_n = \hat{\theta}_n$, $m_n=n$ and $\mathbb{Q}_n$ be the distribution that generates the bootstrap sample. Observe that under $\mathbb{Q}_n$, $\tilde{\epsilon}_n$ and $Z$ are independent and that $Z$ is a continuous random variable with density $\hat{f}_n$. The next result (proved in Section \ref{prueba15}) shows that the bootstrapped least squares estimators achieve the right rate of convergence.

\begin{prop}\label{cs3} If (\ref{eq:consF_n}) and (\ref{ascci}) hold, then w.p.1, the sequence of conditional distributions of
$\left(\sqrt{n}(\alpha_n^* - \hat{\alpha}_n),\sqrt{n}(\beta_n^* - \hat{\beta}_n),n(\zeta_n^* - \hat{\zeta}_n)\right)$ is tight.
\end{prop}

Scheme 3 uses an approximation to the density of $Z$ and this turns out to be crucial. The bootstrap measures now satisfy property (VI) on Section \ref{pgs} and the bootstrap procedure is {\it strongly consistent}, as shown in the next result (proved in Section \ref{prueba16}).

\begin{prop}\label{cos3} For scheme 3, provided that (\ref{eq:consF_n}) and (\ref{ascci}) hold, conditions (I)--(VIII) are satisfied with probability one, and thus,
\[ \left(\begin{array}{c}
\sqrt{n}(\alpha_n^*-\hat{\alpha}_n)\\
\sqrt{n}(\beta_n^* -\hat{\beta}_n)\\
n (\zeta_n^* - \hat{\zeta}_n)
\end{array}\right) \rightsquigarrow \sargmax_{h\in\mathbb{R}^3}\left\{E^*(h)\right\} \textrm{ almost surely.} \]
\end{prop}
\subsection{Scheme 4 ($m$ out of $n$ bootstrap)}
For this scheme we will again use the framework established in Section \ref{pgs}. We take $(m_n)_{n=1}^\infty$ to be any sequence of natural numbers which increases to infinity, $\hat{\theta}_n =\theta_n$ and $\mathbb{Q}_n=\mathbb{P}_n$. The next result (proved in Section \ref{prueba20}) shows the {\it weak consistency} of this procedure.
\begin{prop}\label{cs5} If $m_n = o(n)$ and $m_n \rightarrow \infty$, then conditions (I)--(VIII) hold (in probability) and we have
\begin{equation}\label{eq:m_n_boots}
\left(\begin{array}{c}
\sqrt{m_n}(\alpha_n^*-\hat{\alpha}_n)\\
\sqrt{m_n}(\beta_n^* -\hat{\beta}_n)\\
m_n (\zeta_n^* - \hat{\zeta}_n)
\end{array}\right) \rightsquigarrow \sargmax_{h\in\mathbb{R}^3}\left\{E^*(h)\right\} \textrm{ in probability.}
\end{equation}
\end{prop}
{\bf Remark:} To prove Proposition \ref{cs5}, we will, in fact, show that for every subsequence $(n_{k})_{k=1}^\infty$, there is a further subsequence $(n_{k_s})_{s=1}^\infty$, such that (I)-(VIII) hold w.p. 1 for $(n_{k_s})_{s=1}^\infty$ and (\ref{eq:m_n_boots}) holds almost surely along the subsequence $(n_{k_s})_{s=1}^\infty$.

\section{Simulation experiments}\label{simula}
In this section we report the finite sample performance of the different bootstrap schemes on simulated data. We simulated random draws from four different models following (\ref{ec1}). Each of these corresponded to choosing different pairs $(F,G)$ of distributions for $Z$ and $\epsilon$ (having mean 0), respectively. The pairs considered were $(N(0,2), N(0,1))$, $(4B(4,6)-2,N(0,1))$, $(4B(4,6)-2, \textrm{Unif}(-1,1))$, and  $(4B(4,6)-2,\Gamma(4,2)-2)$, where $B(\cdot,\cdot)$ and $\Gamma(\cdot,\cdot)$ denote the beta and gamma distributions respectively.

For each of these models, we considered 1000 random samples of sizes $n=50,100,200,500$. For each sample, and for each of the bootstrap schemes, we took $4n$ bootstrap replicates to approximate the bootstrap distribution. The following table provides the estimated coverage proportions of nominal 95\% CIs and average lengths of the CIs obtained using the 4 different bootstrap schemes for each of the four models.

At this point, we want to make some remarks about the computation of the estimators. We used a kernel density estimator based on the Gaussian kernel and chose the bandwidth by the so-called ``normal reference rule'' (see \cite{sco}, page 131). In the case of the $m$ out of $n$ bootstrap, we did not use any data driven choice of $m_n$, but tried 3 different possibilities: $\lceil n^{\frac{4}{5}} \rceil, \lceil n^{\frac{9}{10}} \rceil$ and  $\lceil n^{\frac{14}{15}} \rceil$. We will refer to the fixed-design bootstrapping of residuals scheme by FDR. \newline

{\scriptsize
\begin{tabular}{|c|c|c|c|c|c|c|}
\hline
\multicolumn{7}{|c|}{$Z\sim N(0,2) , \epsilon\sim N(0,1)$}\\
\hline
\multirow{2}{*}{Scheme} & \multicolumn{2}{|c|}{$n=50$} & \multicolumn{2}{|c|}{$n=200$} & \multicolumn{2}{|c|}{$n=500$}\\ \cline{2-7}
 & Coverage & Avg Length & Coverage & Avg Length & Coverage & Avg Length \\
\hline
ECDF & 0.83 & 1.14 & 0.79 & 0.22 & 0.81 & 0.08 \\
\hline
Smoothed & 0.94 & 0.94 & 0.95 & 0.19 & 0.95 & 0.07 \\
\hline
FDR & 0.83 & 0.76 & 0.86 & 0.16 & 0.90 & 0.06 \\
\hline
$\lceil n^{4/5} \rceil$ & 0.87 & 0.87 & 0.91 & 0.23 & 0.91 & 0.08 \\
\hline
$\lceil n^{9/10} \rceil$ & 0.85 & 1.02 & 0.87 & 0.21 & 0.87 & 0.079 \\
\hline
$\lceil n^{14/15} \rceil$ & 0.85 & 1.05 & 0.84 & 0.21 & 0.86 & 0.08\\
\hline
\end{tabular}
}

{\scriptsize
\begin{tabular}{|c|c|c|c|c|c|c|}
\hline
\multicolumn{7}{|c|}{ $Z\sim 4B(4,6)-2, \epsilon\sim N(0,1)$ }\\
\hline
\multirow{2}{*}{Scheme} &  \multicolumn{2}{|c|}{$n=50$} & \multicolumn{2}{|c|}{$n=200$} & \multicolumn{2}{|c|}{$n=500$}\\ \cline{2-7}
 & Coverage & Avg Length & Coverage & Avg Length & Coverage & Avg Length \\
\hline
ECDF &  0.80 & 0.54 & 0.80 & 0.11 & 0.81 & 0.04 \\
\hline
Smoothed & 0.96 & 0.46 & 0.94 & 0.11 & 0.95 & 0.47 \\
\hline
FDR & 0.73 & 0.32 & 0.77 & 0.08 & 0.79 & 0.03\\
\hline
$\lceil n^{4/5} \rceil$ & 0.88 & 0.53 & 0.89 & 0.11 & 0.90 & 0.04\\
\hline
$\lceil n^{9/10} \rceil$ & 0.85 & 0.54 & 0.86 & 0.11 & 0.88 & 0.04\\
\hline
$\lceil n^{14/15} \rceil$ & 0.83 & 0.55 & 0.84 & 0.11 & 0.87 & 0.04\\
\hline
\end{tabular}
}

{\scriptsize
\begin{tabular}{|c|c|c|c|c|c|c|}
\hline
\multicolumn{7}{|c|}{$Z\sim 4B(4,6)-2, \epsilon\sim\textrm{Unif}(-1,1)$}\\
\hline
\multirow{2}{*}{Scheme} & \multicolumn{2}{|c|}{$n=50$}& \multicolumn{2}{|c|}{$n=200$} & \multicolumn{2}{|c|}{$n=500$}\\ \cline{2-7}
 & Coverage & Avg Length & Coverage & Avg Length & Coverage & Avg Length \\
\hline
ECDF & 0.80 & 0.40 & 0.80 & 0.08 & 0.81 & 0.03\\
\hline
Smoothed & 0.94 & 0.33 & 0.95 & 0.08 & 0.96 & 0.04\\
\hline
FDR & 0.75 & 0.26 & 0.77 & 0.06 & 0.81 & 0.02\\
\hline
$\lceil n^{4/5} \rceil$ & 0.88 & 0.36 & 0.88 & 0.09 & 0.91 & 0.04\\
\hline
$\lceil n^{9/10} \rceil$ & 0.85 & 0.39 & 0.85 & 0.08 & 0.87 & 0.03\\
\hline
$\lceil n^{14/15} \rceil$ & 0.83 & 0.39 & 0.84 & 0.08 & 0.85 & 0.03\\
\hline
\end{tabular}
}

{\scriptsize
\begin{tabular}{|c|c|c|c|c|c|c|}
\hline
\multicolumn{7}{|c|}{$Z\sim 4B(4,6)-2, \epsilon\sim \Gamma(4,2)-2$}\\
\hline
\multirow{2}{*}{Scheme} & \multicolumn{2}{|c|}{$n=50$} & \multicolumn{2}{|c|}{$n=200$} & \multicolumn{2}{|c|}{$n=500$}\\ \cline{2-7}
 &  Coverage & Avg Length & Coverage & Avg Length & Coverage & Avg Length \\
\hline
ECDF & 0.80 & 0.49 & 0.80 & 0.09 & 0.81 & 0.04 \\
\hline
Smoothed & 0.93 & 0.36 & 0.95 & 0.08 & 0.96 & 0.03 \\
\hline
FDR & 0.76 & 0.30 & 0.77 & 0.06 & 0.80 & 0.02\\
\hline
$\lceil n^{4/5} \rceil$ & 0.87 & 0.43 & 0.88 & 0.10 & 0.91 & 0.03\\
\hline
$\lceil n^{9/10}\rceil$ & 0.85 & 0.46 & 0.84 & 0.09 & 0.88 & 0.03\\
\hline
$\lceil n^{14/15}\rceil$ & 0.83 & 0.48 & 0.85 & 0.09 & 0.85 & 0.03\\
\hline
\end{tabular}
}

We can see from the table that the smoothed bootstrap scheme outperforms all the others in terms of coverage. It must also be noted that this is achieved without a relative increase in the lengths of the intervals. The $m$ out of $n$ bootstrap with $\lceil n^{4/5} \rceil$ also performs reasonably well. It clearly outperforms all other $m$ out of $n$ schemes as well as ECDF and FDR bootstrap procedures (which are inconsistent).


Figure \ref{figura1} shows the histograms of the distribution of $n(\hat{\zeta}_n - \zeta_0)$ (obtained from 1000 random samples) and its bootstrap estimates obtained from the 4 different bootstrap schemes (using 2000 bootstrap samples each) from a single data set of size $n=500$ from model (\ref{ec1}) with $Z\sim 4B(4,6)-2, \epsilon\sim \Gamma(4,2)-2, \alpha_0 = -1, \beta_0 = 1, \zeta_0 = 0$. The histograms clearly show that the smoothed bootstrap (top right panel) provides, by far, the best approximation to both, the actual (top middle panel) and the limiting distributions (top left panel). In fact, the histograms of the distribution of $n(\hat{\zeta}_n - \zeta_0)$  and the corresponding smoothed bootstrap estimate are almost indistinguishable. The $m$ out of $n$ approach, although guaranteed to converge, lacks the efficiency of the smoothed bootstrap. This may be due to the fact that we do not have an optimal way of choosing the tuning parameter $m_n$. The smoothed bootstrap also requires the choice of a tuning parameter, namely, the smoothing bandwidth, but the in our analysis the results were very insensitive to the choice of the bandwidth. This is certainly an advantage for the smoothed bootstrap procedure.
\begin{figure}[h!]
\centering
\includegraphics[height=9cm,width=13cm]{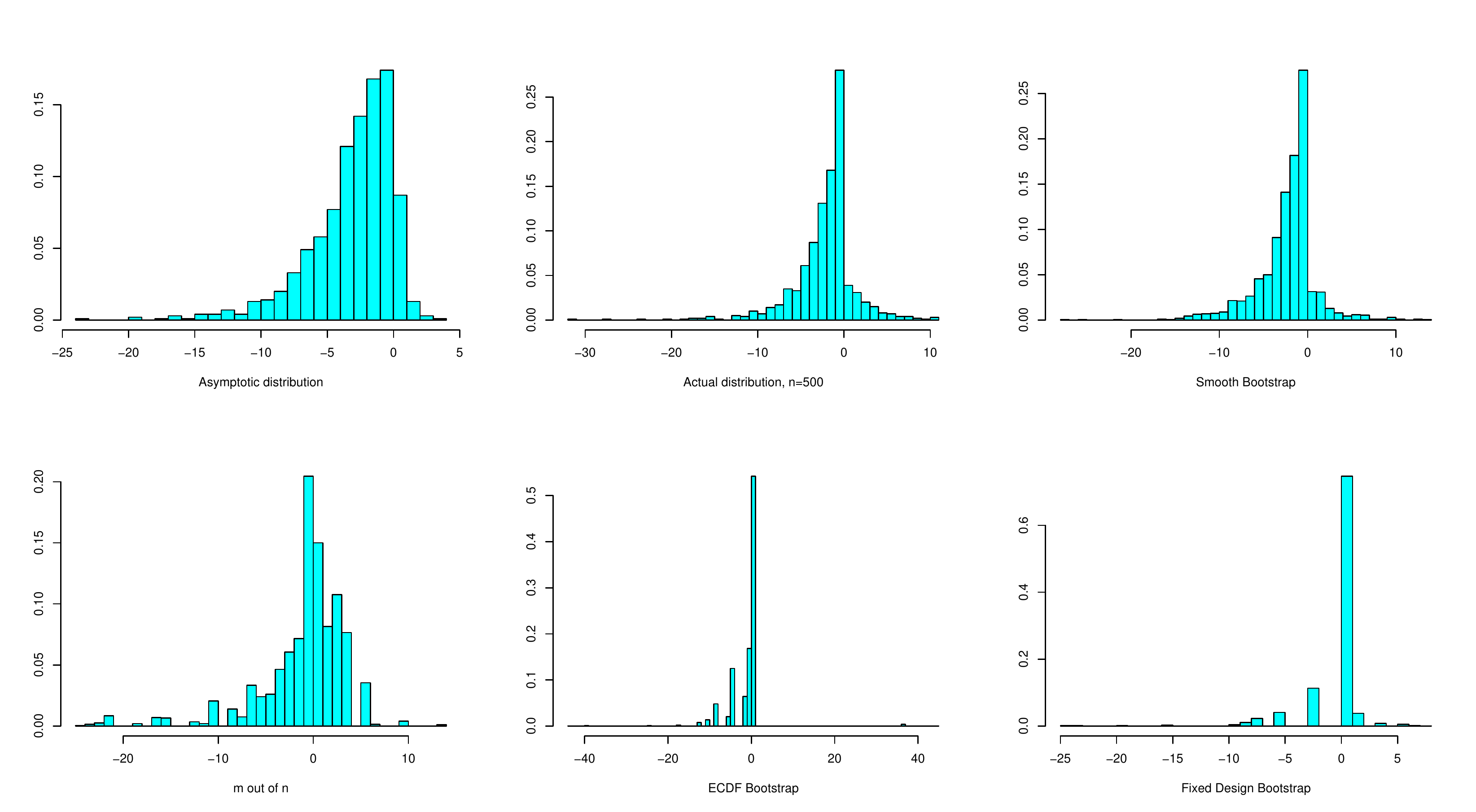}
\caption{Histograms of the distribution of $n(\hat{\zeta}_n - \zeta_0)$ and its bootstrap estimates: the asymptotic distribution of $n(\hat{\zeta}_n - \zeta_0)$ (top left); the actual distribution of $n(\hat{\zeta}_n - \zeta_0)$ (top middle); the distribution of $n(\zeta_n^* - \hat{\zeta}_n)$ for the smoothed (top right), ECDF (bottom middle) and FDR (bottom right) schemes; the distribution of $m_n(\zeta_n^* - \hat{\zeta}_n)$, $m_n=\lceil n^{\frac{4}{5}} \rceil$ (bottom left).}\label{figura1}
\end{figure}
\section{More general change-point regression models}\label{discussion}
In this section we mention some of the broader implications of our analysis of (\ref{ec1}) in the context of more general change-point models in regression. We can consider a model of the form
\begin{equation}\label{ecuaciongeneral}
Y = \psi_{\alpha_0} (W,Z) \mathbf{1}_{Z\leq \zeta_0} + \xi_{\beta_0} (W,Z) \mathbf{1}_{Z> \zeta_0} + \epsilon,
\end{equation}
where $Z$ is a continuous random variable; $W$ is a random vector of covariates; $\alpha_0\in\mathbb{R}^p$ and $\beta_0\in\mathbb{R}^q$ are two unknown Euclidian parameters; $\psi_\alpha (w,z)$ and $\xi_\beta (w,z)$ are known real-valued functions continuous in $(w,z)$ and twice continuously differentiable in $\alpha$ and $\beta$ respectively; $\zeta_0 \in [a,b] \subset \textrm{supp}(Z)\subset \mathbb{R}$ is the change-point; $\epsilon$ is a continuous random variable, independent of $(W,Z)$ with zero expectation and finite variance $\sigma^2 > 0$. We assume that $\psi_{\alpha_0} (W,Z)$ is identifiable from $\xi_{\beta_0} (W,Z)$ and that the least squares problems
\[ \min_{\alpha\in\mathbb{R}^p}\left\{\sum_{Z_j\leq \zeta} (Y_j - \psi_\alpha (W_j,Z_j) )^2\right\} \;\;\;\ \mbox{ and }\;\;\; \min_{\beta\in\mathbb{R}^q}\left\{\sum_{Z_j > \zeta} (Y_j - \xi_\beta (W_j,Z_j) )^2\right\}\] are well-posed for every possible data set $\{(Y_1,Z_1,W_1),\ldots,(Y_n,Z_n,W_n)\}$ and any $\zeta\in \textrm{supp}(Z)^\circ$. We also assume that $\psi_{\alpha_0} (w,\zeta_0) \ne \xi_{\beta_0} (w,\zeta_0)$ for every value of $w$.

Like in the simple case, the method of least squares can be used to compute estimators $\hat{\alpha}_n$, $\hat{\beta}_n$ and $\hat{\zeta}_n$. One simply takes the minimizer $(\hat{\alpha}_n,\hat{\beta}_n,\hat{\zeta}_n)$ of
\[\displaystyle \sum_{j=1}^n \left(Y_j - \psi_\alpha (W_j,Z_j) \mathbf{1}_{Z_j\leq \zeta} + \xi_\beta (W_j,Z_j) \mathbf{1}_{Z_j> \zeta}\right)^2\] with the smallest $\zeta$-component.

Since the simple model (\ref{ec1}) is a particular case of (\ref{ecuaciongeneral}), one can immediately conclude from our analysis that the usual ECDF and residual bootstrap procedures will not be consistent. However, the smoothed bootstrap can be adapted to produce consistent interval estimation. The modified scheme can be described as follows:
\begin{enumerate}
\item Choose some procedure (e.g., kernel density estimation) to build a distribution $\hat{F}_n$ with density $\hat{f}_n$ such that $\hat{f}_n\rightarrow f$ uniformly on some open interval containing $\zeta_0$ w.p. 1, where $f$ is the density of $Z$. Let $\mathbb{P}_n^\epsilon$  and $\mathbb{P}_n^W$ be the empirical measures of the centered residuals (as in the description of Scheme 2 in Section \ref{BootsSchemes}) and $W_1,\ldots,W_n$, respectively.
\item Get i.i.d. replicates $Z_{n,1}^*,\ldots,Z_{n,n}^*$ from $\hat{F}_n$ and sample, independently, \newline $\epsilon_{n,1}^{*},\ldots,\epsilon_{n,n}^{*}\stackrel{i.i.d.}{\sim}\mathbb{P}_{n}^{\epsilon}$ and $W_{n,1}^{*},\ldots,W_{n,n}^{*}\stackrel{i.i.d.}{\sim}\mathbb{P}_{n}^{W}$. We could have also kept the $W_i$'s fixed, i.e., $W_{n,i}^{*} = W_i$.
\item Define $Y_{n,j}^{*} = \psi_{\hat{\alpha}_n}(W_{n,j}^*,Z_{n,j}^*) \mathbf{1}_{Z_{n,j}^{*}\leq \hat{\zeta}_n} + \xi_{\hat{\beta}_n}(W_{n,j}^*,Z_{n,j}^*) \mathbf{1}_{Z_{n,j}^{*}> \hat{\zeta}_n} + \epsilon_{n,j}^{*}$ for all $j=1,\ldots,n$.
\item Compute the bootstrap least squares estimators $(\alpha_n^*,\beta_n^*,\zeta_n^*)$ by taking the minimizer of
\[ \sum_{j=1}^n \left(Y_{n,j}^* - \psi_\alpha (W_{n,j}^*,Z_{n,j}^*) \mathbf{1}_{Z_{n,j}^*\leq \zeta} - \xi_\beta (W_{n,j}^*,Z_{n,j}^*) \mathbf{1}_{Z_{n,j}^*> \zeta}\right)^2 \]
with the smallest $\zeta$-component.
\item Approximate the distribution of $n(\hat{\zeta}_n - \zeta_0)$ with the (conditional) distribution of $n(\zeta_n^* - \hat{\zeta}_n)$.
\end{enumerate}
Although our analysis indicates that this smoothed bootstrap procedure must be consistent, it is difficult to use our methods to prove consistency in such generality. However, the proof of consistency for the simple model (\ref{ec1}) can be adapted to cover the case of parametric additive models, i.e., when $\psi_\alpha(w,z)$ and $\xi_\beta (w,z)$ are of the form
\begin{eqnarray}
\psi_\alpha(w,z) = \sum_{j=1}^p \alpha_j g_j(w,z), \;\;\;\mbox{ and } \;\;\;
\xi_\beta(w,z) = \sum_{k=1}^q \beta_k h_k(w,z),\nonumber
\end{eqnarray}
where $g_j, h_k$, $j=1,\ldots,p$, $k=1,\ldots, q$ are known smooth functions.

\section{Acknowledgements}
We would like to thank Souvik Ghosh for his helpful comments on the proof of Lemma \ref{lsho}.

\appendix
\section{Appendix}\label{appendix}
In this appendix we provide the proofs of most of the results stated in the previous sections. We start with some results that characterize convergence in the space $\mathcal{D}_K$ with metric $d_K$.
\subsection{The space $\mathcal{D}_K$}\label{D_K} Recall that $\mathcal{D}_K$ is the space of all functions on $K \subset \mathbb{R}^3$ that are continuous in the first two co-ordinates and c\`adl\`ag in the third. We keep the notation introduced in Section \ref{pgs}.
\begin{lemma}\label{l1}
Let $K = A \times I \subset \mathbb{R}^3$ be a compact rectangle where $A \subset \mathbb{R}^2$ and $I \subset \mathbb{R}$. Let $W$ a continuous function on $K$ and $(W_n)_{n=1}^{\infty}$ a sequence of elements in $\mathcal{D}_K$ such that $\function{d_K}{W_n,W}\rightarrow 0$. Then, with the notation $\left\|  x \right\|_K = \sup\left\{|\function{x}{h}|: h \in K\right\}$, we have that $\left\|W_n - W\right\|_K\rightarrow 0$.
\end{lemma}
\begin{proof} Let $\epsilon>0$. Since $K$ is compact, $W$ is uniformly continuous on $K$ and therefore there is $\delta>0$ such that $|\function{W}{\theta}-W(\vartheta)|<\frac{\epsilon}{2}$ whenever $|\theta-\vartheta|<\delta$. Also, there is $\rho>0$ such that $\sup\{|s - \lambda(s)|: s\in I\}<\delta$ whenever $\left\|\lambda\right\|<\rho$. It suffices to choose $\rho < \frac{1}{4}\land\frac{\delta}{2L}$ where $L$ is the length of $I$. To see this, assume $I=[a,b]$ and observe that for any $\tau\in(0,\frac{1}{4})$, $\tau < 2\tau - 4\tau^2\leq \log(1+2\tau)$  and for any $\tau>-1$, $\log(1+\tau) \leq \tau$. It follows that for any $s\in I$, $\log(1-2\rho)< -\rho \leq \log\frac{\lambda s -a}{s - a} \leq \rho < 2\rho - 4\rho^2\leq \log(1+2\rho)$ and thus, $|\lambda s -s|<2(s-a)\rho\leq 2L\rho$.

Now, since $\function{d_K}{W_n,W}\rightarrow 0$, there is $N \in \mathbb{N}$ such that for any $n \geq N$ there exists $\lambda_n \in \Lambda_I$ with the property that  $\left\|\lambda_n \right\|<\rho\land\frac{\epsilon}{2}$ and
\begin{equation}
    \sup_{(\alpha,\beta,\zeta)\in K} \left\{ \left|\function{W_n}{\alpha,\beta,\zeta} - \function{W}{\alpha,\beta,\lambda_n\zeta} \right|\right\} <\rho\land\frac{\epsilon}{2}. \label{eq:Wn-W}
\end{equation}
Then, for any $\theta = (\alpha,\beta,\zeta) \in K$ and any $n\geq N$ we have that $|(\alpha, \beta, \lambda_n(\zeta)) - (\alpha,\beta,\zeta)| = |\lambda_n(\zeta) - \zeta|<\delta$ and thus, we can bound $|W_n(\theta) - W(\theta)|$, by
\begin{eqnarray}
& & \left|\function{W_n}{\alpha, \beta, \zeta} - \function{W}{\alpha, \beta, \lambda_n(\zeta)}\right| + \left|\function{W}{\alpha, \beta, \zeta} - \function{W}{\alpha, \beta, \lambda_n(\zeta)}\right|\nonumber\\
& & < \epsilon/2 + \epsilon/2 = \epsilon \nonumber
\end{eqnarray}
using (\ref{eq:Wn-W}) and the uniform continuity of $W$. From this it follows that $\left\|W_n - W\right\|\leq\epsilon$ for any $n\geq N$. \end{proof}

Lemma \ref{l1} shows that as long as the limit is continuous, convergence in the uniform and Skorohod topologies are equivalent. The next result concerns the continuity of the smallest argmax functional.

\begin{lemma}\label{l2}
Let $K\subset\mathbb{R}^3$ be a compact rectangle and $W\in\mathcal{D}_K$ be a continuous function which has a unique maximizer $\theta^* \in K$. Then, the smallest argmax functional is continuous at $W$ (with respect to both, the uniform and Skorohod topologies).
\end{lemma}
\begin{proof}
Let $(W_n)_{n=1}^\infty$ be a sequence converging to $W$ in the Skorohod topology. Let  $\epsilon>0$ be given. Let $G$ be the open ball of radius $\epsilon$ around $\theta^*$ and let $\delta := \left(W(\theta^*) - \sup_{\theta\in K\setminus G}\left\{W(\theta)\right\}\right)/2 > 0$. By Lemma \ref{l1} we have $\left\|W_n - W\right\|_K<\delta$ for all large $n$. If this condition is satisfied, then
\[ W(\theta^*)=2\delta + \sup_{\theta\in K\setminus G}\left\{W(\theta)\right\}> \delta + \sup_{\theta\in K\setminus G}\left\{W_n(\theta)\right\}.\]
But $\left\|W_n - W\right\|_K<\delta$ also implies that $\displaystyle \sup_{\theta\in K}\{W_n(\theta)\} > W(\theta^*)-\delta$. The combination of these two facts shows that if $\left\|W_n - W\right\|_K<\delta$, then any maximizer of $W_n$ must belong to $G$. Thus, $\left|\sargmax_{\vartheta \in K} \{W_n(\vartheta)\} - \theta^*\right|<\epsilon$ for $n$ large enough.
\end{proof}

\subsubsection{A convergence theorem for the smallest argmax functional}\label{app2}
Recall the definitions of $\mathcal{S}$ and $\mathcal{S}_I$, where $I \subset \mathbb{R}$ is any interval containing $0$, which were provided in Section \ref{pgs}. For a compact rectangle $K = I_1 \times I_2 \times I_3 \subset \mathbb{R}^3$ containing the origin, consider the subspace $\mathcal{D}_K^0$ of $\mathcal{D}_K$ consisting of all functions $\psi\in\mathcal{D}_K$ which can be expressed as:
\begin{eqnarray}
    \function{\psi}{h_1,h_2,h_3} = V_0 (h_1,h_2) \ind{a_{-1} \leq h_3 < a_{1}} + \sum_{k=1}^\infty V_{k} (h_1,h_2)\ind{a_{k} \leq h_3 < a_{k+1}} \nonumber \\
    + \sum_{k=1}^\infty  V_{-k}(h_1,h_2)\ind{a_{-k-1} \leq h_3 < a_{-k}} \label{ec58}
\end{eqnarray}
where $\left(\ldots<a_{-k-1}<a_{-k}<\ldots< a_0 = 0<\ldots<a_{k}<a_{k+1}<\ldots\right)_{k\in\mathbb{N}}$ is a sequence of jumps and $\left(V_k\right)_{k\in\mathbb{Z}}$ is a collection of continuous functions.
We write $\mathcal{D}^0$ when $K = \mathbb{R}^3$. Observe that the representation in (\ref{ec58}) is not unique. However, knowledge of the function $\psi$ and of the jumps $(a_k)$ completely determines the continuous functions $(V_k)_{k\in\mathbb{Z}}$. Associate with every $\psi$, expressed as in (\ref{ec58}), a pure jump function $\tilde \psi \in \mathcal{S}$ whose sequence of jumps is exactly the $a_k$'s, i.e.,
\begin{eqnarray}\label{eq:PureJumpProc}
\function{\tilde{\psi}}{t} &=& \sum_{k=1}^\infty \ind{a_{k} \leq t }+
\sum_{k=1}^\infty  \ind{a_{-k} > t}.
\end{eqnarray}
Finally, we denote by sargmax and largmax the smallest and largest argmax functionals, respectively.

The next lemma, which mimics Lemma 3.1 of \cite{lmm}, makes a statement about the continuity of the smallest argmax functional on the space $\mathcal{D}_K^0 \times \mathcal{S}_{I_3}$.

\begin{lemma}\label{l11}
Let $C\in\mathbb{N}$, $K=[-C,C]^3$ and $\left(\psi_n,\tilde{\psi}_n\right)_{n=1}^\infty$, $(\psi_0,\tilde{\psi}_0)$ be functions in $\mathcal{D}_K^0 \times \mathcal{S}_{[-C,C]}$ such that $\psi_n$ satisfies (\ref{ec58}) for the sequence of jumps of $\tilde{\psi}_n$ for any $n\geq 0$. Assume that $(\psi_n,\tilde{\psi}_n)\rightarrow (\psi_0,\tilde{\psi}_0)$ in $\mathcal{D}_K^0 \times \mathcal{S}_{[-C,C]}$ (with the product topology). Suppose, in addition, that $\psi_0$ can be expressed as (\ref{ec58}) for the sequence of jumps $\left(\ldots<a_{-k-1}<a_{-k}< \ldots < a_0 = 0< \ldots< a_{k} \right.$ $\left. < a_{k+1}< \ldots\right)_{k\in\mathbb{N}}$ of $\tilde{\psi}_0$ and some strictly concave functions $(V_j)_{j\in\mathbb{Z}}$ with the property that for any finite subset $A\subset\mathbb{Z}$ there is only one $j\in A$ for which
\begin{equation}\label{ec24}
\max_{m\in A}\left\{\sup_{h\in K}\left\{V_m(h_1,h_2)\right\}\right\} = \sup_{h\in K}\left\{V_j(h_1,h_2)\right\}.
\end{equation}
Then, $\displaystyle \sargmax_{h\in K}\{\psi_n(h)\}$ and $\displaystyle \largmax_{h\in K}\{\psi_n(h)\}$ are well-defined for sufficiently large $n\in\mathbb{N}$ and
\begin{enumerate}[(i)]
\item $\displaystyle \sargmax_{h\in K}\{\psi_n(h)\} \rightarrow \sargmax_{h\in K}\{\psi_0(h)\}$ as $n\rightarrow\infty$
\item $\displaystyle \largmax_{h\in K}\{\psi_n(h)\} \rightarrow \largmax_{h\in K}\{\psi_0(h)\}$  as $n\rightarrow\infty$.
\end{enumerate}
\end{lemma}
\begin{proof} We can write $\psi_n$ in the form (\ref{ec58}) with $\left(\ldots<a_{n,-k-1}<a_{n,-k}<\right.$ $\\$ $\left.\ldots < a_{n,0} = 0 < \ldots\right.$ $\left.< a_{n,k}< a_{n,k+1}< \ldots\right)_{k \in \mathbb{N}}$ being the sequence of jumps of $\psi_n$ and $V_{n,j}$ being the continuous functions. Consequently, $\tilde \psi_n$, the pure jump process associated with $\psi_n$, can be expressed as (\ref{eq:PureJumpProc}) with jumps at $(a_{n,k})_{k \in \mathbb{Z}}$.

Let $N_r$ and $N_l$ be the number of jumps of $\tilde{\psi}_0$ in $[0,C]$ and $[-C,0)$ respectively. Let $\rho > 0$ be sufficiently small such that all the points of the form $a_j \pm \rho$ are continuity points of $\psi_0$, for $-N_l \leq j \leq N_r$. Since convergence in the Skorohod topology of $\tilde{\psi}_n$ to $\tilde{\psi}_0$ implies point-wise convergence for continuity points of $\tilde{\psi}_0$ (see page 121 of \cite{bi}), and all of them are integer-valued functions, we see that $\tilde{\psi}_n (a_j - \rho) = j-1$ and $\tilde{\psi}_n (a_j + \rho) = j$ for any $1\leq j\leq N_r$, and $\tilde{\psi}_n (C) = N_r$ for all sufficiently large $n$. Thus, for all but finitely many $n$'s we have that $\tilde{\psi}_n$ has exactly $N_r$ jumps between 0 and $C$ and that the location of the $j$-th jump to the right of 0 satisfies $|a_{n,j}-a_{j}|<\rho$. The same happens to the left of zero: for all but finitely many $n$'s, $\tilde{\psi}_n$ has exactly $N_l$ jumps in $[-C,0)$ and the sequences of jumps $\left(a_{n,-j}\right)_{n=1}^\infty$, $1\leq j\leq N_l$, converge to the corresponding jumps $a_{-j}$.

Let $ V^* = \sup\left\{V_j (h_1,h_2): h\in K, -N_l\leq j \leq N_r\right\}$. Since all the $V_j$'s are continuous and $K$ is compact, this supremum is actually achieved at some value $(h_1^*,h_2^*)\in [-C,C]^2$. By (\ref{ec24}) and the strict concavity of the $V_j$'s, it is seen that $(h_1^*,h_2^*)$ is unique and that there is a unique ``flat stretch'' at which this supremum is attained. Suppose, without loss of generality, that the maximum value is achieved in an interval of the form $[a_k,a_{k+1}\land C)$ for a unique $k\in\left\{1,\ldots, N_r\right\}$. Now, write $b_0 = 0$; $b_j=\frac{ a_j + C\land a_{j+1}}{2}$ for $1 \leq j \leq N_r$; and $b_j = \frac{a_j + (-C)\lor a_{j-1}}{2}$ for $-N_l \leq j \leq -1$. Note that the $b_j$'s (for any value of the first two variables) are continuity points of both $\psi_0$ and $\tilde{\psi}_0$.

Let $\kappa = \min_{-N_l \leq j \leq N_r+1} (C \land a_j - (-C) \lor a_{j-1})$ be the length of the shortest stretch. Take $0 < \eta, \delta < \kappa/4$. Considering the convergence of the jumps of $\psi_n$ to those of $\psi_0$, there is  $N\in\mathbb{N}$ such that for any $n \geq N$, the following two statements hold:
\begin{enumerate}[(a)]
\item Consider $\rho > 0$ such that if $\left\|\lambda\right\|<\rho$, then $$\sup\left\{ |s - \lambda(s)| : s\in [-C,C]\right\}<\delta,$$ just as in the proof of Lemma \ref{l1}. By the convergence of $\psi_n$ to $\psi_0$ in the Skorohod topology, there exists $\lambda_n \in \Lambda_{[-C,C]}$ such that $\left\|\lambda_n\right\| < \rho$ and \[\displaystyle \sup_{h\in K} \left\{ |\psi_n(h_1,h_2,\lambda_n(h_3)) - \psi_0(h_1,h_2,h_3)|\right\} < \eta.\]
\item For any $1\leq j\leq N_r$ (respectively $j=0$, $-N_l\leq j \leq -1$), $b_j$ lies somewhere inside the interval $\left(a_{n,j} + \delta, C\land a_{n,j+1} - \delta\right)$ (respectively $\left(a_{n,-1} + \delta, \right.$ $\left. a_{n,1} - \delta\right)$, $\left((-C)\lor a_{n,j-1} + \delta, a_{n,j} - \delta \right)$). This follows from what was proven in the first two paragraphs of this proof.
\end{enumerate}
From (a) we see that $|\lambda_n (b_j) - b_j|< \delta$ for all $-N_l \leq j \leq N_r$. But (b) and the size of $\delta$ in turn imply that $b_j$ and $\lambda_n(b_j)$ belong to the same ``flat stretch'' of $\psi_n$ and thus $\psi_n(h_1,h_2,\lambda_n (b_j)) = \psi_n(h_1,h_2,b_j) = V_{n,j}(h_1,h_2)$ for all $(h_1,h_2)\in [-C,C]^2$ and all $-N_l \leq j \leq N_r$. Considering again (b) and the second inequality in (a), we conclude that $\left\|V_{n,j} - V_j\right\|_{[-C,C]^2} < \eta$ for all $-N_l \leq j \leq N_r$ and all $n\geq N$. Hence, all the sequences $(V_{n,j})_{n=1}^\infty$ converge uniformly in $[-C,C]^2$ to their corresponding $V_j$. Consequently:
\begin{eqnarray}
\max_{\substack{-N_l \leq j \leq N_r \\ j\neq k}}\left\{\sup_{h_1,h_2\in [-C,C]} V_{n,j}(h_1,h_2)\right\}
&\longrightarrow&  \max_{\substack{-N_l \leq j \leq N_r \\ j\neq k}}\left\{\sup_{h_1,h_2\in [-C,C]} V_{j}(h_1,h_2)\right\}, \nonumber\\
\max_{h_1,h_2\in [-C,C]}\left\{ V_{n,k} (h_1,h_2) \right\}&\longrightarrow& \max_{h_1,h_2\in [-C,C]}\left\{ V_{k}(h_1,h_2)\right\}  = V_k(h_1^*,h_2^*), \nonumber\\
\argmax_{h_1,h_2\in [-C,C]}\left\{ V_{n,k} (h_1,h_2) \right\}&\longrightarrow& \argmax_{h_1,h_2\in [-C,C]}\left\{ V_{k}(h_1,h_2)\right\}  = (h_1^*,h_2^*), \nonumber\\
 \lsup \max_{\substack{-N_l \leq j \leq N_r \\ j\neq k}}\left\{\sup_{h_1,h_2\in [-C,C]} V_{n,j}(h_1,h_2)\right\} & < & \linf \max_{h_1,h_2\in [-C,C]}\left\{ V_{n,k}(h_1,h_2)\right\}. \nonumber
\end{eqnarray}
The above, together with (\ref{ec24}) and the fact that $a_{n,k}\rightarrow a_k$ and $a_{n,k+1}\rightarrow a_{k+1}$, imply that
\begin{itemize}
\item[] $\displaystyle \sargmax_{h\in K}\{\psi_n(h)\}\rightarrow (h_1^*,h_2^*,a_k) = \sargmax_{h\in K}\{\psi_0(h)\}$
\item[] $\displaystyle \largmax_{h\in K}\{\psi_n(h)\}\rightarrow (h_1^*,h_2^*,a_{k+1}) = \largmax_{h\in K}\{\psi_0(h)\}$
\end{itemize}
as $n\rightarrow\infty$.
\end{proof}

\subsection{Some useful lemmas and proofs}
We first give an account of a series of technical lemmas which will aid us in the proof of Propositions \ref{pg1}, \ref{pg2} and \ref{pg3}.

\begin{lemma}\label{lKP}
Let $\alpha \neq \beta\in\mathbb{R}$. Consider the class of functions from $\mathbb{R}^2$ to $\mathbb{R}$ given by \begin{eqnarray*}
{\footnotesize \mathscr{A}=\left\{\phi(y,z) := (y-\alpha\ind{(-\infty,\zeta]}(z) - \beta\ind{(\zeta,\infty]}(z))\ind{I}(z) | \zeta \in \mathbb{R}, I\subset\mathbb{R} \textrm{ is an interval}\right\}.}  \
\end{eqnarray*}
Then, $\mathscr{A}$ is a VC-subgraph class with envelope $|y| + |\alpha|+|\beta|$. There is an upper bound for the VC-index of $\mathscr{A}$ that is independent of $\alpha$ and $\beta$. Moreover, there is a continuous, increasing function $J_\mathscr{A}$ with $J_\mathscr{A}(1)<\infty$, which is also independent of $\alpha$ and $\beta$, and satisfies the following property: If $\mathscr{D}\subset\mathscr{A}$ is a subclass with envelope $B$ and $W_1,\ldots,W_n$ is a random sample, defined on some probability space $(\Omega,\mathcal{A},\mathbf{P})$, from a distribution $\mu$ for which $\mu(B^2)<\infty$ and $\mu_n$ is the empirical measure defined by the sample, then
\begin{equation}\nonumber
\int \sup_{\varphi\in\mathcal{D}}\left\{\left|(\mu_n -\mu)(\varphi)\right|\right\}d\mathbf{P} \leq \frac{J_{\mathscr{A}}(1)}{\sqrt{n}}\sqrt{\mu(B^2)}.
\end{equation}
\end{lemma}
\begin{proof}
We use the same notation as in Lemmas 2.6.17 and 2.6.18, page 147 of \cite{vw}. Consider the classes of functions $\mathscr{H} = \{y-\alpha\ind{(-\infty,\zeta]}(z) - \beta\ind{(\zeta,\infty]}(z): \zeta\in \mathbb{R}\}$ and $\mathscr{K} = \left\{\ind{(-\infty,\zeta]}(z): \zeta \in\mathbb{R} \right\}$. Then, $\mathscr{K}$ is a VC class with VC-index 2. It follows that $\mathscr{H} = (\beta-\alpha)\cdot\mathscr{K} + (y-\beta)$ is also VC. Recall that $\mathcal{F} = \{\mathbf{1}_I(z): I \subset \mathbb{R} \mbox{ is an interval} \}$. Letting $[\varphi > t] := \{(y,z,t): \varphi(y,z) > t\}$ for $\varphi \in \mathscr{A}$, we see that
\begin{eqnarray*}
    & & \left\{ [\varphi > t]: \varphi \in \mathscr{A} \right\} =
    \bigg( \mathbb{R} \times \{\mathcal{F} \leq 0\} \times (-\infty,0) \bigg)
    \bigsqcup \\
    & & \bigg(\left\{[\psi > t]: \psi \in \mathscr{H}\right\} \sqcap \left(\mathbb{R} \times \left\{\mathcal{F}>0 \right\} \times \mathbb{R} \right) \bigg)
\end{eqnarray*}
from which it follows that $\mathscr{A}$ is VC. Moreover, the VC-indexes of $\mathscr{K}$ and $\mathcal{F}$ are two and three for any choice of $\alpha$ and $\beta$. Hence, the corresponding VC-indexes of $\mathscr{H}$ and $\mathscr{A}$ both have upper bounds independent of $\alpha$ and $\beta$. The existence of the function $J_{\mathscr{A}}$ is a consequence of the maximal inequality 3.1 in \cite{kipo}. Note that $J_{\mathscr{A}}$ only depends on the VC-index of the class $\mathscr{A}$, which in turn has an upper bound independent of $\alpha$ and $\beta$.
\end{proof}

\begin{lemma}\label{l12}
Suppose that (I)-(IV) hold. Then,
\begin{enumerate}[(i)]
\item $\left\|\mathbb{Q}_n (\tilde{\epsilon}_n^2\ind{Z\leq (\cdot)\land\zeta_n})- \sigma^2\mathbb{P}(Z\leq(\cdot)\land\zeta_0)\right\|_{[a,b]}\rightarrow 0$,
\item $\left\|\mathbb{Q}_n (|\tilde{\epsilon}_n|\ind{Z\leq (\cdot)\land\zeta_n})- \mathbb{P}(|\epsilon|)\mathbb{P}(Z\leq(\cdot)\land\zeta_0)\right\|_{[a,b]}\rightarrow 0$,
\item $\left\|\mathbb{Q}_n (\tilde{\epsilon}_n\ind{Z\leq (\cdot)\land\zeta_n})\right\|_{[a,b]}\rightarrow 0$, and
\item $\left\|\mathbb{Q}_n (\ind{Z\leq (\cdot)\land\zeta_n})- \mathbb{P}(\ind{Z\leq(\cdot)\land\zeta_0})\right\|_{[a,b]}\rightarrow 0$.
\end{enumerate}
Also, these statements are true if $\ind{Z\leq (\cdot)\land\zeta_n}$ is replaced by any of $\ind{(\cdot) < Z \leq\zeta_n}$, $\ind{\zeta_n < Z\leq (\cdot)}$ or $\ind{Z> (\cdot)\lor\zeta_n}$.
\end{lemma}
\begin{proof} Since $\zeta_n\rightarrow \zeta_0$ and $Z$ is continuous, for any $\zeta \in [a,b]$, we obtain
\begin{eqnarray*}
    \left|\mathbb{P}(Y^2\ind{Z\leq \zeta\land\zeta_n}) - \mathbb{P}(Y^2\ind{Z\leq \zeta\land\zeta_0})\right| & \leq & \mathbb{P}(Y^2|\ind{Z\leq\zeta_n}-\ind{Z\leq\zeta_0}|) \rightarrow 0, \\
    \left|\mathbb{P}(|Y-\alpha_0|\ind{Z\leq \zeta\land\zeta_n})-\mathbb{P}(|Y-\alpha_0|\ind{Z\leq \zeta\land\zeta_0})\right| & \leq & \mathbb{P}(|Y||\ind{Z\leq\zeta_n}-\ind{Z\leq\zeta_0}|) \rightarrow 0, \\
    \left|\mathbb{P}(Y\ind{Z\leq \zeta\land\zeta_n})-\mathbb{P}(Y\ind{Z\leq \zeta\land\zeta_0})\right| & \leq & \mathbb{P}(|Y||\ind{Z\leq\zeta_n}-\ind{Z\leq\zeta_0}|) \rightarrow 0, \\
    \left|\mathbb{P}(\ind{Z\leq \zeta\land\zeta_n})-\mathbb{P}(\ind{Z\leq \zeta\land\zeta_0})\right| & \leq & \mathbb{P}(|\ind{Z\leq\zeta_n}-\ind{Z\leq\zeta_0}|) \rightarrow 0.
\end{eqnarray*}
Also note that the convergence is uniform in $\zeta \in [a,b]$. Thus,
\begin{eqnarray*}
    \left\|\mathbb{Q}_n(Y^2\ind{Z\leq (\cdot)\land\zeta_n})-\mathbb{P}(Y^2\ind{Z\leq (\cdot)\land\zeta_0})\right\|_{[a,b]}\leq \left\|\mathbb{Q}_n - \mathbb{P}\right\|_{\mathcal{H}} \\
    \qquad \qquad + \left\|\mathbb{P}(Y^2\ind{Z\leq (\cdot)\land\zeta_n})-\mathbb{P}(Y^2\ind{Z\leq (\cdot) \land \zeta_0})\right\|_{[a,b]} \rightarrow 0
\end{eqnarray*}
as $n \rightarrow \infty$ by (III). Similarly, we also obtain that $\|\mathbb{Q}_n(|Y-\alpha_0|\ind{Z\leq (\cdot)\land\zeta_n})-\mathbb{P}(|Y-\alpha_0|\ind{Z\leq (\cdot)\land\zeta_0})\|_{[a,b]}\rightarrow 0$, $\|\mathbb{Q}_n(Y\ind{Z\leq (\cdot)\land\zeta_n})-\mathbb{P}(Y\ind{Z\leq (\cdot)\land\zeta_0})\|_{[a,b]}\rightarrow 0$ and
$\|\mathbb{Q}_n(\ind{Z\leq (\cdot)\land\zeta_n})-\mathbb{P}(\ind{Z\leq (\cdot)\land\zeta_0})\|_{[a,b]}\rightarrow 0$. This proves $(iv)$.

Finally, $(i)$, $(ii)$ and $(iii)$ follow as consequence of the convergence $\alpha_n \rightarrow \alpha_0$ and of the following inequalities:
\begin{eqnarray*}
    & & \left\|\mathbb{Q}_n (\tilde{\epsilon}_n^2\ind{Z\leq (\cdot)\land\zeta_n})- \sigma^2\mathbb{P}(Z\leq(\cdot)\land\zeta_0)\right\|_{[a,b]} \\
    & \leq & \left\|\mathbb{Q}_n(Y^2\ind{Z\leq (\cdot)\land\zeta_n})-\mathbb{P}(Y^2\ind{Z\leq (\cdot)\land\zeta_0})\right\|_{[a,b]} + 2|\alpha_n - \alpha_0|\mathbb{Q}_n(|Y|) + |\alpha_n^2 - \alpha_0^2| \\
    & + & 2|\alpha_0|\left\|\mathbb{Q}_n(Y\ind{Z\leq (\cdot)\land\zeta_n})-\mathbb{P}(Y\ind{Z\leq (\cdot)\land\zeta_0})\right\|_{[a,b]}+ \alpha_0^2 \left\|\mathbb{Q}_n(\ind{Z\leq (\cdot)\land\hat{\zeta}_n})-\mathbb{P}(\ind{Z\leq (\cdot)\land\zeta_0})\right\|_{[a,b]}
\end{eqnarray*}
and
\begin{eqnarray*}
    \left\|\mathbb{Q}_n (|\tilde{\epsilon}_n|\ind{Z\leq (\cdot)\land\zeta_n}) -\mathbb{P} (|\epsilon|\ind{Z\leq (\cdot)\land\zeta_n})\right\|_{[a,b]} &\leq& \\
    \left\|\mathbb{Q}_n(|Y-\alpha_0|\ind{Z\leq (\cdot)\land\zeta_n})-\mathbb{P}(|Y-\alpha_0|\ind{Z\leq (\cdot)\land\zeta_0})\right\|_{[a,b]}  &+& |\alpha_n-\alpha_0|.
\end{eqnarray*}
and
\begin{eqnarray*}
    \left\|\mathbb{Q}_n (\tilde{\epsilon}_n\ind{Z\leq (\cdot)\land\zeta_n})\right\|_{[a,b]} \leq
    \left\|\mathbb{Q}_n(Y\ind{Z\leq (\cdot)\land\zeta_n})-\mathbb{P}(Y\ind{Z\leq (\cdot)\land\zeta_0})\right\|_{[a,b]} \\
    + |\alpha_n-\alpha_0| + |\alpha_0|\left\|\mathbb{Q}_n(\ind{Z\leq (\cdot)\land\zeta_n})-\mathbb{P}(\ind{Z\leq (\cdot)\land\zeta_0})\right\|_{[a,b]}.
\end{eqnarray*}
The other three cases follow from similar arguments.
\end{proof}

\begin{lemma}\label{l13}
Suppose that (I)-(IV) hold. Then,
\begin{enumerate}[(i)]
\item $\left\|(\mathbb{P}_n^* - \mathbb{Q}_n)(\tilde{\epsilon}_n \ind{Z\leq (\cdot)\land\zeta_n})\right\|_{[a,b]}\cip 0 $,
\item $\left\|(\mathbb{P}_n^* - \mathbb{Q}_n)(\ind{Z\leq (\cdot)\land\zeta_n})\right\|_{[a,b]} \cip 0 $.
\end{enumerate}
Also, these statements are true if $\ind{Z\leq (\cdot)\land\zeta_n}$ is replaced by any of $\ind{(\cdot) < Z \leq\zeta_n}$, $\ind{\zeta_n < Z\leq (\cdot)}$ or $\ind{Z> (\cdot)\lor\zeta_n}$.
\end{lemma}
\begin{proof} By the maximal inequality 3.1 from \cite{kipo} and Lemma \ref{lKP} we see that:
\begin{eqnarray}
\e{\left\|(\mathbb{P}_n^* - \mathbb{Q}_n)(\tilde{\epsilon}_n \ind{Z\leq (\cdot)\land\zeta_n})\right\|_{[a,b]}} &\leq& \frac{J_\mathscr{A}(1)}{\sqrt{m_n}}\sqrt{\mathbb{Q}_n (\tilde{\epsilon}_n^2)} \nonumber\\
\e{\left\|(\mathbb{P}_n^* - \mathbb{Q}_n)(\ind{Z\leq (\cdot)\land\zeta_n})\right\|_{[a,b]}} &\leq& \frac{J_\mathcal{F}(1)}{\sqrt{m_n}}. \nonumber
\end{eqnarray}
The lemma now follow directly as $\mathbb{Q}_n (\tilde{\epsilon}_n^2) \rightarrow \sigma^2$ (a consequence of Lemma \ref{l12}). The other statements are proven similarly.
\end{proof}

\subsubsection{Proof of Proposition \ref{pg1}}\label{prueba1}
Noting that $\tilde \epsilon_n = Y - \alpha_n \ind{Z\le \zeta_n} - \beta_n \ind{Z > \zeta_n}$, we write
\begin{eqnarray}\label{eq:simp_m_theta}
    m_\theta(X) & = & -(\tilde \epsilon_n + \alpha_n - \alpha)^2 \ind{Z \le \zeta_n \land \zeta} - (\tilde \epsilon_n + \beta_n - \alpha)^2 \ind{\zeta_n < Z \le \zeta} \nonumber \\
    & & \qquad - (\tilde \epsilon_n + \alpha_n - \beta)^2 \ind{\zeta < Z \le \zeta_n} - (\tilde \epsilon_n + \beta_n - \beta)^2 \ind{Z > \zeta_n \lor \zeta},
\end{eqnarray}
and therefore
\begin{eqnarray}
  -\mathbb{P}_n^*(\tilde{\epsilon}_n^2) & = & M_n^* (\theta_n)  \qquad \le M_n^* (\theta_n^*) \nonumber\\
  & \leq & - \mathbb{P}_n^*[(\tilde{\epsilon}_n -\alpha_n^* + \alpha_n)^2\ind{Z<a}] - \mathbb{P}_n^*[(\tilde{\epsilon}_n -\beta_n^* + \beta_n)^2\ind{Z>b}]. \nonumber
\end{eqnarray}
Letting $\gamma^*_n = (\alpha^*_n, \beta^*_n)$, noticing that $ M_n^* (\hat{\theta}_n) = -\function{\mathbb{P}_n^*}{\tilde{\epsilon}_n^2}$, and by rearranging the terms in the above inequality, we get
\begin{eqnarray}
|\gamma_n^*-\gamma_n|^2 \mathbb{P}_n^* (Z<a)\land\mathbb{P}_n^* (Z>b) \leq \function{\mathbb{P}_n^*}{\tilde{\epsilon}_n^2 \ind{a\leq Z\leq b}} \nonumber \\
+ 2|\gamma_n^*-\gamma_n|\left(|\function{\mathbb{P}_n^*}{\tilde{\epsilon}_n \ind{Z<a}}|+ |\function{\mathbb{P}_n^*}{\tilde{\epsilon}_n\ind{Z>b}}|\right).\nonumber
\end{eqnarray}
Consider $\mathbb{P}_n^* (Z<a)$. By $(ii)$ of Lemma \ref{l13} we see that $|(\mathbb{P}_n^* - \mathbb{Q}_n)(Z<a)| \stackrel{\mathbf{P}}{\rightarrow} 0$ and by $(iv)$ of Lemma \ref{l12} we can show that $|(\mathbb{Q}_n - \mathbb{P})(Z<a)| \rightarrow 0$. Thus, combining the two, we have $\mathbb{P}_n^*(Z<a) \cip \mathbb{P} (Z<a)$. Similarly, we can show that $\mathbb{P}_n^* (Z<a)\land\mathbb{P}_n^*(Z>b)\cip\mathbb{P} (Z<a)\land\mathbb{P} (Z>b)>0$ and also that $|\function{\mathbb{P}_n^*}{\tilde{\epsilon}_n\ind{Z<a}}| + |\function{\mathbb{P}_n^*}{\tilde{\epsilon}_n\ind{Z>b}}|\cip 0 $. Also, observe that $\e{\mathbb{P}_n^*(\tilde{\epsilon}_n^2)}=\mathbb{Q}_n(\tilde{\epsilon}_n^2)\rightarrow\sigma^2$, by assumptions (I)-(III) and so $\mathbb{P}_n^*(\tilde{\epsilon}_n^2)$ is bounded in $\mathbb{L}^1$. Hence, we can write \[ |\gamma_n^*-\gamma_n|^2\leq O_{\mathbf{P}}(1) + |\gamma_n^*-\gamma_n|o_{\mathbf{P}}(1)\]
and therefore $|\gamma_n^*-\gamma_n| = O_{\mathbf{P}}(1)$ (and, consequently, $|\gamma_n^*-\gamma_0| = O_{\mathbf{P}}(1)$).

We first rewrite $m_\theta(X)$ as follows:
\begin{eqnarray}
m_\theta(X) = -\tilde{\epsilon}_n^2 -2(\alpha_n-\alpha)\tilde{\epsilon}_n\mathbf{1}_{Z\leq \zeta\land\zeta_n} - (\alpha_n-\alpha)^2\mathbf{1}_{Z\leq \zeta\land\zeta_n} \nonumber\\
- 2(\beta_n-\alpha)\tilde{\epsilon}_n\mathbf{1}_{\zeta_n < Z\leq \zeta} - (\beta_n-\alpha)^2\mathbf{1}_{\zeta_n < Z\leq \zeta} \nonumber\\
-2(\alpha_n-\beta)\tilde{\epsilon}_n\mathbf{1}_{\zeta < Z\leq \zeta_n} - (\alpha_n-\beta)^2\mathbf{1}_{\zeta < Z\leq \zeta_n} \nonumber\\
- 2(\beta_n-\beta)\tilde{\epsilon}_n\mathbf{1}_{Z> \zeta\lor\zeta_n} - (\beta_n-\beta)^2\mathbf{1}_{Z> \zeta\lor\zeta_n}. \label{ec3}
\end{eqnarray}
We can then decompose $M_n^*$ as in (\ref{ec3}), and use Lemmas \ref{l13} and \ref{l12} and the fact that $\theta_n \rightarrow \theta_0$, to obtain \begin{eqnarray*}
    \left\|M_n^* + \mathbb{P}_n^*(\tilde{\epsilon}_n^2) - M_n - \mathbb{Q}_n(\tilde{\epsilon}_n^2)\right\|_K & \cip & 0. \\
    \left\|M_n^* + \mathbb{P}_n^*(\tilde{\epsilon}_n^2) - M - \sigma^2\right\|_K & \cip & 0
\end{eqnarray*}
for every compact $K \subset \Theta$. But $\theta_0$ is also the unique maximizer of $M + \sigma^2$ and $|\gamma_n^*-\gamma_0| = O_{\mathbf{P}}(1)$. Therefore, the conditions of Corollary 3.2.3 (ii), page 287 of \cite{vw}, hold and we obtain that $\theta_n^* \cip \theta_0$ (and also that $\theta_n^*-\theta_n \cip 0$). $\hfill \square$

\subsubsection{Proof of Proposition \ref{pg2}}\label{prueba2}
We will apply Theorem 3.4.1 of \cite{vw} to prove the result. Let $d:\mathbb{R}^3\times\mathbb{R}^3\rightarrow\mathbb{R}$ be given by $d(\theta,\vartheta)=|(\theta_1,\theta_2)-(\vartheta_1,\vartheta_2)| + \sqrt{|\theta_3-\vartheta_3|}$. Consider $\eta,\rho,L>0$ as in (V) and a compact rectangle $K \subset \Theta$ such that $\left\{\theta\in\Theta :\right.$ $\left. d(\theta,\theta_n)<\eta \textrm{ for some } n \in \mathbb{N}\right\}\subset K$. We can take $L$ large enough so $L > 1\lor \sup\left\{ |\theta_1 - \vartheta_2| \lor |\theta_2 - \vartheta_1 | : \theta, \vartheta\in K\right\}$. Pick $n$ large enough so we can fix some $\delta\in(\frac{2\sqrt{2}}{m_n^{1/4}},\eta)$. Then, taking also (I)-(IV) into account and possibly making $\eta$ smaller, we can find positive constants $c_1,c_2>0$ and $N\in\mathbb{N}$ such that for any $n\geq N$, we have (\ref{rccs1}), (\ref{rccs2}), (\ref{rccs3}) and the inequalities:
\begin{eqnarray}\nonumber
\inf_{d(\theta,\theta_n)<\delta}\left\{|\alpha_n - \beta|^2 \land |\beta_n - \alpha|^2 \right\}&>& c_1,\nonumber\\
\mathbb{Q}_n(Z\leq a)\land \mathbb{Q}_n(Z>b) &>& c_2.\nonumber
\end{eqnarray}
Also, let $\mathbb{M}_n(\theta) := M_n^*(\theta) + \mathbb{P}_n^*(\tilde{\epsilon}_n^2)$ and $\mathcal{M}_n(\theta) := M_n(\theta) + \mathbb{Q}_n(\tilde{\epsilon}_n^2)$ for all $\theta \in \Theta$.

Choose $n\geq N$ and $\theta\in \Theta$ with $\frac{\delta}{2} < d(\theta,\theta_n) < \delta$. Then, considering the properties of the constants just defined and the expression
\begin{eqnarray}\label{eq:M_n-M_n_theta}
\mathcal{M}_n(\theta) - \mathcal{M}_n(\theta_n) =
-2(\alpha_n-\alpha)\mathbb{Q}_n(\tilde{\epsilon}_n\mathbf{1}_{Z\leq \zeta\land\zeta_n}) - (\alpha_n-\alpha)^2\mathbb{Q}_n(\mathbf{1}_{Z\leq \zeta\land\zeta_n}) \nonumber\\
- \ 2(\beta_n-\alpha)\mathbb{Q}_n(\tilde{\epsilon}_n\mathbf{1}_{\zeta_n < Z\leq \zeta}) - (\beta_n-\alpha)^2\mathbb{Q}_n(\mathbf{1}_{\zeta_n < Z\leq \zeta}) \nonumber\\
- \ 2(\alpha_n-\beta)\mathbb{Q}_n(\tilde{\epsilon}_n\mathbf{1}_{\zeta < Z\leq \zeta_n}) - (\alpha_n-\beta)^2\mathbb{Q}_n(\mathbf{1}_{\zeta < Z\leq \zeta_n}) \nonumber\\
- \ 2(\beta_n-\beta)\mathbb{Q}_n(\tilde{\epsilon}_n\mathbf{1}_{Z> \zeta\lor\zeta_n}) - (\beta_n-\beta)^2\mathbb{Q}_n(\mathbf{1}_{Z> \zeta\lor\zeta_n})
\end{eqnarray}
it is seen that the sum of the 1st, 3rd, 5th, and 7th terms in (\ref{eq:M_n-M_n_theta}) can be bounded from above by $\frac{8L^2\delta}{\sqrt{m_n}}$. While we also have,
\begin{eqnarray*}
    (\alpha_n-\alpha)^2 \mathbb{Q}_n(\mathbf{1}_{Z\leq \zeta\land\zeta_n}) & \ge & c_2 (\alpha_n - \alpha)^2, \\
    (\beta_n-\beta)^2 \mathbb{Q}_n(\mathbf{1}_{Z> \zeta \lor \zeta_n}) & \ge & c_2 (\beta_n-\beta)^2, \\
    (\beta_n - \alpha)^2 \mathbb{Q}_n(\mathbf{1}_{\zeta_n < Z\leq \zeta}) & \ge & c_1\rho|\zeta-\zeta_n|,\ \textrm{ if } |\zeta-\zeta_n|\geq \frac{\delta^2}{8} > \frac{1}{\sqrt{m_n}},\\
    (\alpha_n-\beta)^2\mathbb{Q}_n(\mathbf{1}_{\zeta < Z\leq \zeta_n}) & \ge & c_1\rho|\zeta-\zeta_n|, \ \textrm{ if } |\zeta-\zeta_n|\geq \frac{\delta^2}{8} > \frac{1}{\sqrt{m_n}},
\end{eqnarray*}
and therefore, noting that either $(\alpha-\alpha_n)^2 + (\beta-\beta_n)^2 \geq \frac{\delta^2}{8}$ or $|\zeta-\zeta_n|\geq \frac{\delta^2}{8}$, letting $c = \frac{1}{16} c_2\land (c_1\rho)$ and adding all the terms in the previous display, we get
\begin{equation}\nonumber
\sup_{\frac{\delta}{2}< d(\theta,\theta_n) < \delta} \left\{\mathcal{M}_n(\theta) - \mathcal{M}_n({\theta}_n)\right\} \leq \frac{8L^2}{\sqrt{m_n}}\delta - 2c\delta^2\ \ \forall n\geq N.
\end{equation}
Hence, setting $\delta_n = \frac{8L^2}{c\sqrt{m_n}}\land \frac{2\sqrt{2}}{m_n^{1/4}}$ we get that
\begin{equation}\label{ec26}
\sup_{\frac{\delta}{2}< d(\theta,{\theta}_n) < \delta} \left\{\mathcal{M}_n(\theta) - \mathcal{M}_n({\theta}_n)\right\} \leq -c\delta^2\ \ \forall\ \delta_n\leq \delta < \eta,\ \ \forall n\geq N.
\end{equation}
Next we will show
\begin{equation}\label{eq:RateConvCond}
\sqrt{n}\e{\sup_{d(\theta,\theta_n)<\delta}\left\{\left|(\mathbb{M}_n - \mathcal{M}_n)(\theta) - (\mathbb{M}_n - \mathcal{M}_n)({\theta}_n)\right|\right\}}\lesssim \frac{\sqrt{n}}{\sqrt{m_n}}\delta.
\end{equation}
Note that, using the expansion (\ref{ec3}), $\mathbb{M}_n(\theta_n) = \mathcal{M}_n (\theta_n) = 0$. To control the term $(\mathbb{M}_n - \mathcal{M}_n)(\theta)$ observe that it admits a very similar expansion as (\ref{eq:M_n-M_n_theta}) with the $\mathbb{Q}_n$ replaced by $(\mathbb{P}_n^* - \mathbb{Q}_n)$; in particular, we can write the difference $\mathbb{M}_n(\theta) - \mathcal{M}_n(\theta)$ (by re-arranging the terms) as
\begin{eqnarray}
-2(\alpha_n-\alpha)(\mathbb{P}_n^* - \mathbb{Q}_n)(\tilde{\epsilon}_n\mathbf{1}_{Z\leq \zeta\land\zeta_n})
-2(\beta_n-\beta)(\mathbb{P}_n^* - \mathbb{Q}_n)(\tilde{\epsilon}_n\mathbf{1}_{Z> \zeta\lor\zeta_n}) \nonumber \\
-2(\beta_n-\alpha)(\mathbb{P}_n^* - \mathbb{Q}_n)(\tilde{\epsilon}_n\mathbf{1}_{\zeta_n < Z\leq \zeta})
-2(\alpha_n-\beta)(\mathbb{P}_n^* - \mathbb{Q}_n)(\tilde{\epsilon}_n\mathbf{1}_{\zeta < Z\leq \zeta_n})  \nonumber \\
- (\alpha_n-\alpha)^2(\mathbb{P}_n^* - \mathbb{Q}_n)(\mathbf{1}_{Z\leq \zeta\land\zeta_n})
- (\beta_n-\beta)^2(\mathbb{P}_n^* - \mathbb{Q}_n)(\mathbf{1}_{Z> \zeta\lor\zeta_n}) \nonumber \\
- (\alpha_n-\beta)^2(\mathbb{P}_n^* - \mathbb{Q}_n)(\mathbf{1}_{\zeta < Z\leq \zeta_n})
- (\beta_n-\alpha)^2(\mathbb{P}_n^* - \mathbb{Q}_n)(\mathbf{1}_{\zeta_n < Z\leq \zeta}). \label{ec11}
\end{eqnarray}
Each of these terms can be controlled by using Lemma \ref{lKP} as
\begin{eqnarray*}
\e{\left\|(\mathbb{P}_n^* - \mathbb{Q}_n)(\tilde{\epsilon}_n \ind{Z\leq (\cdot)\land \zeta_n})\right\|_{[a,b]}} & \leq & \frac{J_\mathscr{A}(1)}{\sqrt{m_n}}\sqrt{\mathbb{Q}_n(\tilde{\epsilon}_n^2)} \\
\e{\left\|(\mathbb{P}_n^* - \mathbb{Q}_n)(\tilde{\epsilon}_n \ind{(\cdot)< Z \leq \zeta_n})\right\|_{|\zeta-\zeta_n|<\delta^2}} & \leq & \frac{J_\mathscr{A}(1)}{\sqrt{m_n}}\sqrt{\mathbb{Q}_n(\tilde{\epsilon}_n^2\ind{\zeta_n - \delta^2 < Z \leq \zeta_n + \delta^2})}.
\end{eqnarray*}
Lemma \ref{l12} implies that $\mathbb{Q}_n(\tilde{\epsilon}_n^2 \ind {\zeta_n - \delta^2 < Z \leq \zeta_n + \delta^2})\rightarrow \sigma^2 \mathbb{P}({\zeta}_0 - \delta^2 < Z \leq {\zeta}_0 + \delta^2) = \sigma^2 \{2f(\zeta_0)\delta^2 + o(\delta^2)\}$. Hence, there is a constant $R>0$ such that the right side of the above equations are bounded by $R/\sqrt{m_n}$ and $R \sqrt{\delta^2 + o(\delta^2)} /\sqrt{m_n}$. Using similar arguments, we can in fact make $R$ large enough so that the following inequalities hold too
\begin{eqnarray}
\e{\left\|(\mathbb{P}_n^* - \mathbb{Q}_n)(\tilde{\epsilon}_n \ind{Z > (\cdot)\lor\zeta_n})\right\|_{[a,b]}} &\leq& \frac{R}{\sqrt{m_n}} \label{ec29}\\
\e{\left\|(\mathbb{P}_n^* - \mathbb{Q}_n)(\tilde{\epsilon}_n \ind{\zeta_n< Z \leq (\cdot)})\right\|_{|\zeta-\zeta_n|<\delta^2}} &\leq& \frac{R}{\sqrt{m_n}}\sqrt{\delta^2 + o(\delta^2)}. \label{ec30}
\end{eqnarray}
We also assume that $R > J_\mathcal{F}(1)$. Using (\ref{ec29}), (\ref{ec30}), the discussion preceding the display, and grouping two consecutive terms at a time in the expansion (\ref{ec11}), it is easily seen that \[ \sqrt{n}\e{\sup_{d(\theta,\theta_n)<\delta}\left\{\left|(\mathbb{M}_n - \mathcal{M}_n)(\theta) - (\mathbb{M}_n - \mathcal{M}_n)(\theta_n)\right|\right\}}\lesssim \frac{4R\sqrt{n}}{\sqrt{m_n}}\delta\] \[+ \frac{4RL\sqrt{n}}{\sqrt{m_n}}\sqrt{\delta^2 + o(\delta^2)} + \frac{2 R \sqrt{n}}{\sqrt{m_n}}\delta^2 + \frac{2 R L^2 f(\zeta_0)\sqrt{n}}{\sqrt{m_n}}(\delta^2 + o(\delta^2)). \]
Thus by taking $\eta > 0$ small enough we can show that (\ref{eq:RateConvCond}) holds for every $n\geq N$ and any $\delta\in[\delta_n,\eta)$, with $\delta_n$ and $N$ defined as in (\ref{ec26}). Defining $\phi_n(\delta)=\frac{\sqrt{n}}{\sqrt{m_n}}\delta$ and $r_n=\sqrt{m_n}$, the hypotheses of Theorem 3.4.1 of \cite{vw} are satisfied (note that Proposition \ref{pg1} implies that $d(\theta_n,\theta_n^*)\cip 0$). Therefore, $r_n d(\theta_n,\theta_n^*) = \sqrt{m_n(\alpha_n^* - \alpha_n)^2 + m_n(\beta_n^* - \beta_n)^2} + \sqrt{m_n|\zeta_n^* - \zeta_n|} = O_{\mathbf{P}}(1)$. $\hfill \square$

\subsubsection{Proof of Lemma \ref{l14}}\label{prueba3}
Let $\eta>0$ be an upper bound for the norm of the elements in $K$. The maximal inequality from \cite{kipo} and Lemma \ref{lKP} imply
\begin{eqnarray*}
\sqrt{m_n}\e{\left\|(\mathbb{P}_n^* - \mathbb{Q}_n)(\tilde{\epsilon}_n\ind{\zeta_n + \frac{(\cdot)}{m_n} < Z \leq \zeta_n})\right\|_{K}} & \leq & J_\mathscr{A}(1) \sqrt{\mathbb{Q}_n (\tilde{\epsilon}_n^2 \ind{\zeta_n - \frac{\eta}{m_n} < Z \leq \zeta_n})} \\
\sqrt{m_n}\e{\left\|(\mathbb{P}_n^* - \mathbb{Q}_n)(\ind{\zeta_n + \frac{(\cdot)}{m_n} < Z \leq \zeta_n})\right\|_{K}} & \leq & J_\mathcal{F} (1) \sqrt{\mathbb{Q}_n (\ind{\zeta_n - \frac{\eta}{m_n} < Z \leq \zeta_n})}.
\end{eqnarray*}
By  $(i)$ and $(iv)$ of Lemma \ref{l12} applied with $\mathbf{1}_{Z \le (\cdot) \land \zeta_n}$ in place of $\mathbf{1}_{(\cdot) < Z \le \zeta_n}$, we see that the righthand side of both the above inequalities go to zero. On the other hand, using (\ref{ec48}) and (\ref{ec49}) it is easy to see that both $\sqrt{m_n}\|\mathbb{Q}_n({\tilde{\epsilon}_n^2 \ind{\zeta_n + \frac{(\cdot)}{m_n} < Z \leq \zeta_n}})\|_K$ and $\sqrt{m_n} \|\mathbb{Q}_n({\ind{\zeta_n + \frac{(\cdot)}{m_n} < Z \leq \zeta_n}})\|_K$ converge to zero. Now, note that \\ $\sqrt{m_n}\left\|\function{\mathbb{P}_n^*}{\tilde{\epsilon}_n\ind{\zeta_n + \frac{(\cdot)}{m_n} < Z \leq \zeta_n}}\right\|_K$ is bounded by
\begin{eqnarray}
 &  & \sqrt{m_n}\left\|(\mathbb{P}_n^* - \mathbb{Q}_n)(\tilde{\epsilon}_n \ind{\zeta_n + \; \frac{(\cdot)}{m_n} < Z \leq \zeta_n})\right\|_{K} + \sqrt{m_n}\left\|\function{\mathbb{Q}_n}{|\tilde{\epsilon}_n|\ind{\zeta_n + \frac{(\cdot)}{m_n} < Z
\leq \zeta_n}}\right\|_K \nonumber
\end{eqnarray}
and thus $\sqrt{m_n} \left\|\function{\mathbb{P}_n^*}{\tilde{\epsilon}_n\ind{\zeta_n + \frac{(\cdot)}{m_n} < Z \leq \zeta_n}}\right\|_K \stackrel{\mathbb{L}_1}{\longrightarrow} 0. $
Similarly we can bound $\\$ $\sqrt{m_n}\left\|\function{\mathbb{P}_n^*}{\ind{\zeta_n + \frac{(\cdot)}{m_n} < Z
\leq \zeta_n}}\right\|_K$ and show that it converges to zero in mean. Finally, from the expressions
\begin{eqnarray}
A_n^*(h_1) - \hat{A}_n (h_1,h_3) &=& 2h_1\sqrt{m_n}\function{\mathbb{P}_n^*}{\tilde{\epsilon}_n\ind{\zeta_n + \frac{h_3}{m_n} < Z \leq \zeta_n}} - h_1^2 \function{\mathbb{P}_n^*}{\ind{\zeta_n +  \frac{(h_3}{m_n} < Z \leq \zeta_n}},\nonumber \\
C_n^*(h_3) - \hat{C}_n (h_2,h_3) & = &  2 h_2\sqrt{m_n} \function{\mathbb{P}_n^*}{\tilde{\epsilon}_n \ind{\zeta_n + \frac{h_3}{m_n} < Z \leq \zeta_n}}\nonumber \\
&  & \;\; - \; \left(2 h_2\sqrt{m_n}(\alpha_n -\beta_n) - h_2^2\right) \function{\mathbb{P}_n^*}{\ind{\zeta_n + \frac{h_3}{m_n} < Z \leq \zeta_n}}\nonumber
 \end{eqnarray}
we get that $\left\|A_n^* - \hat{A}_n \right\|_K\stackrel{\mathbb{L}_1}{\longrightarrow} 0$ and $\left\|C_n^* - \hat{C}_n\right\|_K \stackrel{\mathbb{L}_1}{\longrightarrow} 0$. With completely analogous
arguments, it is seen that $\left\|B_n^* - \hat{B}_n \right\|_K\stackrel{\mathbb{L}_1}{\longrightarrow} 0$ and $\left\|D_n^* - \hat{D}_n\right\|_K \stackrel{\mathbb{L}_1}{\longrightarrow} 0$ as well.
Observing that $\hat{E}_n = \hat{A}_n + \hat{B}_n + \hat{C}_n + \hat{D}_n -
\mathbb{P}_n^*(\tilde{\epsilon}_n^2)$ completes the proof of the result. $\hfill \square$

\subsubsection{Proof of Lemma \ref{l15}}\label{prueba4}
It suffices to show that each of the components of $(\Xi_n)_{n=1}^\infty$ is tight. Write $\tilde{\epsilon}_{n,j} = \tilde{\epsilon}_n (Z_{n,j},Y_{n,j})$ and let
\begin{eqnarray*}
    r_n & = & m_n\mathbb{Q}_n\left(e^{i\frac{\xi}{\sqrt{m_n}}\tilde{\epsilon}_{n}\ind{Z \leq \zeta_n}} - 1 -i\frac{\xi}{m_n}{\sqrt{m_n}}\tilde{\epsilon}_{n}\ind{Z\leq\zeta_n} + \frac{\xi^2}{2 m_n}\tilde{\epsilon}_{n}^2\ind{Z\leq\zeta_n}\right) \\
    & \le & \frac{m_n^{-1/2} \xi^3 \mathbb{Q}_n |\tilde{\epsilon}_{n}|^3}{6}.
\end{eqnarray*}
Then, assumption (VIII) implies that $r_n\rightarrow 0$ as $n\rightarrow\infty$.
Since the characteristic function of $\sqrt{m_n}\mathbb{P}_n^*(\tilde{\epsilon}_n\ind{Z\leq\zeta_n})$ is given by
\[ \e{ e^{i\xi\sqrt{m_n}\mathbb{P}_n^*(\tilde{\epsilon}_n\ind{Z \leq \zeta_n})}} = \left(1 + i\frac{\xi}{\sqrt{m_n}} \mathbb{Q}_n \left(\tilde{\epsilon}_n \ind{Z\leq\zeta_n}\right) - \frac{\xi^2}{2 m_n}\mathbb{Q}_n\left(\tilde{\epsilon}_{n}^2\ind{Z\leq\zeta_n}\right) +
\frac{r_n}{m_n}\right)^{m_n}\] taking the limit as $n\rightarrow\infty$ we can conclude that
$\sqrt{m_n}\mathbb{P}_n^*(\tilde{\epsilon}_n\ind{Z\leq \zeta_n})\rightsquigarrow  N(0,\mathbb{P}(Z\leq\zeta_0)\sigma^2)$ by using (VII) and the fact that $(1 + \beta_n/n)^n \rightarrow e^\beta$ if $\beta_n \rightarrow \beta$. With similar arguments, it is seen that
$\sqrt{m_n}\mathbb{P}_n^*(\tilde{\epsilon}_n\ind{Z> \zeta_n})\rightsquigarrow  N(0,\mathbb{P}(Z>\zeta_0)\sigma^2)$, so the first two components of the random vector of interest are uniformly tight.

Consider now the processes $\Gamma_n(t) = m_n\mathbb{P}_n^* (\ind{\zeta_n < Z\leq\zeta_n +\frac{t}{m_n}})$ and $\\$ $\Psi_n(t) = m_n\mathbb{P}_n^* (\tilde{\epsilon}_n\ind{\zeta_n < Z\leq\zeta_n +\frac{t}{m_n}})$.
For any process $\Psi \in \tilde{\mathcal{D}_I}$, $I \subset\mathbb{R}$ compact interval, $\delta>0$, we write \[ \function{w_{\Psi}^{''}}{\delta} = \sup\left\{\left|\Psi(t_1) - \Psi(t)\right| \land \left|\Psi(t_2)-\Psi(t)\right|\right\}\] where the supremum is taken over all $t_1\leq t\leq t_2\in I$ with $0\leq t_2 - t_1 \leq \delta$. Also, for any $A\subset I$, define $\displaystyle \function{w_\Psi}{A} = \sup_{s,t\in A} \left\{\left|\Psi(t)-\Psi(s)\right|\right\}$. This agrees with the notation defined in Chapter 14 of \cite{bi}. Let $\eta> 0$ be an upper bound for the absolute values of the elements of $I$, consider any $\rho >0$, and define the numbers $a_{\Psi}^\rho$ and $a_{\Gamma}^\rho$ by,
\begin{eqnarray}
a_{\Psi}^\rho &=& \frac{1}{\rho}\sup_{n\in\mathbb{N}}\left\{m_n\mathbb{Q}_n \left(|\tilde{\epsilon}_n|\ind{\zeta_n < Z \leq \zeta_n + \frac{\eta}{m_n}}\right)\right\}\nonumber\\
a_{\Gamma}^\rho &=& \frac{1}{\rho}\sup_{n\in\mathbb{N}}\left\{m_n\mathbb{Q}_n \left(\ind{\zeta_n < Z \leq \zeta_n + \frac{\eta}{m_n}}\right)\right\}.\nonumber
\end{eqnarray}
Then, using Markov's inequality,
\begin{eqnarray}
\lsup_{n\rightarrow\infty} \p{\sup_{t\in I} \left\{\left|\Psi_n(t)\right|\right\} > a_{\Psi}^\rho}&\leq& \rho
\label{ec36}\\
\lsup_{n\rightarrow\infty} \p{\sup_{t\in I} \left\{\left|\Gamma_n(t)\right|\right\} > a_{\Gamma}^\rho}&\leq& \rho.
\label{ec37}
\end{eqnarray}
Now, let $\rho,\gamma>0$ be any pair of positive numbers and assume that $I=[a,b]$. Then, choose $\delta<\frac{\gamma}{8|b-a|f(\zeta_0)^2}\land\frac{|b-a|}{4}\land\frac{1}{f(\zeta_0)}$ so there is an integer $N\geq 2$ such that $\delta<\frac{|b-a|}{N}<2\delta$. Define $s_j= a + \frac{j}{N}(b-a)$ and consider the partition $\left\{a=s_0<s_1<\ldots<s_N=b\right\}$ of $I$. Notice that if $\Psi$ is a step function on $I$, for $\function{w_{\Psi}^{''}}{\delta}$ to be positive, we need at least two jumps in an interval of size at most $\delta$. Then, the probability that at least two jumps of the process $\Psi_n$ happens on any interval $(s_{j-2},s_j]$ is bounded from above by
\begin{eqnarray*}
a_{j,m_n} & := & \p{\bigcup_{1\leq k < l \leq m_n}\bigg[m_n(Z_{n,k}-\zeta_n),m_n(Z_{n,l}-\zeta_n)\in (s_{j-2},s_{j}]\bigg]} \\
& \leq & \frac{m_n^2}{2}\mathbb{Q}_n\left(\zeta_n + \frac{s_{j-2}}{m_n} < Z \leq \zeta_n + \frac{s_j}{m_n}\right)^2
\end{eqnarray*}
and hence the limit superior of the probability that either $\Psi_n$ or $\Gamma_n$ has two jumps in any interval of the form $(s_{j-2},s_j]$ is bounded from above by $2 |b - a|^2 f(\zeta_0)^2/N^2$ by (VI). Therefore, the probability that at least two jumps happen in any interval of size at most $\delta$ is asymptotically bounded from above by
\[ \sum_{i=2}^N a_{j,m_n} \le \sum_{i=2}^N 2 |b - a|^2 f(\zeta_0)^2/N^2 \le 4 (N-1) f(\zeta_0)^2 |b - a| \delta /N \le \gamma. \] Thus,
\begin{eqnarray}
\lsup_{n\rightarrow\infty} \p{\function{w_{\Psi_n}^{''}}{\delta}>\rho} & < & \gamma \label{ec38}
\end{eqnarray}
The exact same argument can be used to show that
\begin{eqnarray}
    \lsup_{n\rightarrow\infty} \p{\function{w_{\Gamma_n}^{''}}{\delta}>\rho} &<& \gamma. \label{ec39}
\end{eqnarray}
Now, note that
\begin{eqnarray*}
    \p{\function{w_{\Psi_n}}{[a, a + \delta)}>\rho} & \le & \p{ \bigcup_{j =1}^{m_n} m_n (Z_{n,j} - \zeta_n) \in [a, a + \delta) >\rho} \\
    & \le & m_n \mathbb{Q}_n \left(\zeta_n + \frac{a}{m_n} < Z \leq \zeta_n + \frac{a + \delta}{m_n}\right)
\end{eqnarray*}
which implies that
\begin{eqnarray}
\lsup_{n\rightarrow\infty} \p{\function{w_{\Psi_n}}{[a, a + \delta)}>\rho} \le \delta f(\zeta_0) < \gamma. \label{ec40}
\end{eqnarray}
A similar analysis leads to the following bounds
\begin{eqnarray}
\lsup_{n\rightarrow\infty} \p{\function{w_{\Psi_n}}{[b-\delta,b)}>\rho} &<& \gamma \label{ec41}\\
\lsup_{n\rightarrow\infty} \p{\function{w_{\Gamma_n}}{[a,a + \delta)}>\rho} &<& \gamma \label{ec42}\\
\lsup_{n\rightarrow\infty} \p{\function{w_{\Gamma_n}}{[b-\delta,b)}>\rho} &<& \gamma \label{ec43}.
\end{eqnarray}
Putting together (\ref{ec36}), (\ref{ec37}), (\ref{ec38}), (\ref{ec39}), (\ref{ec40}), (\ref{ec41}), (\ref{ec42}) and (\ref{ec43}) and using Theorem 15.3 of \cite{bi} we obtain that both sequences $(\Psi_n)_{n=1}^\infty$ and $(\Gamma_n)_{n=1}^\infty$ are uniformly tight in $\tilde{\mathcal{D}_I}$. Similar arguments show the tightness of the third and fourth components of the process. Therefore, $(\Xi_n)_{n=1}^\infty$ is uniformly tight. The uniform tightness of $(E_n^*)_{n=1}^\infty$ now follows from the fact that $(\Xi_n)_{n=1}^\infty$ is uniformly tight and $E_n^*$ is a continuous function of $\Xi_n$. $\hfill \square$

\subsubsection{Proof of Lemma \ref{l16}}\label{prueba5}
In view of Lemma \ref{l15}, to show $(i)$ it suffices to show convergence of the finite dimensional distributions. To this end, consider the real numbers $t_{-N_-}<\ldots<t_{-1}<0 = t_0< t_1 <\ldots <t_{N_+}$ and the linear combination
\begin{eqnarray}
W_n &=& \mu \sqrt{m_n}\mathbb{P}_n^* (\tilde{\epsilon}_n \ind{Z\leq \zeta_n}) + \lambda \sqrt{m_n}\mathbb{P}_n^* \left( \tilde{\epsilon}_n\ind{Z> \zeta_n} \right)\nonumber\\
& & + \sum_{j=1}^{N_{-}} \left\{ \xi_{-j} m_n \mathbb{P}_n^* \left( \tilde{\epsilon}_n\ind{\zeta_n + \frac{t_{-j}}{m_n}<Z\leq \zeta_n } \right) + \eta_{-j}m_n\mathbb{P}_n^* \left(\ind{\zeta_n + \frac{t_{-j}}{m_n}<Z\leq \zeta_n} \right) \right\}  \nonumber\\
& & + \sum_{j=1}^{N_{+}} \left\{ \xi_{j}m_n\mathbb{P}_n^* \left( \tilde{\epsilon}_n\ind{\zeta_n<Z\leq \zeta_n + \frac{t_{j}}{m_n}} \right) +\eta_{j}m_n\mathbb{P}_n^* \left( \ind{\zeta_n<Z\leq \zeta_n + \frac{t_{j}}{m_n}} \right) \right\} \label{ec50}
\end{eqnarray}
where $\mu$, $\lambda$ and the $\xi_j$'s and the $\eta_j$'s are arbitrary real numbers. Now, set $\xi_0 = \eta_0 = 0$ and define
\begin{eqnarray}
    \mu_{\pm j} = \sum_{k=j}^{N_{\pm}}  \eta_{\pm k } \; \mbox{  and } \; \lambda_{\pm j} = \sum_{k=j}^{N_{\pm}} \xi_{\pm k }. \label{ec51}
\end{eqnarray}
Then grouping terms appropriately we can rewrite $W_n$ as
\begin{eqnarray}
W_n & = & \mu \sqrt{m_n}\mathbb{P}_n^* \left(\tilde{\epsilon}_n\ind{Z\leq \zeta_n+\frac{t_{-N_{-}}}{m_n}} \right) + \lambda \sqrt{m_n}\mathbb{P}_n^* \left(\tilde{\epsilon}_n\ind{Z> \zeta_n+\frac{t_{N_{+}}}{m_n}} \right) \nonumber\\
& & + \sum_{j=1}^{N_{-}} (\lambda_{-j} m_n + \mu \sqrt{m_n})\mathbb{P}_n^* \left( \tilde{\epsilon}_n \ind{\zeta_n + \frac{t_{-j}}{m_n}<Z\leq \zeta_n + \frac{t_{-j+1}}{m_n}} \right)\nonumber \\
& & + \sum_{j=1}^{N_{-}} \mu_{-j}m_n \mathbb{P}_n^* \left( \ind{\zeta_n + \frac{t_{-j}}{m_n}<Z\leq \zeta_n + \frac{t_{-j+1}}{m_n}} \right)   \nonumber\\
& & + \sum_{j=1}^{N_{+}} (\lambda_{j}m_n + \lambda \sqrt{m_n})\mathbb{P}_n^* \left( \tilde{\epsilon}_n \ind{\zeta_n + \frac{t_{j-1}}{m_n} <Z \leq \zeta_n + \frac{t_{j}}{m_n}} \right) \nonumber \\
& & + \sum_{j=1}^{N_{+}} \mu_{j} m_n \mathbb{P}_n^* \left( \ind{\zeta_n + \frac{t_{j-1}}{m_n} <Z\leq \zeta_n + \frac{t_{j}}{m_n}} \right). \nonumber
\end{eqnarray}
Using the independence of $X_{n,1},\ldots,X_{n,m_n}$, the characteristic function of $W_n$ is
\begin{eqnarray}
\e{e^{i s W_n}} = \left[ 1 +  \sum_{j=1}^{N_{-}} \function{\mathbb{Q}_n}{(e^{is( \frac{\mu}{\sqrt{m_n}}+\lambda_{-j})\tilde{\epsilon}_n+is\mu_{-j}}-1)\ind{\zeta_n + \frac{t_{-j}}{m_n} < Z \leq \zeta_n + \frac{t_{-j+1}}{m_n}}} \right.\nonumber \\
+ \function{\mathbb{Q}_n}{(e^{i \frac{s\mu}{\sqrt{m_n}}\tilde{\epsilon}_n}-1)\ind{Z \leq \zeta_n + \frac{t_{-N_{-}}}{m_n}}} + \function{\mathbb{Q}_n}{(e^{i \frac{s\lambda}{\sqrt{m_n}} \tilde{\epsilon}_n}-1)\ind{Z > \zeta_n + \frac{t_{N_{+}}}{m_n}}} \nonumber \\
\left. + \sum_{j=1}^{N_{+}} \function{\mathbb{Q}_n}{(e^{is( \frac{\lambda}{\sqrt{m_n}} + \lambda_{j})\tilde{\epsilon}_n+is\mu_{j}}-1)\ind{\zeta_n + \frac{t_{j-1}}{m_n} < Z \leq \zeta_n + \frac{t_{j}}{m_n}}}\right]^{m_n}. \label{ec52}
\end{eqnarray}
Let $r_n$ be given by
\[ r_n = m_n\mathbb{Q}_n\left[\left( e^{i\frac{s\mu}{\sqrt{m_n}}\tilde{\epsilon}_{n}} - 1 -i\frac{s\mu}{\sqrt{m_n}}\tilde{\epsilon}_{n} + \frac{s^2 \mu^2}{2 m_n}\tilde{\epsilon}_{n}^2 \right) \ind{Z\leq\zeta_n + \frac{t_{-N_-}}{m_n}}\right] \le \frac{s^3 \mathbb{Q}_n |\tilde \epsilon_n^3| }{6 \sqrt{m_n}}.\]
Condition (VIII) now implies that $r_n = o(1)$. But note that
\begin{eqnarray*}
    \function{\mathbb{Q}_n}{(e^{i \frac{s\mu}{\sqrt{m_n}}\tilde{\epsilon}_n}-1)\ind{Z \leq \zeta_n + \frac{t_{-N_{-}}}{m_n}}} = i\frac{s\mu}{m_n}\sqrt{m_n}\mathbb{Q}_n \left( \tilde{\epsilon}\ind{Z\leq\zeta_n+ \frac{t_{-N_-}}{m_n}} \right) \\
    - \frac{s^2 \mu^2}{2 m_n} \mathbb{Q}_n \left( \tilde{\epsilon}_n^2 \ind{Z \leq \zeta_n + \frac{t_{-N_-}}{m_n}} \right) + \frac{r_n}{m_n}
\end{eqnarray*}
and so $(i)$ of Lemma \ref{l12} together with condition (VII) and (\ref{ec48}) imply that
\begin{equation}\label{ec53}
m_n \function{\mathbb{Q}_n}{(e^{i \frac{s\mu}{\sqrt{m_n}}\tilde{\epsilon}_n}-1)\ind{Z \leq \zeta_n + \frac{t_{-N_{-}}}{m_n}}} = -\frac{s^2 \mu^2 }{2}\sigma^2\mathbb{P}(Z\leq \zeta_0) + o(1).
\end{equation}
Following a completely analogous argument one can show that
\begin{equation}\label{ec54}
m_n \function{\mathbb{Q}_n}{ \left( e^{i \frac{s\lambda}{\sqrt{m_n}}\tilde{\epsilon}_n} - 1 \right) \ind{Z > \zeta_n + \frac{t_{N_{+}}}{m_n}}} = -\frac{s^2 \lambda^2 }{2}\sigma^2\mathbb{P}(Z> \zeta_0) + o(1).
\end{equation}
Now, take $1\leq j \leq N_+$, and observe that equation (\ref{ec48}) implies
\begin{eqnarray*}
    m_n\left|\function{\mathbb{Q}_n}{(e^{is( \frac{\lambda}{\sqrt{m_n}}+\lambda_{j})\tilde{\epsilon}_n+is\mu_{j}}-e^{is\lambda_{j}\tilde{\epsilon}_n + is\mu_{j}})\ind{\zeta_n + \frac{t_{j-1}}{m_n} < Z \leq \zeta_n + \frac{t_{j}}{m_n}}}\right| \\
    \le |s\lambda| \sqrt{m_n} \mathbb{Q}_n \left( |\tilde{\epsilon}_n|\ind{\zeta_n + \frac{t_{j-1}}{m_n} < Z \leq \zeta_n + \frac{t_{j}}{m_n}} \right) \rightarrow 0.
\end{eqnarray*}
Using (VI) we can write
\[ m_n \function{\mathbb{Q}_n}{(e^{is( \frac{\lambda}{\sqrt{m_n}}+\lambda_{j})\tilde{\epsilon}_n+is\mu_{j}}-1)\ind{\zeta_n + \frac{t_{j-1}}{m_n} < Z \leq \zeta_n + \frac{t_{j}}{m_n}}} \]
\[ = (\varphi(s\lambda_j)e^{is\mu_j}-1) f(\zeta_0)(t_j - t_{j-1}) + o(1) \]
where $\varphi$ is the characteristic function of $\epsilon$ (under $\mathbb{P}$). Thus,
\begin{eqnarray}
m_n \sum_{j=1}^{N_{+}} \function{\mathbb{Q}_n}{(e^{is( \frac{\lambda}{\sqrt{m_n}}+\lambda_{j})\tilde{\epsilon}_n+is\mu_{j}}-1)\ind{\zeta_n + \frac{t_{j-1}}{m_n} < Z \leq \zeta_n + \frac{t_{j}}{m_n}}} \nonumber \\
= \sum_{j=1}^{N_{+}}(t_j - t_{j-1})f(\zeta_0)(\varphi(s\lambda_j)e^{is\mu_j}-1)  + o(1).\label{ec55}
\end{eqnarray}
Similarly, one can prove that
\begin{eqnarray}
m_n \sum_{j=1}^{N_{-}} \function{\mathbb{Q}_n}{(e^{is( \frac{\mu}{\sqrt{m_n}}+\lambda_{-j})\tilde{\epsilon}_n+is\mu_{-j}}-1)\ind{\zeta_n + \frac{t_{-j}}{m_n} < Z \leq \zeta_n + \frac{t_{-j+1}}{m_n}}} \nonumber\\
= \sum_{j=1}^{N_{-}}(t_{-j+1} - t_{-j})f(\zeta_0)(\varphi(s\lambda_{-j})e^{is\mu_{-j}}-1)  + o(1).\label{ec56}
\end{eqnarray}
So putting (\ref{ec50}), (\ref{ec51}), (\ref{ec52}), (\ref{ec53}), (\ref{ec54}), (\ref{ec55}) and (\ref{ec56}) together we see that,
\begin{eqnarray}
    \e{e^{i s W_n}} & \rightarrow & \exp \left[ \sum_{j=1}^{N_{-}} f(\zeta_0)(t_{-j+1} - t_{-j}) \left\{ \varphi \left( s ( \sum_{k=j}^{N_{-}}  \xi_{-k } ) \right) e^{is \sum_{k=j}^{N_-}  \eta_{-k }}-1 \right\} \right. \nonumber \\
    & & -\frac{s^2\mu^2\sigma^2}{2}\mathbb{P}(Z\leq\zeta_0) -\frac{s^2 \lambda^2 \sigma^2}{2} \mathbb{P}(Z>\zeta_0) \nonumber \\
    & & \left. + \sum_{j=1}^{N_{+}} f(\zeta_0)(t_{j} - t_{j-1}) \left\{ \varphi \left( s(\sum_{k=j}^{N_{+}}  \xi_{k }) \right) e^{is \left( \sum_{k=j}^{N_{+}}  \eta_{ k } \right)} - 1 \right\} \right]. \label{ec57}
\end{eqnarray}
But the right-hand side of (\ref{ec57}) is precisely $\e{e^{isW}}$ where, with the notation of (\ref{ec45}), $W$ is given by
\begin{eqnarray*}
    W = \mu\mathbf{Z}_1 +\lambda\mathbf{Z}_2 + \sum_{k=1}^{N_{-}}\left(\xi_{-k}\sum_{0<j\leq \nu_1(-t_{-k})} v_k \ind{t_{-k} < 0} + \eta_{-k}\nu_1(-t_{-k}) \ind{t_{-k} < 0}\right) \\
    + \sum_{k=1}^{N_{+}}\left(\xi_{k}\sum_{0<j\leq \nu_2(t_{k})} u_k \ind{t_{k} \geq 0} + \eta_{k}\nu_2(t_{k}) \ind{t_{k} \geq 0}\right)
\end{eqnarray*}
and thus $W_n \rightsquigarrow W$. From the fact that $\mu$, $\lambda$, the $\xi_j$'s and the $\eta_j$'s were arbitrarily chosen, by the Cramer-Wold device
{\footnotesize
\[ \left(\Xi_n(t_{-N_{-}}),\ldots,\Xi_n(t_{-1}),
    \Xi_n(t_{1}),\ldots,\Xi_n(t_{N_{+}})
    \right)'\rightsquigarrow
    \left(\Xi(t_{-N_{-}}),\ldots,\Xi(t_{-1}),
    \Xi(t_{1}),\ldots,
    \Xi(t_{N_{+}})
    \right)'.
\]}

This gives the convergence of the finite dimensional distributions, proving $(i)$. An application of the continuous mapping theorem shows that $(i)$ implies $(ii)$. Further, Lemma \ref{l14} and $(ii)$ now imply $(iii)$. $\hfill \square$

\subsubsection{Proof of Lemma \ref{l17}}\label{prueba6}
Every sample path of $E^*=E^* (h_1,h_2,h_3)$ can be written as
\begin{eqnarray*}
2h_1 \mathbf{Z}_1 - h_1^2 \mathbb{P}(Z\leq \zeta_0) + 2h_2 \mathbf{Z}_2 - h_2^2 \mathbb{P}(Z>\zeta_0)
+   \ind{h_3<0} 2(\alpha_0 - \beta_0)\sum_{j=1}^{\nu_1 (-h_3)}v_j \\
- (\alpha_0 - \beta_0)^2\nu_1 (-h_3)\ind{h_3<0} + \ind{h_3\geq 0}2(\beta_0 - \alpha_0)\sum_{j=1}^{\nu_2 (h_3)}u_j - \ind{h_3\geq 0}(\alpha_0 - \beta_0)^2\nu_2 (h_3).
\end{eqnarray*}
From this last expression it is obvious that for any fixed $h_3$, the $E^* (\cdot,\cdot,h_3)$ gets maximized at $\phi_1^* = \mathbf{Z}_1/\mathbb{P}(Z\leq \zeta_0)$ and $\phi_2^* = \mathbf{Z}_2/\mathbb{P}(Z> \zeta_0) $. The independence of the three co-ordinates follows from the fact that $\phi_1^*$ depends only on $\mathbf{Z}_1$, $\phi_2^*$ depends only on $\mathbf{Z}_2$, and $\phi_3^*$ depends only on $\mathbf{u}$, $\mathbf{v}$, $\nu_1$ and $\nu_2$. Since $E^*$ is piecewise constant in the third argument $h_3$, to complete the proof it is enough to show that $E^*(\phi_1^*, \phi_2^*, h_3)\rightarrow -\infty$ as $|h_3|\rightarrow\infty$. But this follows from the law of the iterated logarithm (applied to the random walks defined by the $v_i$'s and $u_i$'s) together with the fact that $\nu_1(t)\land\nu_2(t)\cas \infty$ as $t\rightarrow\infty$. Note that $\sum_{j=1}^{\nu_1(-h_3)} v_j$ is of order $O(\sqrt{\nu_1 \log \log \nu_1})$ a.s. as $h_3 \rightarrow \infty$. $\hfill \square$

\subsubsection{Proof of Proposition \ref{pg3}}\label{prueba7}
Lemma \ref{l17} and the fact that the $u_i$'s and the $v_i$'s come from a continuous distribution, show that $(E^*,J^*)$ satisfy the hypotheses of Lemma \ref{l11}, and in particular that (\ref{ec24}) holds. Moreover, Proposition \ref{pg2} shows that the sequence $(\sqrt{m_n}(\alpha_n^*-\alpha_n),\sqrt{m_n}(\beta_n^*-\beta_n),m_n(\zeta_n^*-\zeta_n))$ is tight. Now, consider $C\in\mathbb{N}$ and let $\phi_n$, $\phi_{n,C}$ and $\phi_C$ be the smallest maximizers of $\hat{E}_n$, $\hat{E}_n|_{[-C,C]^3}$ and $E^*|_{[-C,C]^3}$. To prove the result, we will apply Lemma \ref{l11} and Lemma 3.3 of \cite{lmm}. Using the notation of the latter, set $\epsilon = \frac{1}{C}$, $W_{n\epsilon} = \phi_{n,C}$, $W_\epsilon = \phi_C$, $W_n = \phi_n$ and $W=\phi^*$. From Proposition \ref{pg2} we see that $\displaystyle \lim_{\epsilon\rightarrow 0} \lsup_{n\rightarrow\infty} \p{W_{n\epsilon} \neq W_n} = 0$. Lemma \ref{l17} implies that $\displaystyle \lim_{\epsilon\rightarrow 0} \p{ W_\epsilon \neq W} = 0$. Finally, Lemma \ref{l11} and an application of Skorohod's Representation Theorem (see Theorem 1.8, page 102 of \cite{ek}) show that $W_{n\epsilon} \rightsquigarrow W_{\epsilon}$ and hence, from Lemma 3.3 of \cite{lmm}, we conclude that $\phi_n \rightsquigarrow \phi^*$. $\hfill \square$

\subsubsection{Proof of Lemma \ref{l3}}\label{app3}
We expand $m_\theta(X)$ as in (\ref{eq:simp_m_theta}) but with $\epsilon = Y  - \alpha_0 \mathbf{1}_{Z \le \zeta_0}  - \beta_0 \mathbf{1}_{Z > \zeta_0}$ in place of $\tilde \epsilon_n$ to get
\begin{eqnarray}\label{eq:simp_m_theta2}
    m_\theta(X) & = & -(\epsilon + \alpha_0 - \alpha)^2 \ind{Z \le \zeta_0 \land \zeta} - (\epsilon + \beta_0 - \alpha)^2 \ind{\zeta_0 < Z \le \zeta} \nonumber \\
    & & - (\epsilon + \alpha_0 - \beta)^2 \ind{\zeta < Z \le \zeta_0} - (\epsilon + \beta_0 - \beta)^2 \ind{Z > \zeta_0 \lor \zeta}.
\end{eqnarray}
Letting $\hat \gamma_n = (\hat \alpha_n, \hat \beta_n)$, we can also bound $M_n (\theta_0)$ using a similar argument as in the proof of Proposition \ref{pg1} to obtain
\begin{eqnarray*}
    |\hat{\gamma}_n-\gamma_0|^2 \mathbb{P}_n (Z<a)\land\mathbb{P}_n (Z>b) \qquad \qquad \qquad  \\
    \leq \function{\mathbb{P}_n}{\epsilon^2 \ind{a\leq Z\leq b}} + 2|\hat{\gamma}_n - \gamma_0| \left(|\function{\mathbb{P}_n}{\epsilon\ind{Z<a}}| + |\function{\mathbb{P}_n}{\epsilon\ind{Z>b}}|\right).
\end{eqnarray*}
By the strong law of large numbers
\begin{eqnarray*}
    \mathbb{P}_n (Z<a)\land\mathbb{P}_n (Z>b) & \stackrel{a.s.}{\longrightarrow} & \mathbb{P} (Z<a) \land \mathbb{P} (Z>b) \\
    \function{\mathbb{P}_n}{\epsilon^2 \ind{a\leq Z \leq b}} & \stackrel{a.s.}{\longrightarrow} & \sigma^2\function{\mathbb{P}}{a\leq Z\leq b} \; \mbox{ and } \\ |\function{\mathbb{P}_n}{\epsilon\ind{Z<a}}| + |\function{\mathbb{P}_n}{\epsilon\ind{Z>b}}| & \stackrel{a.s.}{\longrightarrow} & 0.
\end{eqnarray*}
Therefore, w.p. 1 we can write $$|\hat{\gamma}_n-\gamma_0|^2 \leq O(1) +|\hat{\gamma}_n-\gamma_0|o(1)$$ and thus the sequence $\left(\hat{\gamma}_n-\gamma_0\right)_{n=1}^{\infty}$ is bounded w.p. 1.

Now, take any compact set $K\subset\Theta$ and consider the classes of functions
\begin{eqnarray}
\mathcal{K}_1 &=& \left\{ \left(\epsilon + \alpha_0 - \alpha\right)^2\ind{(-\infty,\zeta\land\zeta_0]}\right\}_{\theta\in K}\nonumber\\
\mathcal{K}_2 &=& \left\{ \left(\epsilon + \beta_0 - \alpha\right)^2\ind{(\zeta_0,\zeta]}\right\}_{\theta\in K}\nonumber\\
\mathcal{K}_3 &=& \left\{ \left(\epsilon + \alpha_0 - \beta\right)^2\ind{(\zeta,\zeta_0]}\right\}_{\theta\in K}\nonumber\\
\mathcal{K}_4 &=& \left\{ \left(\epsilon + \beta_0 - \beta\right)^2\ind{(\zeta\lor\zeta_0,\infty)}\right\}_{\theta\in K}. \nonumber
\end{eqnarray}
If $A^*$ is an upper bound for the norm of the elements in $K$, we can see that each of these classes is a VC-subgraph class with integrable envelope $(|\epsilon|+ A^* + |\gamma_0|)^2$. With the notation $\left\|Q\right\|_\mathcal{F}=\sup\left\{|Qf|:f\in\mathcal{F}\right\}$ for classes of functions $\mathcal{F}$ and probability measures $Q$, a combination of Theorems 2.6.7 and 2.4.3 of \cite{vw} shows that all four quantities $\left\|\mathbb{P}_n - \mathbb{P}\right\|_{\mathcal{K}_j}$, $j=1,2,3,4$, converge to zero almost surely. Therefore using (\ref{eq:simp_m_theta2}), we get the inequality $$\displaystyle \left\|M_n - M\right\|_K \leq \sum_{1\leq j \leq 4} \left\|\mathbb{P}_n - \mathbb{P}\right\|_{\mathcal{K}_j}$$ which now implies $(i)$ ( Since $M_n,M\in\mathcal{D}_K$, $\left\|M_n - M\right\|_K$ is measurable.). The second assertion follows immediately from $(ii)$.

Consider a family of compact rectangles $\Theta_n\subset\Theta_{n+1}$ such that $\displaystyle \Theta=\cup_{n=1}^\infty \Theta_n$. Then, since the sequence $\left( \hat{\gamma}_n - \gamma_0\right)_{n=1}^{\infty}$ is almost surely bounded, w.p. 1 we have that there is some $m\in\mathbb{N}$ such that $\Theta_m$ contains both $\theta_0$ and the entire sequence $(\hat{\theta}_n)_{n=1}^\infty$. Finally, from (\ref{ec3}) with $\theta_n$ replaced by $\theta_0$ it is seen that
\begin{eqnarray*}
    M(\theta) = -\sigma^2 - (\alpha_0-\alpha)^2\mathbb{P}(Z\leq \zeta\land\zeta_0) - (\alpha_0-\beta)^2\mathbb{P}(\zeta<Z\leq\zeta_0) \\
    - (\alpha-\beta_0)^2\mathbb{P}(\zeta_0<Z\leq\zeta) - (\beta_0-\beta)^2\mathbb{P}(Z>\zeta\lor\zeta_0).
\end{eqnarray*}
As $\alpha_0\neq\beta_0$ and $Z$ has a strictly positive density on $[a,b]$, the last equation shows that $M$ satisfies the conditions of Lemma \ref{l2}. Since the event that $M_n\rightarrow M$ in $\mathcal{D}_{\Theta_k}$ for all $k \in \mathbb{N}$ has probability one, Lemma \ref{l2} allows us to conclude that $\sargmax(M_n) = \hat{\theta}_n \cas \theta_0$. $\hfill \square$ \newline

\subsubsection{Proof of Lemma \ref{l5}}\label{pl5}
Let $\rho,\delta>0$. We know from Corollary \ref{corolario} that the sequences $\left(\sqrt{n}(\hat{\alpha}_n - \alpha_0)\right)_{n=1}^\infty$, $\left(n(\hat{\zeta}_n - \zeta_0)\right)_{n=1}^\infty$ and $\left(n\function{\mathbb{P}_n}{\zeta_0 - \frac{h}{n} < Z \leq \zeta_0 + \frac{h}{n}}\right)_{n=1}^\infty$, for any $h>0$, are all stochastically bounded. Thus, since $m_n = O(n)$ there is $L>0$ such that $\p{m_n|\hat{\zeta}_n - \zeta_0|>L}<\rho$ and $\p{\sqrt{m_n}|\hat{\alpha}_n - \alpha_0|>L}<\rho$ for any $n\in\mathbb{N}$. Therefore,
\begin{eqnarray*}
    & & \p{m_n^\gamma\left\|\mathbb{P}_n(\hat{\zeta}_n + \frac{(\cdot)}{m_n} < Z \leq
    \hat{\zeta}_n)\right\|_K > \delta} \\
    & \leq & \frac{m_n^\gamma}{\delta}\e{\mathbb{P}_n \left( \zeta_0 - \frac{L+\eta}{m_n} < Z \leq \zeta_0 + \frac{L}{m_n}\right)} + \p{m_n|\hat{\zeta}_n - \zeta_0|>L} \\
    & \leq & f(\zeta_0)\frac{\eta+2L}{\delta}m_n^{\gamma -1} + o\left(m_n^{\gamma - 1}\right) + \rho,
\end{eqnarray*}
so by letting $n\rightarrow\infty$ and then $\rho\rightarrow 0$ we get $(i)$.

We prove $(ii)$ for when $p=1$, the case $p=2$ follows from similar arguments. Note that if $m_n|\hat{\zeta}_n - \zeta_0| \leq L$, then $m_n^\gamma \| \mathbb{P}_n(|\tilde{\epsilon}_n| \ind{\hat{\zeta}_n + \frac{(\cdot)}{m_n} < Z \leq \hat{\zeta}_n})\|_K$ can be bounded by
{\footnotesize
\[m_n^\gamma\left\|\function{\mathbb{P}_n}{|\epsilon|\ind{\zeta_0 - \frac{L}{m_n} + \frac{(\cdot)}{m_n} < Z \leq \zeta_0 +\frac{L}{m_n}}}\right\|_K + m_n^\gamma |\hat{\alpha}_n-\alpha_0|\left\|\mathbb{P}_n(\zeta_0 - \frac{L}{m_n} + \frac{(\cdot)}{m_n} < Z \leq \zeta_0 +\frac{L}{m_n})\right\|_K . \]}
But just as in the proof of $(i)$, we have
{\footnotesize
\begin{eqnarray*}
    & & \p{m_n^\gamma\left\|\mathbb{P}_n(|\epsilon|\ind{\hat{\zeta}_n + \frac{(\cdot)}{m_n} < Z \leq
    \hat{\zeta}_n})\right\|_K > \delta} \\
    & \leq & \p{m_n^\gamma\left\|\function{\mathbb{P}_n}{|\epsilon|\ind{\zeta_0 - \frac{L}{m_n} + \frac{(\cdot)}{m_n} < Z \leq \zeta_0 +\frac{L}{m_n}}}\right\|_K > \frac{\delta}{2}} + \\
    & & \p{m_n^\gamma |\hat{\alpha}_n-\alpha_0|\mathbb{P}_n(\zeta_0 - \frac{L}{m_n} + \frac{\eta}{m_n} < Z \leq \zeta_0 +\frac{L}{m_n})>\frac{\delta}{2}} + \p{m_n|\hat{\zeta}_n - \zeta_0|>L} \\
    & \leq & \frac{2m_n^\gamma}{\delta}\e{\function{\mathbb{P}_n}{|\epsilon|\ind{\zeta_0 - \frac{L}{m_n} + \frac{\eta}{m_n} < Z \leq \zeta_0 +\frac{L}{m_n}}}} + \\
    & & \p{m_n^\gamma |\hat{\alpha}_n-\alpha_0|\mathbb{P}_n(\zeta_0 - \frac{L}{m_n} + \frac{\eta}{m_n} < Z \leq \zeta_0 +\frac{L}{m_n})>\frac{\delta}{2}} + \p{m_n|\hat{\zeta}_n - \zeta_0|>L} \\
    & \leq & f(\zeta_0)\e{|\epsilon|}\frac{2(\eta+2L)}{\delta}m_n^{\gamma -1} + o\left(m_n^{\gamma -1} \right)+\\
    & & \p{m_n^\gamma |\hat{\alpha}_n-\alpha_0|\mathbb{P}_n(\zeta_0 - \frac{L}{m_n} + \frac{\eta}{m_n} < Z \leq \zeta_0 +\frac{L}{m_n})>\frac{\delta}{2}} + \rho.
\end{eqnarray*}}
The result follows again by letting $n\rightarrow\infty$ and $\rho\rightarrow 0$.$\hfill \square$

The next results will be useful to support our conjecture of inconsistency of some of our bootstrap scenarios.

\begin{lemma}\label{l18}
Let $\lambda, B>0$, $\rho\in(0,\frac{1}{2})$ and $H_\lambda$ be the distribution function of a Poisson random variable with mean $\lambda$. For each value of $\lambda$ write $L_{\lambda+B}^\rho = \min\left\{n\in\mathbb{N}: H_{\lambda+B} (n)>\rho\right\}$ and $U_{\lambda}^\rho = \max\left\{n\in\mathbb{N}: 1-H_\lambda(n)>\rho\right\}$. Then, there is $\lambda_*>0$ such that $L_{\lambda+B}^\rho < U_{\lambda}^\rho$ for all $\lambda\geq \lambda_*$.
\end{lemma}
\begin{proof} Let $c_\lambda$ be the median (i.e. $c_\lambda = \min \{n
\in \mathbb{N}: H_{\lambda}(n)>\frac{1}{2}\}.$) of $H_\lambda$. Observe that $c_\lambda \leq U_{\lambda}^\rho$. According to \cite{haz}, $|c_\lambda - \lambda|<\log (2)$ for any positive $\lambda$. Letting $\lfloor x \rfloor$ denote the greatest integer less than or equal to $x$, we have
\begin{eqnarray}
    & & \left|H_{\lambda+B}(c_{\lambda+B}) - H_{\lambda+B}(c_{\lambda})\right| \nonumber \\
    & \le & \left|H_{\lambda+B}( \lambda + B +\log (2))-H_{\lambda+B}(\lambda - \log (2))\right| \nonumber \\
    & \le & (B+2\log(2))e^{-(\lambda+B)}\frac{(\lambda + B)^{\lfloor \lambda + B \rfloor}}{\lfloor \lambda + B \rfloor !}  \rightarrow 0 \textrm{ as } \lambda \rightarrow \infty.\nonumber
\end{eqnarray}
as the Poisson mass function has a maximum at $\lfloor \lambda + B \rfloor$. Therefore, \\ $\linf_{\lambda \rightarrow \infty} H_{\lambda + B} (U_{\lambda}^\rho)\geq 1/2$. But we also note that $\sup_{n\in\mathbb{N}} \{H_{\lambda+B}(n+1) $ $- H_{\lambda+B}(n)\}\rightarrow 0$ as $\lambda \rightarrow \infty$. Thus,
\[ \lsup_{\lambda\rightarrow\infty}H_{\lambda+B}(L_{\lambda + B}^\rho+1) = \rho < \frac{1}{2} \leq \linf_{\lambda\rightarrow\infty} H_{\lambda + B}(U_{\lambda}^\rho).\]
It follows that $U_{\lambda}^\rho > L_{\lambda + B}^\rho$ for all $\lambda$ sufficiently large. \end{proof}

\begin{lemma}\label{l18bis}
$\\$
Let $\lambda,B>0$, $0<\rho<\frac{1}{2}$, $\mu$ and $\nu$ be two nondegenerate Borel probability measures on $\mathbb{R}$ and $H_{\mu,\lambda}$ denote the compound Poisson distribution with intensity $\lambda$ and compounding distribution $\mu$. For each value of $\lambda$ write $L_{\nu,\lambda+B}^\rho = \inf\left\{s\in\mathbb{R}: H_{\nu,\lambda+B} (s)\geq\rho\right\}$ and $U_{\mu,\lambda}^\rho = \sup\left\{s\in\mathbb{R}: 1-H_{\mu,\lambda}(s)\geq\rho\right\}$. In addition, assume that $\int x^2\nu(dx),\int x^2\mu(dx)<\infty$ and that $\int x \nu(dx) \leq \int x \mu (dx)$.  Then there is $\lambda_*>0$ such that $L_{\nu,\lambda+B}^\rho < U_{\mu,\lambda}^\rho$ for all $\lambda\geq \lambda_*$. Moreover, let $0<r<1$, suppose that there is another Borel probability measure $\gamma$ on $\mathbb{R}$ and define $\nu_\gamma:= \frac{rB}{\lambda+B}\gamma+\frac{\lambda+(1-r)B}{\lambda+B}$ and the corresponding constant $L_{\nu_\gamma,\lambda+B}^\rho = \inf\left\{s\in\mathbb{R}: H_{\nu_\gamma,\lambda+B} (s)\geq\rho\right\}$. Then there is $\lambda_*>0$ such that $L_{\nu_\gamma,\lambda+B}^\rho < U_{\mu,\lambda}^\rho$ for all $\lambda\geq \lambda_*$.
\end{lemma}
\begin{proof}
Denote by $\Phi$ the standard normal distribution and $\mathbf{z}_\alpha$ the lower $\alpha$-quantile of $\Phi$ (i.e. $\Phi(\mathbf{z}_\alpha)=\alpha$). Also, write $c_\mu := \int x\mu(dx)$, $d_\mu:= \int x^2\mu(dx)$ and define the corresponding quantities $c_\nu$ and $d_\nu$ for $\nu$. For any possible value of $\lambda$ and $\mu$ denote by $T_{\mu,\lambda}$ a random variable with distribution $H_{\mu,\lambda}$. It is easily seen (as, for instance, in Theorem 2.1 of \cite{mo2005}) that $\displaystyle S_{\mu,\lambda}:=\frac{T_{\mu,\lambda} - \lambda c_\mu}{\sqrt{\lambda d_\mu}}\rightsquigarrow \Phi$ as $\lambda\rightarrow\infty$. Since the standard normal distribution is continuous, the distributions of $S_{\mu,\lambda}$ converge uniformly on $\mathbb{R}$ to $\Phi$ as $\lambda\rightarrow\infty$.

Let $1<\kappa<1/(2\rho)$. Then, since the distributions of $S_{\mu,\lambda}$ converge uniformly to $\Phi$, there is $\lambda_1$ such that {$1- \function{\Phi}{\frac{ U_{\mu,\lambda}^\rho - \lambda c_\mu}{\sqrt{\lambda d_\mu}}}<\kappa\rho$} for $\lambda>\lambda_1$  and $\lambda_2>0$ such that {$\function{\Phi}{\frac{L_{\nu,\lambda+B}^\rho - (\lambda + B) c_\nu}{\sqrt{(\lambda+B) d_\nu}}}<\kappa\rho$} for all $\lambda>\lambda_2$. These two inequalities in turn imply that
\begin{eqnarray*}
U_{\mu,\lambda}^\rho &>& \lambda c_\mu - \sqrt{\lambda d_\mu}\mathbf{z}_{\kappa\rho},\\
L_{\nu,\lambda+B}^\rho &<& (\lambda+B) c_\nu + \sqrt{(\lambda +B) d_\nu}\mathbf{z}_{\kappa\rho}.
\end{eqnarray*}
Since $c_\mu\geq c_\nu$ we can find $\lambda_3$ such that
\[ (\lambda+B) c_\nu + \sqrt{(\lambda +B) d_\nu}\mathbf{z}_{\kappa\rho} <  \lambda c_\mu - \sqrt{\lambda d_\mu}\mathbf{z}_{\kappa\rho}\ \ \ \textrm{for all}\ \lambda \geq \lambda_3.\]
The first part of the result now follows by taking $\lambda_*:=\lambda_1\lor\lambda_2\lor\lambda_3$.
To prove the result for the measure $\nu_\gamma$ it suffices to see that we also have $\displaystyle \frac{T_{\nu_\gamma,\lambda+B} - (\lambda+B)c_{\nu_\gamma}}{\sqrt{(\lambda+B)d_{\nu_\gamma}}}\rightsquigarrow\Phi$, as $\lambda\rightarrow\infty$ (this is easily seen by analyzing the characteristic functions). The rest follows from the same argument used to prove the first part of the lemma.
\end{proof}


\subsubsection{Proof of Lemma \ref{l19}}\label{prueba21}
\noindent {\bf Proof of $(i)$:} Let $s < t$. Note that $(Z_n)_{n=1}^\infty$ is a collection of i.i.d. random variables and  $n \mathbb{P}_n(\zeta_0+\frac{s}{n}<Z\leq\zeta_0 + \frac{t}{n})$ is permutation invariant, so the Hewitt-Savage 0-1 law (see page 304 of \cite{bi2}) implies that any convergent subsequence must converge to a constant. On the other hand, Lemma \ref{l16} implies that $n \mathbb{P}_n(\zeta_0 + \frac{s}{n}< Z \leq \zeta_0 + \frac{t}{n}) \rightsquigarrow \textrm{Poisson}((t-s)f(\zeta_0))$. Therefore, $\left(n \mathbb{P}_n(\zeta_0+\frac{s}{n}<Z\leq\zeta_0 + \frac{t}{n})\right)_{n=1}^\infty$ has no almost surely convergent subsequence.

\noindent {\bf Proof of $(ii)$:} Now, let $\delta\in(0,\frac{1}{4})$. From Proposition \ref{pg2} we know that there is $B_\delta > 0$ such that $\p{n|\hat{\zeta}_n - \zeta_0|\leq B_\delta}>1-\delta$ for any $n\in\mathbb{N}$.
Choose $h>2B_{\delta}$ and take any increasing sequence of natural numbers $n_k$. Write $\hat{T}_k = n_k \mathbb{P}_{n_{k}}(\hat{\zeta}_{n_{k}}< Z \leq \hat{\zeta}_{n_{k}} + \frac{h}{n_k})$, $S_k = n_k \mathbb{P}_{n_{k}}(\zeta_0 - \frac{B_\delta}{n_{k}}< Z \leq \zeta_0 + \frac{h+B_\delta}{n_{k}})$ and $T_k = n_k \mathbb{P}_{n_{k}}(\zeta_0+\frac{B_\delta}{n_{k}}<Z\leq\zeta_0 + \frac{h-B_\delta}{n_{k}})$. Then, $\left\{n_{k}|\hat{\zeta}_{n_k} - \zeta_0|\leq B_\delta\right\} \subset \left\{S_k \geq \hat{T}_k \geq T_k\right\}$ and therefore we have $\p{\hat{T}_k\geq T_k}\land\p{S_k\geq \hat{T}_k}>1-\delta$ for all $k$.

We know that $T_k\rightsquigarrow\textrm{Poisson}((h-2B_\delta)f(\zeta_0))$ and $S_k\rightsquigarrow\textrm{Poisson}((h+2B_\delta)f(\zeta_0))$, so in view of Lemma \ref{l18} with $B = 4 B_\delta f(\zeta_0) $ and $\lambda = (h - 2 B_\delta) f(\zeta_0)$, there is a number $h_* > 2B_\delta$ large enough so that whenever $h \geq h_*$ we can find two numbers $N_{1,h} < N_{2,h}\in\mathbb{N}$ with the property that, $\linf_{k\rightarrow\infty}\p{T_k > N_{2,h}}>2\delta$ and $\linf_{k\rightarrow\infty}\p{S_k \leq N_{1,h}}>2\delta$. Thus, for $h\geq h_*$, $\p{T_k > N_{2,h}} > 2\delta$ and $\p{S_k \leq N_{1,h}} > 2\delta$ for all but a finite number of $k$'s. Therefore, for any $k$ large enough, $\p{T_k > N_{2,h}}\land\p{S_k \leq N_{1,h}} > 2\delta$. Using the fact that $\p{S_k \geq \hat{T}_k \geq T_k} > 1 -\delta$ we get that $\p{\hat{T}_k \geq T_k > N_{2,h}}\land \p{N_{1,h}\geq S_k\geq \hat{T}_k}> \delta$ for all but finitely many $k$'s. Thus, whenever $h\geq h_*$, $$\p{\hat{T}_k \geq T_k > N_{2,h}, i.o.}>\delta \mbox{ and } \p{N_{1,h}\geq S_k\geq \hat{T}_k, i.o.}> \delta.$$ But for every $k\in\mathbb{N}$, the events $\left\{\hat{T}_k \geq T_k > N_{2,h}\right\}$ and $\left\{N_{1,h}\geq S_k\geq \hat{T}_k\right\}$ are permutation-invariant on the i.i.d. random vectors $X_1,\ldots,X_{n_k}$. Hence, the Hewitt-Savage 0-1 law implies that $\p{\hat{T}_k \geq T_k > N_{2,h}, i.o.}=1$ and $\p{N_{1,h}\geq S_k\geq \hat{T}_k, i.o.}=1$. Since $N_{1,h} < N_{2,h}$ it follows that $\hat{T}_k = n_k \mathbb{P}_{n_{k}}(\hat{\zeta}_{n_{k}} < Z \leq \hat{\zeta}_{n_{k}} + h/n_k)$ does not have an almost sure limit. But the choice of the subsequence $n_k$ was arbitrary and independent of $h_*$ so we can conclude that for any $h\geq h_*$, the sequence
$\left\{n \mathbb{P}_n( \hat{\zeta}_n < Z \leq \hat{\zeta}_n + \frac{h}{n})\right\}_{n=1}^\infty$ does not converge in probability. Proceeding analogously, we can prove the same for $\left\{n \mathbb{P}_n(\hat{\zeta}_n - \frac{h}{n}<Z\leq\hat{\zeta}_n)\right\}_{n=1}^\infty$.

\noindent {\bf Proof of $(iii)$:} We introduce some notation, for any two Borel probability measures $\mu$ and $\nu$ on $\mathbb{R}$ we write $\mu\bigstar\nu$ for their convolution and for $\lambda>0$ we write $\textrm{CPoisson}(\mu,\lambda)$ for the compound Poisson distribution with intensity $\lambda$ and compounding distribution $\mu$. Let $\mu_\alpha$ and $\mu_\beta$ be, respectively, the distributions under $\mathbb{P}$ of $\phi(\epsilon+\alpha_0)$ and $\phi(\epsilon+\beta_0)$.

Observe that depending on whether $t<0$, $s<0<t$ or $s>0$ we have that $n \mathbb{P}_n(\phi(Y)\ind{\zeta_0+\frac{s}{n}<Z\leq\zeta_0 + \frac{t}{n}})$ converges weakly to $\textrm{CPoisson}(\mu_\alpha,(t-s)f(\zeta_0))$, \newline $\textrm{CPoisson}(\mu_\alpha,sf(\zeta_0))\bigstar\textrm{CPoisson}(\mu_\beta,tf(\zeta_0))$ or $\textrm{CPoisson}(\mu_\beta,(t-s)f(\zeta_0))$, respectively. This follows easily from convergence of the corresponding characteristic functions. Considering that $\{(Y_n,Z_n)\}_{n=1}^\infty$ is a collection of i.i.d. random vectors and that $n \mathbb{P}_n(\phi(Y)\ind{\zeta_0+\frac{s}{n}<Z\leq\zeta_0 + \frac{t}{n}})$ is permutation invariant for $(Y_1,Z_1),\ldots,(Y_n,Z_n)$ the same argument as in (i) applies here as well.

\noindent {\bf Proof of (iv):} We keep the notation used in the proof of $(iii)$. The argument here is quite similar to the one used to show $(ii)$. Assume without loss of generality that $\phi\leq 0$.

Now, let $\delta\in\left(0,\frac{1}{4}\right)$ and $N\in\mathbb{N}$. From Proposition \ref{pg2} we know that there is $B_\delta > 0$ such that $\p{n|\hat{\zeta}_n - \zeta_0|\leq B_\delta}>1-\delta$ for any $n\in\mathbb{N}$.
Choose $h>2B_{\delta}$ and take any increasing sequence of natural numbers $n_k$. Write $\hat{T}_{k,h}^\phi = n_k \mathbb{P}_{n_{k}}(\phi(Y)\ind{\hat{\zeta}_{n_{k}}< Z \leq \hat{\zeta}_{n_{k}} + \frac{h}{n_k}})$, $S_{k,h}^\phi = n_k \mathbb{P}_{n_{k}}(\phi(Y)\ind{\zeta_0 - \frac{B_\delta}{n_{k}}< Z \leq \zeta_0 + \frac{h+B_\delta}{n_{k}}})$ and \newline {\footnotesize $T_{k,h}^\phi = n_k \mathbb{P}_{n_{k}}(\phi(Y)\ind{\zeta_0+\frac{B_\delta}{n_{k}}<Z\leq\zeta_0 + \frac{h-B_\delta}{n_{k}}})$}. Then, $\left\{n_{k}|\hat{\zeta}_{n_k} - \zeta_0|\leq B_\delta\right\} \subset \left\{S_{k,h}^\phi \leq \hat{T}_{k,h}^\phi \leq T_{k,h}^\phi\right\}$ and therefore we have $\p{\hat{T}_{k,h}^\phi\leq T_{k,h}^\phi}\land\p{S_{k,h}^\phi\leq \hat{T}_{k,h}^\phi}>1-\delta$ for all $k$.

We know that $T_{k,h}^\phi\rightsquigarrow\textrm{CPoisson}(\mu_\beta,(h-2B_\delta)f(\zeta_0))$ and
$$ S_k^\phi\rightsquigarrow\textrm{CPoisson}(\mu_\alpha,2B_\delta f(\zeta_0))\bigstar\textrm{CPoisson}(\mu_\beta,(h+B_\delta)f(\zeta_0))$$
$$\equiv
\textrm{CPoisson}\left(\frac{B_\delta }{h+2B_\delta}\mu_\alpha+\frac{h+B_\delta }{h+2B_\delta}\mu_\beta,(h+2B_\delta)f(\zeta_0)\right),$$
as $k\rightarrow\infty$.

An application of Lemma \ref{l18bis} with $\mu = \nu = \mu_\beta$, $\gamma=\mu_\alpha$, $B = 4 B_\delta f(\zeta_0)$, $r=\frac{1}{4}$ and $\lambda = (h - 2 B_\delta) f(\zeta_0)$, shows the existence of an $h_* > 2B_\delta$ large enough so that whenever $h \geq h_*$ we can find two numbers $R_{1,h} > R_{2,h}\in\mathbb{N}$ with the property that $\linf_{k\rightarrow\infty}\p{T_{k,h}^\phi < R_{2,h}}>2\delta$ and $\linf_{k\rightarrow\infty}\p{S_{k,h}^\phi \geq R_{1,h}}>2\delta$. Thus, for $h\geq h_*$, $\p{T_{k,h}^\phi < R_{2,h}} > 2\delta$ and $\p{S_{k,h}^\phi \geq R_{1,h}} > 2\delta$ for all but a finite number of $k$'s. Therefore, for any $k$ large enough, $\p{T_{k,h}^\phi < R_{2,h}}\land\p{S_{k,h}^\phi \geq R_{1,h}} > 2\delta$. Using the fact that $\p{S_{k,h}^\phi \leq \hat{T}_{k,h}^\phi \leq T_{k,h}^\phi} > 1 -\delta$ we get that $\p{\hat{T}_{k,h}^\phi \leq T_{k,h}^\phi < R_{2,h}}\land \p{R_{1,h}\leq S_{k,h}^\phi \leq \hat{T}_{k,h}^\phi}> \delta$ for all but finitely many $k$'s. Thus, whenever $h\geq h_*$, $$\p{\hat{T}_{k,h}^\phi \leq T_{k,h}^\phi < R_{2,h}, i.o.}>\delta \mbox{ and } \p{R_{1,h}\leq S_{k,h}^\phi\leq \hat{T}_{k,h}^\phi, i.o.}> \delta.$$
The argument relying on the Hewitt-Savage 0-1 law applied in the proof of $(ii)$ can be used to finish this proof. \newline A completely analogous proof applies for $\left\{n \mathbb{P}_n(\phi(Y)\ind{\hat{\zeta}_n - \frac{h}{n}<Z\leq\hat{\zeta}_n})\right\}_{n=1}^\infty$. $\hfill \square$

\subsubsection{Proof of Lemma \ref{lsho}}\label{plsho}
We start by computing the characteristic functions of the weak limits of the last two components of the process $\tilde \Xi_n$ as defined in (\ref{ecsins}). Let $g_n(\xi)$ and $\psi_n(\xi)$ be the (unconditional) characteristic functions of $n\mathbb{P}_n^*(\ind{\zeta_0<Z\leq\zeta_0 + \frac{t}{n}})$ and $n\mathbb{P}_n^*(\epsilon\ind{\zeta_0<Z\leq\zeta_0 + \frac{t}{n}})$, respectively. Fix $\xi\in\mathbb{R}$ and write
\begin{eqnarray}
\Lambda_n &:=& \ce{\textrm{e}^{i\xi n\mathbb{P}_n^*(\epsilon\ind{\zeta_0<Z\leq\zeta_0 + \frac{t}{n}})}},\nonumber\\
\Psi_n &:=& n\function{\mathbb{P}_n}{\left(\textrm{e}^{i\xi\epsilon}-1\right)\ind{\zeta_0<Z\leq \zeta_0 + \frac{t}{n}}},\nonumber\\
\Psi_\xi^* &:=& \sum_{1\leq k\leq\nu(t)} \left(e^{i\xi\epsilon_k}-1\right)\nonumber,
\end{eqnarray}
where $(\nu(s))_{s\geq 0}$ is a Poisson process with rate $f(\zeta_0)$ independent of $(\epsilon_n)_{n=1}^\infty$. Then, $\psi_n(\xi) = \e{\Lambda_n}$ and $|\Lambda_n|\leq 1$. By the conditional independence of the bootstrap samples, we have
\begin{equation}\nonumber
\Lambda_n =  \left( 1 + \frac{1}{n}\Psi_n\right)^n .
\end{equation}
We now consider the characteristic functions of the complex-valued random variables $\Psi_n$.
Taking into account the independence of the $X$'s, we obtain that for any $\eta\in\mathbb{R}^2$,
{\small \begin{eqnarray}
\e{\textrm{e}^{i\eta_1 \textrm{Re}(\Psi_n)+i\eta_2 \textrm{Im}(\Psi_n)}} &=& \left(1 + \frac{1}{n}\function{\mathbb{P}}{\textrm{e}^{i\eta_1(\function{\cos}{\xi\epsilon}-1) + i\eta_2\function{\sin}{\xi\epsilon}}-1}\left(n\mathbb{P}(\ind{\zeta_0<Z<\zeta_0+\frac{t}{n}})\right)\right)^n\nonumber\\
\e{\textrm{e}^{i\eta_1 \textrm{Re}(\Psi_n)+i\eta_2 \textrm{Im}(\Psi_n)}} &\rightarrow& \textrm{e}^{tf(\zeta_0)\e{\textrm{e}^{i\eta_1(\function{\cos}{\xi\epsilon}-1) + i\eta_2\function{\sin}{\xi\epsilon}}-1}} = \e{\textrm{e}^{i\eta_1\textrm{Re}{\Psi_\xi^*}+i\eta_2\textrm{Im}{\Psi_\xi^*}}}. \nonumber
\end{eqnarray}}

\noindent Therefore, $\Psi_n\rightsquigarrow\Psi_\xi^*$ and, from the continuous mapping theorem, $\Lambda_n\rightsquigarrow e^{\Psi_\xi^*}$. Thus, Lebesgue's Dominated Convergence Theorem implies
\begin{equation}\label{ecsins2}
\psi_n(\xi) = \e{\Lambda_n} \rightarrow \e{\textrm{e}^{\Psi_\xi^*}} = \textrm{e}^{tf(\zeta_0)\left(\e{\textrm{e}^{\textrm{e}^{i\xi\epsilon}-1}}-1\right)}\ \ \forall\ \xi\in\mathbb{R}.
\end{equation}
With simpler arguments, we can also show that
\begin{equation}\label{ecsins3}
g_n(\xi) \rightarrow \textrm{e}^{tf(\zeta_0)\left(\textrm{e}^{(\textrm{e}^{i\xi\epsilon}-1)}-1\right)}\ \ \forall\ \xi\in\mathbb{R}.
\end{equation}
While (\ref{ecsins3}) is immediately recognized as the characteristic function of a compound Poisson process with rate $f(\zeta_0)$ and compounding distribution $\textrm{Poisson}(1)$, the characteristic function in (\ref{ecsins2}) can be shown to correspond to another compound Poisson process which can be written as
\begin{equation}\label{ecdreff1}
\sum_{1\leq j \leq \nu(t)} \epsilon_j\tau_j ,
\end{equation}
where $(\tau_n)_{n=1}^\infty\stackrel{i.i.d.}{\sim}\textrm{Poisson}(1)$, $(\nu(s))_{s\geq 0}$ is a Poisson process with rate $f(\zeta_0)$, and $(\tau_n)_{n=1}^\infty$, $(\epsilon_n)_{n=1}^\infty$ and $(\nu(s))_{s\geq 0}$ are mutually independent.

Therefore, the fifth and sixth components of $\tilde\Xi_n$ as defined in (\ref{ecsins}) converge, respectively, to a compound Poisson process with rate $f(\zeta_0)$ and Poisson(1) as compounding distribution and to the process described in (\ref{ecdreff1}). A similar analysis shows the analogous results for the third and fourth components of $\tilde\Xi_n$. The first and second components of $\tilde\Xi_n$ can easily be seen (by using the Lindeberg-Feller Central Limit Theorem) to be asymptotically normal with mean $0$ and variances $\sigma^2 \mathbb{P}(Z\leq \zeta_0)$ and $\sigma^2 \mathbb{P}(Z > \zeta_0)$, respectively.

All these facts indicate that the finite dimensional distributions of the limiting process of $\tilde\Xi_n$ match those of the process $\tilde \Xi$. In fact, we can proceed as in the proof of Proposition \ref{pg3} (i.e., proving tightness and convergence of the finite dimensional distributions using the Cramer-Wold device) to show $(i)$ and $(ii)$. For the sake of brevity, we omit the full technical details.

Then, arguing as in Proposition \ref{pg2} one can show that the sequence $(\sqrt{n}(\alpha_n^* - \alpha_0),\sqrt{n}(\beta_n^* - \beta_0),n(\zeta_n^* - \zeta_0))$ is stochastically bounded and then conclude that the (unconditional) asymptotic distribution of $(\sqrt{n}(\alpha_n^* - \alpha_0),\sqrt{n}(\beta_n^* - \beta_0),n(\zeta_n^* - \zeta_0)$ is that of $\sargmax_{h\in\mathbb{R}^3}\{ \tilde E^*(h)\}$, with $\tilde E^*(h)$ as defined in (\ref{ecdreff2}) and (\ref{ecdreff3}). For the sake of brevity we omit the full technical details of these arguments.

As $n(\zeta_n^* - \zeta_0) = n(\zeta_n^* - \hat \zeta_n) + n(\hat \zeta_n - \zeta_0)$, and if the ECDF bootstrap were consistent, the conditional distribution of $n(\zeta_n^* - \hat \zeta_n)$ (given the data) and the unconditional distribution of $n(\hat \zeta_n - \zeta_0)$ would have had the same weak limit. Then, as a consequence of Lemma 3.1 in \cite{sebawo} (also see Theorem 2.2 in \cite{kost}) the unconditional asymptotic distribution of $n(\zeta_n^* - \zeta_0)$ must be that of the sum of two independent copies of the asymptotic distribution of the $n(\hat{zeta}_n - \zeta_0)$. The result now follows. $\hfill \square$

\subsubsection{Proof of Lemma \ref{l7}}\label{prueba11}
Let $\mathbb{G}_n$ be the ECDF of $\epsilon_1,\ldots,\epsilon_n$. We first observe that
\[ \int e^{i\xi x}d\mathbb{P}_n^\epsilon (x) = e^{-i\xi\bar{\epsilon}_n}\function{\mathbb{P}_n}{e^{i\xi\widetilde{\epsilon}_n}}\]
and hence, for any $\xi\in\mathbb{R}$ with $|\xi|\leq\eta$ we have,
\begin{eqnarray}
 \left|\int e^{i\xi x}d\mathbb{P}_n^\epsilon (x) -  e^{-i\xi\bar{\epsilon}_n}\int e^{i\xi x}d\mathbb{G}_n (x)\right| &=& \left|\function{\mathbb{P}_n}{e^{i\xi\widetilde{\epsilon}_n}} -\function{\mathbb{P}_n}{e^{i\xi\epsilon}}\right| \nonumber\\
 &\leq& |\eta|\function{\mathbb{P}_n}{|\tilde{\epsilon}_n - \epsilon|}\nonumber
\end{eqnarray}
but $\function{\mathbb{P}_n}{|\tilde{\epsilon}_n - \epsilon|}$ is bounded from above by
\[ |\hat{\alpha}_n - \alpha_0| + \left(|\alpha_0|+|\beta_0|\right)|\mathbb{P}_n (\ind{Z\leq\hat{\zeta}_n}-\ind{Z\leq\zeta_0})| + |\hat{\beta}_n - \beta_0|\]
which goes to zero almost surely as consequence of Lemmas \ref{l3} and \ref{l12} $(iv)$, with $\mathbb{Q}_n = \mathbb{P}_n$. Thus,
\[ \sup_{|\xi|\leq \eta}\left\{\left|\int e^{i\xi x}d\mathbb{P}_n^\epsilon (x) -  e^{-i\xi\bar{\epsilon}_n}\int e^{i\xi x}d\mathbb{G}_n (x)\right|\right\}\cas 0\]
and $(i)$ follows immediately because $\bar{\epsilon}_n=\mathbb{P}_n (\tilde{\epsilon}_n)\cas 0 $ and $\mathbb{G}_n$ converges to $G$ in total variation distance with probability one. The second assertion is seen to be true at once because $G$ is assumed to be continuous and condition $(i)$ implies that the characteristic functions of $\mathbb{P}_n^\epsilon$ converge to the characteristic function of $G$ on the entire real line with probability one.
Statements $(ii)$ and $(iii)$ are straightforward: On the one hand, we have shown that conditions (I)-(IV) hold for the ECDF, so Lemma \ref{l12} implies that $\displaystyle \int x^2 d\mathbb{P}_n^\epsilon(x)  = \mathbb{P}_n(\tilde{\epsilon}_n^2) - \mathbb{P}_n(\tilde{\epsilon}_n)^2 \cas \sigma^2$. On the other hand,
\begin{eqnarray}
\left| \int |x|d\mathbb{P}_n^\epsilon - \int |\epsilon|d\mathbb{P}_n\right| &=& |\mathbb{P}_n (|\tilde{\epsilon}_n - \bar{\epsilon}_n| - |\epsilon|)|\nonumber\\
 &\leq& \mathbb{P}_n (|\tilde{\epsilon}_n-\epsilon|) + |\bar{\epsilon}_n| \cas 0. \nonumber
\end{eqnarray}

To prove $(iv)$, we first notice that
\[ \int |x|^3 d \mathbb{P}_n^\epsilon (x) \leq |\bar{\epsilon}_n|^3 + 3 |\bar{\epsilon}_n|^2 \emp{|\tilde{\epsilon}_n|} + 3|\bar{\epsilon}_n|\emp{\tilde{\epsilon}_n^2} + \emp{|\tilde{\epsilon}_n|^3}.\]
Then, from Lemma \ref{l13} all but the last summand on the right-hand side converge almost surely. Hence, it suffices to show that  $\lsup\emp{|\tilde{\epsilon}_n|^3}<\infty$ w. p. 1. With this in mind, let $L_n = |\alpha_0| + |\hat{\alpha}_n|+|\beta_0|+|\hat{\beta}_n|$ and observe that
\[\emp{|\tilde{\epsilon}_n|^3} \leq \emp{|\epsilon|^3} + 3\emp{|\epsilon|^2}L_n + 3\emp{|\epsilon|}L_n^2 + L_n^3.
\]
The result then is an immediate consequence of the third moment assumption on $\epsilon$, the strong law of large numbers and the almost sure convergence of the least squares estimators. $\hfill \square$

\subsubsection{Proof of Proposition \ref{pfd1}}\label{prueba17}
Just as in the proof of Proposition \ref{pg1} we have
\begin{eqnarray}
-\frac{1}{n}\sum_{k=1}^n (\tilde{\epsilon}_{n,j}^*)^2 &=& R_n(\hat{\theta}_n)\nonumber\\
 &\leq& R_n(\theta_n^*) \leq -\frac{1}{n}\sum_{j=1}^n (\tilde{\epsilon}_{n,j}^* + \hat{\alpha}_n - \alpha_n^*)^2\ind{Z_j<a} + (\tilde{\epsilon}_{n,j}^* + \hat{\beta}_n - \beta_n^*)^2\ind{Z_j>b}\nonumber
\end{eqnarray}
from which we can see that
\[ |\gamma_n^*-\gamma_n|^2 \mathbb{P}_n (Z<a)\land\mathbb{P}_n (Z>b) \leq \]
\[\frac{1}{n}\sum_{j=1}^n (\tilde{\epsilon}_{n,j}^*)^2 \ind{a\leq Z_j\leq b} + \frac{2}{n}|\gamma_n^*-\gamma_n|\left(\left|\sum_{j=1}^n \tilde{\epsilon}_{n,j}^*\ind{Z_j<a}\right|+ \left|\sum_{j=1}^n \tilde{\epsilon}_{n,j}^*\ind{Z_j>b}\right|\right).\]
But the first of the terms on the right-hand side of the previous inequality is conditionally bounded in $\mathbb{L}_1$ (an upper bound for the conditional expectations is $\displaystyle \sup_{n\in\mathbb{N}} \left\{\int x^2 d\mathbb{P}_n^\epsilon(x) \right\}<\infty$). The terms $\displaystyle \frac{1}{n} \sum_{j=1}^n \tilde{\epsilon}_{n,j}^*\ind{Z_j<a}$ and $\displaystyle  \frac{1}{n} \sum_{j=1}^n \tilde{\epsilon}_{n,j}^*\ind{Z_j>b}$ both have zero conditional expectation and conditional variances equal to $\displaystyle \frac{1}{n} \mathbb{P}_n(Z<a)\int x^2 d\mathbb{P}_n^\epsilon(x)$ and $\displaystyle \frac{1}{n}\mathbb{P}_n(Z>b)\int x^2 d\mathbb{P}_n^\epsilon(x)$ respectively. So we have that
\[ \left|\frac{1}{n}\sum_{j=1}^n \tilde{\epsilon}_{n,j}^*\ind{Z_j<a}\right|+ \left|\frac{1}{n}\sum_{j=1}^n \tilde{\epsilon}_{n,j}^*\ind{Z_j>b}\right|\ccipas 0. \]
Thus,
\begin{equation}\label{ec59}
|\gamma_n^* - \hat{\gamma}_n| = O_{\mathbf{P}_\mathfrak{X}}(1) \textrm{ almost surely.}
\end{equation}
Now, let $Z_{(k)}$ be the $k$-th order statistic from the sample $(Z_1,\ldots,Z_n)$ and $r_k$ a number such that $Z_{(k)} = Z_{r_k}$. For any $\zeta\in[a,b]$ define $m_\zeta = \max\{1\leq j\leq n: Z_{(j)} \leq \zeta\land\hat{\zeta}_n\}$ and observe that we have
\begin{equation}\label{ec60}
\frac{1}{n}\sum_{j=1}^n \tilde{\epsilon}_{n,j}^*\ind{Z_j\leq \zeta\land\hat{\zeta}_n} = \frac{1}{n}\sum_{1\leq j \leq m_\zeta}\tilde{\epsilon}_{n,r_j}^*,
\end{equation}
and thus
\begin{equation}\label{ec60bis}
\sup_{\zeta\in[a,b]}\left\{\left|\frac{1}{n}\sum_{j=1}^n \tilde{\epsilon}_{n,j}^*\ind{Z_j\leq \zeta\land\hat{\zeta}_n}\right|\right\} \leq \max_{1\leq k\leq n}\left\{\frac{1}{n}\left|\sum_{1\leq j \leq k}\tilde{\epsilon}_{n,r_j}^*\right|\right\}.
\end{equation}
But the indexes $r_k$ and the order statistics are functions of $Z_1,\ldots,Z_n$ and therefore $\mathfrak{X}$-measurable. Hence, conditionally, $\displaystyle \sum_{1\leq j \leq k}\tilde{\epsilon}_{n,r_j}^*\ind{Z_{r_j}\leq \zeta\land\hat{\zeta}_n}$ is a square integrable martingale with zero expectation. Hence, from Doob's submartingale inequality (see \cite{pwm}, Theorem 14.6, page 137) we get
\[ \cp{\max_{1\leq k\leq n}\left\{\frac{1}{n}\left|\sum_{1\leq j \leq k}\tilde{\epsilon}_{n,r_j}^*\right|\right\} > \rho} \leq \frac{1}{n\rho^2}\mathbb{P}_n(\tilde{\epsilon}_n^2)\]
and consequently, equations (\ref{ec60}) and (\ref{ec60bis}) show that
\begin{equation}\label{ec61}
\cp{\left\|\frac{1}{n}\sum_{j=1}^n \tilde{\epsilon}_{n,j}^*\ind{Z_j\leq (\cdot)\land\hat{\zeta}_n}\right\|_{[a,b]} >\rho}\leq \frac{1}{\rho^2 n}\mathbb{P}_n(\tilde{\epsilon}_n^2) \cas 0.
\end{equation}
Similar arguments give that (\ref{ec61}) is also true if we replace $\ind{Z_j \leq (\cdot)\land\hat{\zeta}_n}$ by any of $\ind{ (\cdot)<Z_j \leq\hat{\zeta}_n}$, $\ind{\hat{\zeta}_n<Z_j \leq (\cdot)}$ or $\ind{Z_j > (\cdot)\lor\hat{\zeta}_n}$. Now, if we write $R_n$ like
\begin{eqnarray}
R_n(\theta) = -\mathbb{P}_n^*(\tilde{\epsilon}_n^2) -\frac{2}{n}(\hat{\alpha}_n-\alpha)\sum_{j=1}^n\tilde{\epsilon}_{n,j}^*\mathbf{1}_{Z_j\leq \zeta\land\hat{\zeta}_n} - (\hat{\alpha}_n-\alpha)^2\mathbb{P}_n(\mathbf{1}_{Z\leq \zeta\land\hat{\zeta}_n}) \nonumber\\
- \frac{2}{n}(\hat{\beta}_n-\alpha)\sum_{j=1}^n\tilde{\epsilon}_{n,j}^*\mathbf{1}_{\hat{\zeta}_n < Z\leq \zeta} - (\hat{\beta}_n-\alpha)^2\mathbb{P}_n(\mathbf{1}_{\hat{\zeta}_n < Z\leq \zeta}) \nonumber\\
-\frac{2}{n}(\hat{\alpha}_n-\beta)\sum_{j=1}^n\tilde{\epsilon}_{n,j}^*\mathbf{1}_{\zeta < Z\leq \hat{\zeta}_n} - (\hat{\alpha}_n-\beta)^2\mathbb{P}_n(\mathbf{1}_{\zeta < Z\leq \hat{\zeta}_n}) \nonumber\\
- \frac{2}{n}(\hat{\beta}_n-\beta)\sum_{j=1}^n\tilde{\epsilon}_{n,j}^*\mathbf{1}_{Z> \zeta\lor\hat{\zeta}_n} -
(\hat{\beta}_n-\beta)^2\mathbb{P}_n(\mathbf{1}_{Z> \zeta\lor\hat{\zeta}_n}), \label{ec62}
\end{eqnarray}
$(ii)$ follows immediately from (\ref{ec61}), applied for all the four possible types of indicator functions. Note that the four terms on the far right of all the rows in the previous display vanish when we subtract $M_n$ from $R_n$. Lemma \ref{l3} shows that $(ii)$ implies $(i)$, while Corollary 3.2.3 $(ii)$, page 287, of \cite{vw} together with (\ref{ec59}) allows one to derive $(iii)$ from $(i)$ and $(ii)$. $\hfill \square$

\subsubsection{Proof of Lemma \ref{pfd3}}\label{prueba19}
The proof is analogous to the proof of Lemma \ref{l21}. We again consider the number $h_*>0$ defined in the statement of Lemma \ref{l19} and take $K\subset\mathbb{R}^3$ to be any compact rectangle containing the point $(0,0,h_*)$. To prove the theorem it suffices to show that the sequence $(\hat{E}_n(0,0,h_3))_{n=1}^\infty$ does not have a weak limit in probability whenever $h_3\geq h_*$ and $(0,0,h_3)\in K$. But in view of Lemma \ref{l19} this is straightforward because the (conditional) characteristic function of $\hat{E}_n(0,0,h_3)$ is given by
\[ \left(\int e^{i2(\hat{\alpha}_n - \hat{\beta}_n)\xi x - i\xi(\hat{\alpha}_n - \hat{\beta}_n)^2} d\mathbb{P}_n^\epsilon(x)\right)^{n\mathbb{P}_n(\hat{\zeta}_n<Z\leq \hat{\zeta}_n + \frac{h_3}{n})}.\]
and Lemma \ref{l7} and the strong consistency of the least squares estimator imply that
\[ \int e^{i2(\hat{\alpha}_n - \hat{\beta}_n)\xi x - i\xi(\hat{\alpha}_n - \hat{\beta}_n)^2} d\mathbb{P}_n^\epsilon(x)\cas e^{ - i\xi(\alpha_0 - \beta_0)^2}\function{\varphi}{2(\alpha_0 - \beta_0)\xi}.\]
Thus, for all $\xi$ in a neighborhood of the origin, this characteristic function will converge if and only if $n\mathbb{P}_n(\hat{\zeta}_n<Z\leq \hat{\zeta}_n + \frac{h_3}{n})$ converges. We know that this is not the case from Lemma \ref{l19}. $\hfill \square$

\subsubsection{Proof of Proposition \ref{cs3}}\label{prueba15}
We will show that conditions (I)-(V) in Section \ref{pgs} hold w.p. 1 for the bootstrap measures arising in this scheme. Note that (IV) is a consequence of Lemma \ref{l3}. That $\left\|\mathbb{Q}_n - \mathbb{P}\right\|_\mathcal{F}\cas 0$ follows immediately from the fact that $\|\hat{F}_n - F\|_\infty \cas 0$. Now, for any $g = y\psi\in\mathcal{G}$ with $\psi\in\mathcal{F}$, we have
\begin{eqnarray}
 \mathbb{Q}_n (g) &=& \hat{\alpha}_n \mathbb{Q}_n(\ind{Z\leq\hat{\zeta}_n}\psi) + \hat{\beta}_n \mathbb{Q}_n(\ind{Z>\hat{\zeta}_n}\psi), \nonumber\\
 \mathbb{P} (g) &=& \alpha_0 \mathbb{P}(\ind{Z\leq\zeta_0}\psi) + \beta_0 \mathbb{P}(\ind{Z>\zeta_0}\psi), \nonumber
\end{eqnarray}
from which we see that
\begin{eqnarray}
\left\|\mathbb{Q}_n - \mathbb{P}\right\|_\mathcal{G} & \leq & \left(\left|\hat{\alpha}_n - \alpha_0\right| + \left|\hat{\beta}_n - \beta_0\right|\right) + (|\alpha_0| + |\beta_0|)\left\|\mathbb{Q}_n - \mathbb{P}\right\|_\mathcal{F} \nonumber \\
& & \;\;\; + \; (|\alpha_0| + |\beta_0|) \int_{\mathbb{R}}|\ind{z\leq \hat{\zeta}_n} - \ind{z\leq \zeta_0}|\hat{f}_n(z)dz. \nonumber
\end{eqnarray}
Lebesgue's dominated convergence theorem shows that the last integral goes almost surely to zero and the strong consistency of the least squares estimators and property (I) now yields  $\left\|\mathbb{Q}_n - \mathbb{P}\right\|_\mathcal{G}\cas 0$. Finally, we can write any $h\in\mathcal{H}$ in the form $h=y^2\psi$ for some $\psi\in\mathcal{F}$. Using this representation we obtain,
\begin{eqnarray}
 \mathbb{Q}_n (h) &=& \hat{\alpha}_n^2 \mathbb{Q}_n(\ind{Z\leq\hat{\zeta}_n}\psi) + \hat{\beta}_n^2 \mathbb{Q}_n(\ind{Z>\hat{\zeta}_n}\psi) +  \mathbb{P}_n^\epsilon (\tilde{\epsilon}_n^2)\mathbb{Q}_n(\psi), \nonumber\\
 \mathbb{P} (h) &=& \alpha_0^2 \mathbb{P}(\ind{Z\leq\zeta_0}\psi) + \beta_0^2 \mathbb{P}(\ind{Z>\zeta_0}\psi) + \sigma^2\mathbb{P}(\psi) ,\nonumber
\end{eqnarray}
and the triangle inequality then implies that
\begin{eqnarray*}
&& \left\| \mathbb{Q}_n - \mathbb{P}\right\|_\mathcal{H} \leq  (|\hat{\alpha}_n^2 - \alpha_0^2| + |\hat{\beta}_n^2 - \beta_0^2|) + (\alpha_0^2 + \beta_0^2 + \sigma^2)\left\|\mathbb{Q}_n - \mathbb{P}\right\|_\mathcal{F}\\
& & + \; |\mathbb{P}_n^\epsilon (\tilde{\epsilon}_n^2) - \mathbb{P}(\epsilon^2)| + (\alpha_0^2 + \beta_0^2) \int_{\mathbb{R}}|\ind{z\leq \hat{\zeta}_n} - \ind{z\leq \zeta_0}|\hat{f}_n(z)dz \cas 0.
\end{eqnarray*}
It remains to show (V). Observe that (\ref{rccs2}) and (\ref{rccs3}) hold automatically because under $\mathbb{Q}_n$, $\tilde{\epsilon}_n$ and $Z$ are independent. Hence, we only require to show that (\ref{rccs1}) holds w.p. 1. As (\ref{ascci}) holds, we have
\[ \inf_{\zeta\in [c,d]}\left\{\hat{f}_n(\zeta)\right\}\cas \inf_{\zeta\in [c,d]}\left\{f(\zeta)\right\} > 0.\]
The mean value theorem implies that for any $\zeta,\xi\in [c,d]$, there is $\vartheta\in [0,1]$ such that
$|\hat{F}_n(\zeta) - \hat{F}_n(\xi)| = |\xi-\zeta| \hat{f}_n(\zeta + \vartheta (\xi-\zeta)) $. It follows that for $\eta > 0$ small enough, \[ \inf_{0< |\zeta-\hat{\zeta}_n| < \delta^2} \left\{ \frac{1}{|\zeta-\hat{\zeta}_n|} |\hat{F}_n(\zeta) - \hat{F}_n(\hat{\zeta}_n)|\right\}\geq \inf_{\zeta\in [c,d]}\left\{\hat{f}_n(\zeta)\right\}\ \ \forall\ n\in\mathbb{N} \] and consequently (V) holds w.p.1 for all $\delta < \eta$ for all large $n$. $\hfill \square$

\subsubsection{Proof of Proposition \ref{cos3}}\label{prueba16}
We already know that conditions (I)-(V) hold w.p. 1. Condition (VII) holds automatically because $Z$ and $\tilde{\epsilon}_n$ are independent under $\mathbb{Q}_n$ and $\mathbb{Q}_n(\tilde{\epsilon}_n)=0$. Lemma \ref{l7} $(v)$ implies that condition (VIII) holds a.s. It remains to prove (VI).

Write $I = [c,d]$ and consider the sequence of events $\left\{A_N\right\}_{N\in\mathbb{N}}$ given by
\[ A_N = \left[ \hat{\zeta}_n -\frac{\delta}{n},\hat{\zeta}_n +\frac{\eta}{n}\in I, \mbox{ almost always},\ \forall\ \delta,\eta\in(0,N) \right]\cap\left[\|\hat{f}_n-f\|_I\rightarrow 0\right].\]
Fix $N \in \mathbb{N}$, let $\psi$ be the function $\psi(x) = e^{i\xi x}$ for some $\xi\in\mathbb{R}$ or the function $\psi(x) = |x|^p$, $p=1,2$, and $\eta,\delta>0$ be any positive real numbers smaller than $N$. Then,
\[m_n\mathbb{Q}_n(\psi(\tilde{\epsilon}_n)\ind{\zeta_n - \frac{\delta}{n} < Z \leq \zeta_n + \frac{\eta}{n}}) = n \mathbb{P}_n^\epsilon\left(\psi\right) \int_{\hat{\zeta}_n- \frac{\delta}{n}}^{\hat{\zeta}_n + \frac{\eta}{n}} \hat{f}_n(x)dx. \]
Lemma \ref{l7} implies that $\mathbb{P}_n^\epsilon\left(\psi\right)\cas \mathbb{P}\left(\psi(\epsilon)\right)$. And, when $A_N$ holds, we also have
\[ n\left|\int_{\hat{\zeta}_n- \frac{\delta}{n}}^{\hat{\zeta}_n + \frac{\eta}{n}} \hat{f}_n(x)dx - \int_{\hat{\zeta}_n- \frac{\delta}{n}}^{\hat{\zeta}_n + \frac{\eta}{n}} f(x)dx\right| \leq 2N\left\|\hat{f}_n - f\right\|_{[c,d]}\rightarrow 0.\]
Hence, condition (VI) holds for all $0<\delta,\eta<N$ on $A_N$. But the strong consistency of the least squares estimators and the conditions on $\hat{f}_n$ imply that each of these events have probability one. Therefore, $\p{\cap_{N\in\mathbb{N}} A_N} = 1$. Hence, condition (VI) holds w.p.1 and the result follows from an application of Proposition \ref{pg3}. $\hfill \square$

\subsubsection{Proof of Proposition \ref{cs5}}\label{prueba20}
Since $\mathbb{Q}_n$ is just the ECDF, the validity of conditions (I)-(IV) follows from the result established for the regular ECDF bootstrap and Lemma \ref{l3}. (VIII) is a consequence of the strong law of large numbers. It remains to show (V)-(VII).

We start with (VI). First observe that $m_n\mathbb{P}(\psi(\epsilon)\ind{\zeta_0  - \frac{\delta}{m_n} < Z \leq \zeta_0 + \frac{\eta}{m_n}})\rightarrow (\delta+\eta) f(\zeta_0)\mathbb{P}(\psi(\epsilon))$. We will proceed as follows: we will first use this simple observation just made to show that the following equations are true,

\begin{eqnarray}
m_n\left\| \mathbb{P}_n(\psi(\epsilon)\ind{\zeta_0 - \frac{(\cdot)}{m_n}< Z \leq\zeta_0}) - (\cdot)\mathbb{P}(\psi(\epsilon))f(\zeta_0) \right\|_K &\cip& 0 \label{ecuacionfinal5}\\
m_n\left\| \mathbb{P}_n(\psi(\epsilon)\ind{\zeta_0< Z \leq\zeta_0 + \frac{(\cdot)}{m_n}}) - (\cdot)\mathbb{P}(\psi(\epsilon))f(\zeta_0)\right\|_K &\cip& 0 \label{ecuacionfinal6}\\
m_n\left\| \mathbb{P}_n(\psi(\tilde{\epsilon}_n)\ind{\hat{\zeta}_n - \frac{(\cdot)}{m_n}< Z \leq\hat{\zeta}_n}) - \mathbb{P}_n(\psi(\epsilon)\ind{\zeta_0 - \frac{(\cdot)}{m_n}< Z \leq\zeta_0})\right\|_K &\cip& 0 \label{ecuacionfinal1}\\
m_n\left\| \mathbb{P}_n(\psi(\tilde{\epsilon}_n)\ind{\hat{\zeta}_n< Z\leq\hat{\zeta}_n + \frac{(\cdot)}{m_n}}) - \mathbb{P}_n(\psi(\epsilon)\ind{\zeta_0< Z \leq\zeta_0 + \frac{(\cdot)}{m_n}})\right\|_K &\cip& 0 \label{ecuacionfinal2}
\end{eqnarray}
for any compact interval $K\subset\mathbb{R}$. All these facts put together will give
\begin{eqnarray}
m_n\left\| \mathbb{P}_n(\psi(\tilde{\epsilon}_n)\ind{\hat{\zeta}_n - \frac{(\cdot)}{m_n}< Z \leq\hat{\zeta}_n}) - (\cdot)\mathbb{P}(\psi(\epsilon))f(\zeta_0) \right\|_K &\cip& 0 \label{ecuacionfinal3}\\
m_n\left\| \mathbb{P}_n(\psi(\tilde{\epsilon}_n)\ind{\hat{\zeta}_n< Z \leq\hat{\zeta}_n + \frac{(\cdot)}{m_n}}) - (\cdot)\mathbb{P}(\psi(\epsilon))f(\zeta_0)\right\|_K &\cip& 0 \label{ecuacionfinal4}
\end{eqnarray}
for any compact interval $K\subset\mathbb{R}$. Having achieved this, we will be able to conclude that (VI) holds in probability. For if (\ref{ecuacionfinal3}) and (\ref{ecuacionfinal4}) are both true, we can take an increasing sequence of compacts $(K_n)_{n=1}^\infty$ whose union is $\mathbb{R}$ and then for any subsequence $(n_k)_{k=1}^\infty$ find a further subsequence $(n_{k_s})_{s=1}^\infty$ such that
{\footnotesize
\begin{eqnarray}
\p{m_{n_{k_s}}\left\| \mathbb{P}_{n_{k_s}}(\psi(\tilde{\epsilon}_{n_{k_s}})\ind{\hat{\zeta}_{n_{k_s}} - \frac{(\cdot)}{m_{n_{k_s}}}< Z \leq\hat{\zeta}_{n_{k_s}}}) - (\cdot)\mathbb{P}(\psi(\epsilon))f(\zeta_0) \right\|_{K_{s}} > \frac{1}{s}}&<&\frac{1}{s^2} \nonumber\\
\p{m_{n_{k_s}}\left\| \mathbb{P}_{n_{k_s}}(\psi(\tilde{\epsilon}_{n_{k_s}})\ind{\hat{\zeta}_{n_{k_s}}< Z \leq \hat{\zeta}_{n_{k_s}} + \frac{(\cdot)}{m_{n_{k_s}}} }) - (\cdot)\mathbb{P}(\psi(\epsilon))f(\zeta_0) \right\|_{K_{s}} > \frac{1}{s}}&<&\frac{1}{s^2}. \nonumber
\end{eqnarray}
}
The Borel-Cantelli Lemma will then imply that (VI) holds almost surely for the subsequence $(n_{k_s})_{s=1}^\infty$. Therefore, it suffices to show (\ref{ecuacionfinal5}), (\ref{ecuacionfinal6}), (\ref{ecuacionfinal1}) and (\ref{ecuacionfinal2}).

First consider the case where $\psi(\cdot) = |\cdot|$ and a positive number $\eta>0$. Let $t\in\mathbb{R}$ and write
\[ r_n = n\function{\mathbb{P}}{e^{i\frac{m_n}{n}t|\epsilon|} - 1 -\frac{m_n}{n}t|\epsilon|}\mathbb{P}(\ind{\zeta_0<Z\leq \zeta_0 + \frac{\eta}{m_n}}).\]
Then, $|r_n|\leq t^2 \sigma^2 \frac{m_n}{n}m_n\mathbb{P}(\ind{\zeta_0<Z\leq \zeta_0 + \frac{\eta}{m_n}}) \rightarrow 0$. The characteristic function of $\\$ $m_n\mathbb{P}_n (|\epsilon| \ind{\zeta_0<Z\leq + \frac{\eta}{m_n}})$ can be written as
\[ \varphi_n (t) = \left(1+i\frac{m_n}{n}t\mathbb{P}(|\epsilon|)\mathbb{P}( \ind{\zeta_0<Z\leq \zeta_0 + \frac{\eta}{m_n}}) + \frac{r_n}{n}\right)^n\rightarrow e^{it\eta\mathbb{P}(|\epsilon|)f(\zeta_0)}\]
and therefore
\begin{equation}\nonumber
m_n\mathbb{P}_n (|\epsilon| \ind{\zeta_0<Z\leq \zeta_0 + \frac{\eta}{m_n}})\cip \eta f(\zeta_0)\mathbb{P}(|\epsilon|).
\end{equation}
But
\[\sup_{n\in\mathbb{N}}\left\{\e{m_n\left\|\mathbb{P}_n (|\epsilon| \ind{\zeta_0<Z\leq \zeta_0 + \frac{(\cdot)}{m_n}})\right\|_{[0,\eta]}}\right\} <\infty\]
and hence the sequence of processes $\left(m_n\mathbb{P}_n (|\epsilon| \ind{\zeta_0<Z\leq \zeta_0 + \frac{(\cdot)}{m_n}})\right)_{n=1}^\infty$ is tight in $\mathcal{D}_{[0,\eta]}$.
It follows that
\[ m_n\mathbb{P}_n (|\epsilon| \ind{\zeta_0<Z\leq \zeta_0 + \frac{(\cdot)}{m_n}})\rightsquigarrow
(\cdot) f(\zeta_0)\mathbb{P}(|\epsilon|)\ \ \textrm{in }\ \mathcal{D}_{[0,\eta]}\]
but since the limiting process is continuous and deterministic we actually obtain
\begin{equation}\label{ec65}
\left\|m_n\mathbb{P}_n (|\epsilon| \ind{\zeta_0<Z\leq \zeta_0 + \frac{\cdot}{m_n}})-
(\cdot) f(\zeta_0)\mathbb{P}(|\epsilon|)\right\|_{[0,\eta]}\cip 0.
\end{equation}
And with similar arguments one can also prove that
\begin{equation}\label{ec66}
\left\|m_n\mathbb{P}_n (|\epsilon| \ind{\zeta_0 - \frac{(\cdot)}{m_n}<Z\leq \zeta_0}) -(\cdot)f(\zeta_0)\mathbb{P}(|\epsilon|)\right\|_{[0,\eta]}\cip 0.
\end{equation}
Pick a positive number $\eta>0$. Taking into account that $\epsilon\ind{\zeta_0 < Z \leq \zeta_0 + \frac{\eta}{m_n}} = (y-\beta_0)\ind{\zeta_0 < Z \leq \zeta_0 + \frac{\eta}{m_n}}$ and the analogous result for $\tilde{\epsilon}_n$ with $\hat{\zeta}_n$ and $\hat{\beta}_n$ instead of $\zeta_0$ and $\beta_0$ we see that
\[ m_n \left\| \mathbb{P}_n (|\tilde{\epsilon}_n|\ind{\hat{\zeta}_n < Z \leq \hat{\zeta}_n + \frac{(\cdot)}{m_n}})  - \mathbb{P}_n (|\epsilon|\ind{\zeta_0 < Z \leq \zeta_0 + \frac{(\cdot)}{m_n}})\right\|_{[0,\eta]} \leq\]
\[ m_n \left\| \mathbb{P}_n \left(|Y-\hat{\beta}_n|\left(\ind{\hat{\zeta}_n < Z \leq \hat{\zeta}_n + \frac{(\cdot)}{m_n}} - \ind{\zeta_0 < Z \leq \zeta_0 + \frac{(\cdot)}{m_n}}\right)\right)\right\|_{[0,\eta]} +\]
 \[ m_n \left\|\mathbb{P}_n \left(\left(|
Y-\hat{\beta}_n| - |Y-\beta_0|\right)\ind{\zeta_0 < Z \leq \zeta_0 + \frac{(\cdot)}{m_n}}\right)\right\|_{[0,\eta]}\]
and consequently
\[ m_n \left\| \mathbb{P}_n (|\tilde{\epsilon}_n|\ind{\hat{\zeta}_n < Z \leq \hat{\zeta}_n + \frac{(\cdot)}{m_n}})  - \mathbb{P}_n (|\epsilon|\ind{\zeta_0 < Z \leq \zeta_0 + \frac{(\cdot)}{m_n}})\right\|_{[0,\eta]} \leq\]
\[ m_n \left\| \mathbb{P}_n \left(|Y-\beta_0|\left(\ind{\hat{\zeta}_n < Z \leq \hat{\zeta}_n + \frac{(\cdot)}{m_n}} - \ind{\zeta_0 < Z \leq \zeta_0 + \frac{(\cdot)}{m_n}}\right)\right)\right\|_{[0,\eta]} +\]
\[ |\hat{\beta}_n - \beta_0|m_n \left\| \mathbb{P}_n \left(\ind{\hat{\zeta}_n < Z \leq \hat{\zeta}_n + \frac{(\cdot)}{m_n}} - \ind{\zeta_0 < Z \leq \zeta_0 + \frac{(\cdot)}{m_n}}\right)\right\|_{[0,\eta]} +\]
\begin{equation}\label{ec67}
|\hat{\beta}_n - \beta_0| m_n \left\|\mathbb{P}_n \left(\ind{\zeta_0 < Z \leq \zeta_0 + \frac{(\cdot)}{m_n}}\right)\right\|_{[0,\eta]}.
\end{equation}
We will show that each of the terms on the right-hand side of (\ref{ec67}) goes to zero in probability. Since $n(\hat{\zeta}_n - \zeta_0) = O_\mathbf{P} (1)$, we know that for any $\delta>0$ there is $R_\delta>0$ such that
$\p{n|\hat{\zeta}_n - \zeta_0|>R_\delta}<\delta$. Then,
\[ \p{m_n \left\| \mathbb{P}_n \left(|Y-\beta_0|\left(\ind{\hat{\zeta}_n < Z \leq \hat{\zeta}_n + \frac{(\cdot)}{m_n}} - \ind{\zeta_0 < Z \leq \zeta_0 + \frac{(\cdot)}{m_n}}\right)\right)\right\|_{[0,\eta]} > \delta} \leq \delta + \]
\[ \p{m_n  \mathbb{P}_n \left(|\epsilon|\ind{\zeta_0 - \frac{R_\delta}{n}<Z\leq\zeta_0}\right) > \frac{\delta}{3}} +\]
\[ \p{m_n \left\| \mathbb{P}_n \left(|\epsilon|\ind{\zeta_0 < Z \leq \zeta_0 + \frac{(\cdot)}{m_n} + \frac{R_\delta}{n}}\right)\right\|_{[0,\eta]} > \frac{\delta}{3}} + \]
\begin{equation}\nonumber
\p{m_n  |\alpha_0 - \beta_0|\mathbb{P}_n \left(\ind{\zeta_0 - \frac{R_\delta}{n}<Z\leq\zeta_0}\right) > \frac{\delta}{3}}
\end{equation}
but from equations (\ref{ec65}) and (\ref{ec66}), and the fact that $\frac{m_n}{n}\rightarrow 0$, we actually get that all the terms of the right-hand side are asymptotically smaller than $\frac{\delta}{3}$. Thus,
\begin{equation}\label{ec68}
\lsup_{n\rightarrow\infty}\p{m_n \left\| \mathbb{P}_n \left(|Y-\beta_0|\left(\ind{\hat{\zeta}_n < Z \leq \hat{\zeta}_n + \frac{(\cdot)}{m_n}} - \ind{\zeta_0 < Z \leq \zeta_0 + \frac{(\cdot)}{m_n}}\right)\right)\right\|_{[0,\eta]} > \delta}<2\delta.
\end{equation}
An argument similar in spirit to the one just employed gives
\begin{equation}\label{ec69}
\lsup_{n\rightarrow\infty} \p{m_n \left\| \mathbb{P}_n \left(\left(\ind{\hat{\zeta}_n < Z \leq \hat{\zeta}_n + \frac{(\cdot)}{m_n}} - \ind{\zeta_0 < Z \leq \zeta_0 + \frac{(\cdot)}{m_n}}\right)\right)\right\|_{[0,\eta]}>\delta} < \delta
\end{equation}
while equation (\ref{ec70}), for $\xi = 0$, and the strong consistency of the least squares estimator give
\begin{equation}\nonumber
|\hat{\beta}_n - \beta_0| m_n \left\|\mathbb{P}_n \left(\ind{\zeta_0 < Z \leq \zeta_0 + \frac{(\cdot)}{m_n}}\right)\right\|_{[0,\eta]}\cip 0.
\end{equation}
Then, combining  the last identity with (\ref{ec67}), (\ref{ec68}) and (\ref{ec69}) we get
\begin{equation}\nonumber
 \lim_{\delta\rightarrow 0} \lsup_{n\rightarrow\infty}m_n \p{\left\| \mathbb{P}_n (|\tilde{\epsilon}_n|\ind{\hat{\zeta}_n < Z \leq \hat{\zeta}_n + \frac{(\cdot)}{m_n}})  - \mathbb{P}_n (|\epsilon|\ind{\zeta_0 < Z \leq \zeta_0 + \frac{(\cdot)}{m_n}})\right\|_{[0,\eta]}>\delta} = 0.
\end{equation}
Completely analogous arguments prove that
\[m_n \left\| \mathbb{P}_n (|\tilde{\epsilon}_n|\ind{\hat{\zeta}_n - \frac{(\cdot)}{m_n} < Z \leq \hat{\zeta}_n})  - \mathbb{P}_n (|\epsilon|\ind{\zeta_0 - \frac{(\cdot)}{m_n} < Z \leq \zeta_0})\right\|_{[0,\eta]}\cip 0.\]
Since $\eta>0$ was arbitrarily chosen, we have shown (IV) for $\psi(\cdot) = |\cdot|$. The case $\psi = |\cdot|^2$ is proven in a very similar manner. For the sake of brevity, we omit the proof.

Now, we consider the case where $\psi(x) = e^{i\xi x}$ for some $\xi\in\mathbb{R}$. Again, fix $\eta>0$. We will proceed in the same way as before. Let $t\in\mathbb{R}$ and write
\[ \rho_n = n\function{\mathbb{P}}{e^{i\frac{m_n}{n}t\function{\cos}{\xi\epsilon}} - 1 -\frac{m_n}{n}t\function{\cos}{\xi\epsilon}}\mathbb{P}(\ind{\zeta_0<Z\leq \zeta_0 + \frac{\eta}{m_n}}).\]
Then, $|\rho_n|\leq t^2 \frac{m_n}{n}m_n\mathbb{P}(\ind{\zeta_0<Z\leq \zeta_0 + \frac{\eta}{m_n}}) \rightarrow 0$. The characteristic function of $\\$ $m_n\mathbb{P}_n (\function{\cos}{\xi\epsilon} \ind{\zeta_0<Z\leq + \frac{\eta}{m_n}})$ can be written as
\[ \varphi_n (t) = \left(1+i\frac{m_n}{n}t\mathbb{P}(\function{\cos}{\xi\epsilon})\mathbb{P}( \ind{\zeta_0<Z\leq \zeta_0 + \frac{\eta}{m_n}}) + \frac{r_n}{n}\right)^n\rightarrow e^{it\eta\mathbb{P}(\function{\cos}{\xi\epsilon})f(\zeta_0)}\]
and therefore
\begin{equation}\nonumber
m_n\mathbb{P}_n (\function{\cos}{\xi\epsilon} \ind{\zeta_0<Z\leq \zeta_0 + \frac{\eta}{m_n}})\cip \eta f(\zeta_0)\mathbb{P}(\function{\cos}{\xi\epsilon}).
\end{equation}
Applying the same arguments to the function $\function{\sin}{\xi\epsilon}$ we obtain that
\begin{equation}\nonumber
m_n\mathbb{P}_n (\function{\sin}{\xi\epsilon} \ind{\zeta_0<Z\leq \zeta_0 + \frac{\eta}{m_n}})\cip \eta f(\zeta_0)\mathbb{P}(\function{\sin}{\xi\epsilon}).
\end{equation}
and hence
\begin{equation}\nonumber
m_n\mathbb{P}_n (e^{i\xi\epsilon} \ind{\zeta_0<Z\leq \zeta_0 + \frac{\eta}{m_n}})\cip \eta f(\zeta_0)\varphi{\xi}= \eta f(\zeta_0)\mathbb{P}(e^{i\xi\epsilon}).
\end{equation}
The same tightness argument that was applied to prove (\ref{ec65}) can be used here to conclude that
\begin{equation}\label{ec70}
\left\|m_n\mathbb{P}_n (e^{i\xi\epsilon} \ind{\zeta_0<Z\leq \zeta_0 + \frac{\cdot}{m_n}})-
(\cdot) f(\zeta_0)\mathbb{P}(e^{i\xi\epsilon})\right\|_{[0,\eta]}\cip 0
\end{equation}
and similarly
\begin{equation}\label{ec71}
\left\|m_n\mathbb{P}_n (e^{i\xi\epsilon} \ind{\zeta_0 - \frac{(\cdot)}{m_n}<Z\leq \zeta_0}) -(\cdot)f(\zeta_0)\mathbb{P}(e^{i\xi\epsilon})\right\|_{[0,\eta]}\cip 0.
\end{equation}

Using the triangular inequality together with the definition of $\tilde{\epsilon}_n$ we get
\[ m_n \left\| \mathbb{P}_n (e^{i\xi \tilde{\epsilon}_n}\ind{\hat{\zeta}_n < Z \leq \hat{\zeta}_n + \frac{(\cdot)}{m_n}})  - \mathbb{P}_n (e^{i\xi \epsilon}\ind{\zeta_0 < Z \leq \zeta_0 + \frac{(\cdot)}{m_n}})\right\|_{[0,\eta]} \leq\]
\[m_n \left\| \mathbb{P}_n\left(\ind{\hat{\zeta}_n < Z \leq \hat{\zeta}_n + \frac{(\cdot)}{m_n}} - \ind{\zeta_0 < Z \leq \zeta_0 + \frac{(\cdot)}{m_n}}\right)\right\|_{[0,\eta]} + \]
\[ m_n\left\|\mathbb{P}_n\left((e^{i\xi(Y-\hat{\beta}_n)} - e^{i\xi(Y-\beta_0)})\ind{\zeta_0 < Z \leq \zeta_0 + \frac{(\cdot)}{m_n}}\right)\right\|_{[0,\eta]}.\]
But (\ref{ec66})  implies that
\[m_n \left\| \mathbb{P}_n\left(\ind{\hat{\zeta}_n < Z \leq \hat{\zeta}_n + \frac{(\cdot)}{m_n}} - \ind{\zeta_0 < Z \leq \zeta_0 + \frac{(\cdot)}{m_n}}\right)\right\|_{[0,\eta]} \cip 0\]
while (\ref{ec70}) applied when $\xi = 0$ and the strong consistency of $\hat{\beta}_n$ yield
\[m_n\left\|\mathbb{P}_n\left((e^{i\xi(Y-\hat{\beta}_n)} - e^{i\xi(Y-\beta_0)})\ind{\zeta_0 < Z \leq \zeta_0 + \frac{(\cdot)}{m_n}}\right)\right\|_{[0,\eta]} \leq \]
\[|\hat{\beta}_n - \beta_0| m_n\left\|\mathbb{P}_n\left(\ind{\zeta_0 < Z \leq \zeta_0 + \frac{(\cdot)}{m_n}}\right)\right\|_{[0,\eta]}\cip 0.\]
Therefore,
\[m_n \left\| \mathbb{P}_n (e^{i\xi \tilde{\epsilon}_n}\ind{\hat{\zeta}_n < Z \leq \hat{\zeta}_n + \frac{(\cdot)}{m_n}})  - \mathbb{P}_n (e^{i\xi \epsilon}\ind{\zeta_0 < Z \leq \zeta_0 + \frac{(\cdot)}{m_n}})\right\|_{[0,\eta]}\cip 0\]
which together with (\ref{ec70}) proves that
\begin{equation}\nonumber
\left\|m_n\mathbb{P}_n (e^{i\xi\tilde{\epsilon}_n} \ind{\hat{\zeta}_n<Z\leq \hat{\zeta}_n + \frac{(\cdot)}{m_n}})-
(\cdot) f(\zeta_0)\mathbb{P}(e^{i\xi\epsilon})\right\|_{[0,\eta]}\cip 0.
\end{equation}
With completely analogous arguments one shows
\[m_n \left\| \mathbb{P}_n (e^{i\xi\tilde{\epsilon}_n}\ind{\hat{\zeta}_n - \frac{(\cdot)}{m_n} < Z \leq \hat{\zeta}_n})  - (\cdot) f(\zeta_0)\mathbb{P}(e^{i\xi\epsilon})\right\|_{[0,\eta]}\cip 0.\]
This proves that (VI) holds in probability.

We now proceed to prove that (V) and (VII) hold in probability. Before embarking in this task, we want to make the following remark. Consider that class of functions $\mathcal{C}:=\left\{\epsilon\ind{I}(z):I\subset\mathbb{R} \textrm{ is an interval}\right\}$. Then, this class has a square integrable envelope $|\epsilon|$ and $\mathbb{P}(\psi)=0$ for any $\psi\in\mathcal{C}$. Therefore, the maximal inequality 3.1 from \cite{kipo} implies that $\left\|\mathbb{P}_n\right\|_\mathcal{C} = O_\mathbf{P} \left(n^{-\frac{1}{2}}\right)$. Similar observations also show that $\left\|\mathbb{P}_n-\mathbb{P}\right\|_\mathcal{F} = O_\mathbf{P} \left(n^{-\frac{1}{2}}\right)$. All these considerations, in addition with Corollary \ref{corolario}, (\ref{ecuacionfinal1}), (\ref{ecuacionfinal2}), (\ref{ecuacionfinal5}) and (\ref{ecuacionfinal6}) show that
\begin{eqnarray}
\sqrt{m_n}(\hat{\alpha}_n - \alpha_0) &\cip& 0 \label{ecuacionfinal7}\\
\sqrt{m_n}(\hat{\beta}_n - \beta_0) &\cip& 0 \label{ecuacionfinal8}\\
m_n(\hat{\zeta}_n - \zeta_0) &\cip& 0 \label{ecuacionfinal9}\\
\sqrt{m_n}\left\|\mathbb{P}_n\right\|_\mathcal{C} &\cip& 0\label{ecuacionfinal10}\\
\sqrt{m_n}\left\| \mathbb{P}_n(|\epsilon|\ind{\zeta_0 - \frac{(\cdot)}{m_n}< Z \leq\zeta_0 + \frac{(\cdot)}{m_n}})\right\|_K &\cip& 0 \label{ecuacionfinal11}\\
\sqrt{m_n}\left\| \mathbb{P}_n(|\tilde{\epsilon}_n|\ind{\hat{\zeta}_n - \frac{(\cdot)}{m_n}< Z\leq\hat{\zeta}_n + \frac{(\cdot)}{m_n}})\right\|_K &\cip& 0 \label{ecuacionfinal12}\\
\sqrt{m_n}\left\|\mathbb{P}_n-\mathbb{P}\right\|_\mathcal{F} &\cip& 0\label{ecuacionfinal15}
\end{eqnarray}
for any compact set $K\subset\mathbb{R}$.

Let $\eta>0$ be fixed. Take any subsequence $(n_k)_{k=1}^\infty$ and find a further subsequence $(n_{k_s})_{s=1}^\infty$ such that all the statements in the previous display happen almost surely with the compact set $K$ taken to be $K=[\zeta_0 -2\eta,\zeta_0 +2\eta]$. Now, for such a subsequence, there is $N\in\mathbb{N}$ such that $m_{n_{k_{s}}}|\zeta_0 - \hat{\zeta}_{n_{k_{s}}}|<\eta$ $\forall\ s\geq N$. Then, for any $\delta>0$ and $s\geq N$, the following inequalities are true

\begin{eqnarray}
\sup_{|\hat{\zeta}_{n_{k_{s}}}-\zeta|<\delta^2}\left\{|\mathbb{P}_{n_{k_{s}}}(\tilde{\epsilon}_{n_{k_{s}}}\ind{\zeta\land\hat{\zeta}_{n_{k_{s}}} < Z \leq \zeta\lor\hat{\zeta}_{n_{k_{s}}}})|\right\}  \leq  |\hat{\alpha}_{n_{k_{s}}}-\alpha_0| + |\hat{\beta}_{n_{k_{s}}}-\beta_0| +\nonumber\\
\mathbb{P}_{n_{k_{s}}} (|\tilde{\epsilon}_{n_{k_{s}}}|\ind{\hat{\zeta}_{n_{k_{s}}} - \frac{\eta}{m_{n_{k_{s}}}} < Z \leq \hat{\zeta}_{n_{k_{s}}} + \frac{\eta}{m_{n_{k_{s}}}}}) +  \mathbb{P}_{n_{k_{s}}} (|\epsilon|\ind{\zeta_0- \frac{2\eta}{m_{n_{k_{s}}}} < Z \leq \zeta_0 + \frac{2\eta}{m_{n_{k_{s}}}}}) +  \left\|\mathbb{P}_{n_{k_{s}}}\right\|_\mathcal{C}\nonumber
\end{eqnarray}
\begin{eqnarray}
\sup_{|\hat{\zeta}_{n_{k_{s}}}-\zeta|<\delta^2}\left\{|\mathbb{P}_{n_{k_{s}}}(\tilde{\epsilon}_{n_{k_{s}}}\ind{Z \leq \zeta\land\hat{\zeta}_{n_{k_{s}}}})|+|\mathbb{P}_{n_{k_{s}}}(\tilde{\epsilon}_{n_{k_{s}}}\ind{Z > \zeta\lor\hat{\zeta}_{n_{k_{s}}}})|\right\}  \leq  |\hat{\alpha}_{n_{k_{s}}}-\alpha_0| + |\hat{\beta}_{n_{k_{s}}}-\beta_0| +\nonumber\\
\mathbb{P}_{n_{k_{s}}} (|\tilde{\epsilon}_{n_{k_{s}}}|\ind{\hat{\zeta}_{n_{k_{s}}} - \frac{\eta}{m_{n_{k_{s}}}} < Z \leq \hat{\zeta}_{n_{k_{s}}} + \frac{\eta}{m_{n_{k_{s}}}}})
+\mathbb{P}_{n_{k_{s}}} (|\epsilon|\ind{\zeta_0- \frac{2\eta}{m_{n_{k_{s}}}} < Z \leq \zeta_0 + \frac{2\eta}{m_{n_{k_{s}}}}})
+  \left\|\mathbb{P}_{n_{k_{s}}}\right\|_\mathcal{C}.\nonumber
\end{eqnarray}
These last inequalities together with (\ref{ecuacionfinal7})-(\ref{ecuacionfinal15}) imply that
\begin{eqnarray}
\lsup_{s\rightarrow \infty}\sqrt{m_{n_{k_s}}}\sup_{|\hat{\zeta}_{n_{k_{s}}}-\zeta|<\delta^2}\left\{|\mathbb{P}_{n_{k_{s}}}(\tilde{\epsilon}_{n_{k_{s}}}\ind{\zeta\land\hat{\zeta}_{n_{k_{s}}} < Z \leq \zeta\lor\hat{\zeta}_{n_{k_{s}}}})|\right\} & = & 0\ \textit{a.s.}\nonumber\\
\lsup_{s\rightarrow\infty}\sqrt{m_{n_{k_s}}}
\sup_{|\hat{\zeta}_{n_{k_{s}}}-\zeta|<\delta^2}\left\{|\mathbb{P}_{n_{k_{s}}}(\tilde{\epsilon}_{n_{k_{s}}}\ind{Z \leq \zeta\land\hat{\zeta}_{n_{k_{s}}}})|+|\mathbb{P}_{n_{k_{s}}}(\tilde{\epsilon}_{n_{k_{s}}}\ind{Z > \zeta\lor\hat{\zeta}_{n_{k_{s}}}})|\right\} & = & 0\ \textit{a.s.}\nonumber
\end{eqnarray}
The previous equations show that (\ref{rccs2}) and (\ref{rccs3}) in (V) as well as (VII) hold with probability one for the subsequence $(n_{k_s})_{s=1}^\infty$. We conclude by noting that if $\displaystyle \kappa = \inf_{z\in[a,b]}\left\{f(z)\right\}$, then the mean value theorem implies
\[ \inf_{\frac{1}{\sqrt{m_{n_{k_s}}}}\leq|\zeta-\hat{\zeta}_{n_{k_s}}|<\delta^2}\left\{ \frac{1}{|\zeta-\zeta_{n_{k_s}}|}\mathbb{P}_{n_{k_s}} (\ind{\zeta\land\hat{\zeta}_{n_{k_s}} < Z\leq\zeta\lor\hat{\zeta}_{n_{k_s}}})\right\}\geq \kappa - \sqrt{m_{n_{k_s}}}\left\|\mathbb{P}_{n_{k_s}} - \mathbb{P}\right\|_\mathcal{F}\]
which in consequence shows
\[ \linf_{s\rightarrow\infty} \inf_{\frac{1}{\sqrt{m_{n_{k_s}}}}\leq|\zeta-\hat{\zeta}_{n_{k_s}}|<\delta^2}\left\{ \frac{1}{|\zeta-\zeta_{n_{k_s}}|}\mathbb{P}_{n_{k_s}} (\ind{\zeta\land\hat{\zeta}_{n_{k_s}} < Z\leq\zeta\lor\hat{\zeta}_{n_{k_s}}})\right\}\geq \kappa > 0\ \ \ \textrm{a. s.}\]
This finishes the proof. $\hfill \square$

\bibliography{Referencias}
\bibliographystyle{apalike}
\end{document}